\documentclass[10pt,reqno]{amsart}
\usepackage[top=1in,bottom=1.2in,left=1.25in,right=1.25in]{geometry}
\usepackage{xr}
\usepackage{hyperref}
\usepackage{graphicx}
\usepackage[foot]{amsaddr}
\usepackage{amsmath}
\usepackage{amssymb}
\usepackage{amsthm}
\usepackage{xcolor}
\usepackage{float}
\usepackage{indentfirst}
\usepackage{placeins}
\usepackage{booktabs}
\usepackage{mathtools}
\usepackage{cleveref}
\usepackage{subcaption}
\usepackage{tikz}
\usepackage[backend=biber,
            style=numeric,
            giveninits=true,
            maxbibnames=99,
            minbibnames=99,
            sortcites=true]{biblatex}
\numberwithin{equation}{section}

\theoremstyle{plain}
\newtheorem{theorem}{Theorem}[section]

\newtheorem{lemma}[theorem]{Lemma}

\theoremstyle{definition}
\newtheorem{definition}{Definition}[section]
\newtheorem{assumption}{Assumption}[section]

\theoremstyle{remark}
\newtheorem{remark}{Remark}[section]
\newcommand{\R}{\mathbb{R}}
\newcommand{\C}{\mathbb{C}}
\newcommand{\cO}{\mathcal{O}}

\newcommand{\bx}{\boldsymbol x}
\newcommand{\by}{\boldsymbol y}
\newcommand{\bz}{\boldsymbol z}
\newcommand{\bc}{\boldsymbol c}
\newcommand{\bu}{\boldsymbol u}

\newcommand{\bp}{\boldsymbol p}
\newcommand{\bt}{\boldsymbol t}
\newcommand{\bgamma}{\boldsymbol \gamma}

\newcommand{\erfc}{\mathrm{erfc}}

\newcommand{\VpsiTwo}{{V_\psi^{\mathrm{(2d)}}}}
\newcommand{\VpsiThree}{V_\psi^{\mathrm{(3d)}}}

\crefname{figure}{Figure}{FIGURES}
\crefname{section}{Section}{SECTIONS}

\title[Fast Multipole Method with Complex Coordinates]{Fast Multipole Method with Complex Coordinates}

\author[T. Goodwill]{Tristan Goodwill\textsuperscript{1}}
\address{\textsuperscript{1}Department of Mathematics, University of Florida}
\email{\textsuperscript{1}tgoodwill@ufl.edu}

\author[L. Greengard]{Leslie Greengard\textsuperscript{2}}
\address{\textsuperscript{2}Courant Institute of Mathematical Sciences, New York University and Center for Computational Mathematics, Flatiron Institute}
\email{\textsuperscript{2}lgreengard@flatironinstitute.org}

\author[J. Hoskins]{Jeremy Hoskins\textsuperscript{3}}
\address{\textsuperscript{3}Department of Statistics and CCAM, University of Chicago and NSF-Simons National Institute for Theory and Mathematics in Biology}
\email{\textsuperscript{3}jeremyhoskins@uchicago.edu}

\author[M. Rachh]{Manas Rachh\textsuperscript{4}}
\address{\textsuperscript{4}Department of Mathematics, Indian Institute of Technology Bombay}
\email{\textsuperscript{4}mrachh@iitb.ac.in}

\author[Y. Wang]{Yuguan Wang\textsuperscript{5}}
\address{\textsuperscript{5}Program in Applied and Computational Mathematics, Princeton University}
\email{\textsuperscript{5}yuguanw@princeton.edu}

\begin{document}

\begin{abstract}
When solving problems of potential theory in unbounded layered media,
recently developed integral equation methods rely on coordinate complexification
to permit truncation of the computational domain while retaining high order accuracy.
As a result, classical fast multipole methods (FMMs) cannot be used to accelerate
the evaluation of layer potentials. In this paper, we present a variant of the FMM which
applies when the source and target locations have complex coordinates in two and three
dimensions. This \emph{complex coordinate FMM} is based on the analytic continuation of the
various special function  identities used in the real-space setting. To achieve the
same linear-time complexity as in the classical FMM, we construct
a hierarchical tree based solely on the real parts of the complex point locations and derive convergence rates for truncated expansions when the imaginary parts of the locations depend Lipschitz-continuously on the corresponding real parts. We demonstrate the efficiency of our approach through several numerical examples and illustrate its application to solving large-scale linearized water wave problems and Helmholtz transmission problems.
\end{abstract}


\maketitle

\section{Introduction}
Boundary integral equation (BIE) methods offer an effective framework for solving boundary value problems in a number of settings, including electrostatics, time-harmonic wave propagation, and Maxwell's equations. Broadly speaking, these methods seek to represent the solutions of partial differential equations (PDEs) in terms of unknown densities defined on domain boundaries, thereby reducing the dimensionality of the problem and often yielding well-conditioned second-kind Fredholm integral equations. On the other hand, the resulting discretized integral operators are dense, and hence applying them naively as part of an iterative algorithm results in a computational cost that is quadratic in the number of boundary discretization points. The classical fast multipole method (FMM)~\cite{GREENGARD1987325} addresses this issue by exploiting the low-rank structure of distant interactions. It organizes source and target locations into a hierarchical tree and approximates interactions between well-separated clusters via {\it multipole} and {\it local} expansions, so that the cost of applying integral operators scales linearly. Related approaches, such as the kernel-independent FMM (kiFMM)~\cite{malhotra2015pvfmm,YING2004591} and the dual-space multilevel kernel-splitting method (DMK)~\cite{https://doi.org/10.1002/cpa.22240}, which work on broader classes of kernels, have also been developed. When coupled to an iterative algorithm such as GMRES, this class of algorithms enables the solution of the discretized BIEs in a time that typically scales linearly with the number of discretization points.

In many applications of interest, particularly wave propagation in layered media, it is desirable to model portions of the boundary or interfaces as infinite. While traditional BIE methods extend in a formal sense, the analysis and numerical solution of these equations are challenging because of the slow decay of both the solutions and the boundary data frequently encountered in such problems. One method for addressing this issue is \emph{complex scaling} \cite{doi:10.1137/23M1607866,complexification} or \emph{integral-equation-based perfectly matched layers}~\cite{lu2018perfectly}, which involve `complexifying' the boundary coordinates and deforming the boundary into the complex plane. For outgoing fields and data satisfying the required analyticity assumptions, and for admissible deformations whose imaginary parts grow at least linearly, complex scaling transforms outgoing oscillations into exponential decay. Truncating the deformed contour then introduces an exponentially small error. {We refer to \cite{epstein2025complexscalingopenwaveguides,goodwill2024,hoskins2023quadraturesingularintegraloperators,doi:10.1137/24M1650703,WANG2026114956,doi:10.1137/22M1532457,doi:10.1137/21M1439705,10.1098/rsta.2024.0342,lu2026,pan2026multilayeredproblemsmultipletwolayered} for detailed analyses of this method and its numerical implementation. Recent work also shows that this technique can be applied to problems with periodic boundaries and obstacles~\cite{agocs2026integralequationsflexuralscattering,tan2026uniformlyhighaccuracypmlbiemethod}.} For more complicated geometries, particularly in three dimensions, although this approach converts infinite-domain problems into essentially finite ones, the discretized systems may still be large, necessitating fast solvers. Unfortunately, classical FMMs designed for real-valued point locations cannot be directly applied. {One might hope to use a hybrid strategy in which interactions among points in the undeformed real region are accelerated by a classical FMM, while the rest are evaluated directly. In practical complex-scaled discretizations, however, most point interactions can involve complex coordinates. This motivates an FMM that handles both real and complex point locations.}

In this paper, we extend the classical FMM to handle source and target locations with complex coordinates. Typically, when complex scaling is used, the points lie in $\C \times \R$ in two dimensions and in $\C^{2} \times \R$ in three dimensions, with certain monotonicity assumptions on the imaginary parts as functions of their corresponding real parts. While traditional kernel-independent methods can be applied to these point clouds after embedding them in $\mathbb{R}^{3}$ in the two-dimensional case and in $\mathbb{R}^{5}$ in the three-dimensional case, the associated constants of the linear-complexity scheme can be prohibitively high. The primary contributions of our work are threefold: constructing a hierarchical data structure based solely on the real parts of the coordinates, extending multipole and local expansions to the complexified domains, and proving truncation error estimates under Lipschitz assumptions on the imaginary parts as functions of the real parts. In particular, we show that the convergence rates for the truncated multipole and local expansions deteriorate with increasing Lipschitz constant $L$ as compared to the standard FMM and that the rate converges to the standard FMM convergence rate in the limit $L\to 0$. We demonstrate, through several numerical examples, the effectiveness of the complex-coordinate FMM for solving large scale boundary value problems.

\section{Background}
\label{sec:background}

\subsection{Layer potentials}
\label{sec:layer_potentials}
In the following, let $\Omega \subset \R^2$ be a region in the plane with smooth boundary $\Gamma=\partial \Omega$. The \emph{single layer potential} associated with the Helmholtz equation with wavenumber $\kappa>0$ is defined by the following formula
\begin{align}
\label{eqn:single_layer}
   S_{\Gamma,\kappa}[\sigma](\bx) = \int_{\Gamma} G_{\kappa} (\bx,\by) \sigma(\by) \, \mathrm{d}S(\by), \quad \bx \in \Omega,
\end{align}
where $G_\kappa(\cdot,\cdot)$ is the free-space Green's function of the 2-D Helmholtz equation
\begin{align}
\label{eqn:green_fun_helm_2d}
    G_\kappa(\bx,\by) = \frac{i}{4} H_0^{(1)}(\kappa \|\bx-\by\|), 
\end{align}
and $H_0^{(1)}(\cdot)$ is the \emph{Hankel function of the first kind}. The function $\sigma(\cdot)$ defined on $\Gamma$ is typically referred to as the \emph{density}, and the variables $\bx,\by$ in~\eqref{eqn:single_layer} are referred to as the \textit{target location} and \textit{source location}, respectively.   By construction, the single layer potential automatically satisfies the Helmholtz equation
\begin{align}
    (\Delta+\kappa^2)u=0, \quad \text{ in } \Omega. 
\end{align}

The \emph{double layer potential} with density function $\tau$ is defined by 
\begin{align}
 \label{eqn:double_layer}
D_{\Gamma,\kappa}[\tau](\bx)=\int_\Gamma  \left(  \frac{\partial}{\partial \nu(\by)}  G_{\kappa} (\bx,\by) \right)  \tau(\by)\, \mathrm{d}S(\by),\quad \bx \in \Omega,
\end{align}
where $\nu(\by)$ denotes the outward unit normal vector to $\Gamma$  at the point $\by$. The double layer potential also automatically satisfies the Helmholtz equation in $\Omega$. The single and double layer potentials associated with the 3-D Helmholtz equation, as well as the 2-D and 3-D Laplace equations, can be defined in a similar manner by replacing the kernel function with the corresponding free-space Green's function or its normal derivative with respect to  $\by$. For the Laplace layer potentials, we simplify the notation and denote the layer potentials  by $S_\Gamma[\sigma]$ and $D_\Gamma[\tau]$.

When the target location $\bx$ lies on the surface $\Gamma$, we denote the principal-value and finite-part limits of the normal derivatives of the single- and double layer potentials at $\bx$ by $S'_{\Gamma,\kappa}[\sigma]$ and $D'_{\Gamma,\kappa}[\tau]$, respectively.

The cost of naively evaluating these layer potentials or their normal derivatives at $N$ target locations, with $\Gamma$ discretized using $M$ points, is $\cO(M\cdot N)$. The classical FMM and its variations~\cite{10.7551/mitpress/5750.001.0001,wideband3d,CHENG1999468,Greengard_Rokhlin_1997,malhotra2015pvfmm,YING2004591,714591} reduce this cost to $\cO(M+N)$ for a given precision.

In this work, we are interested in evaluating layer potentials where portions of the boundary $\tilde \Gamma$ have complex coordinates. In this case, the Green's function~\eqref{eqn:green_fun_helm_2d} in the definitions of $S_{\tilde \Gamma,\kappa}$ and $D_{\tilde \Gamma,\kappa}$ is replaced by $G_\kappa(\bx,\by)=\frac{i}{4}H_0^{(1)}(\kappa r_{\bx,\by})$, where $\bx,\by\in\C^2$, $\bx\ne\by$, and $r_{\bx,\by}=\sqrt{(x_1-y_1)^2+(x_2-y_2)^2}$ is the \emph{complexified distance function}. In this paper, for each multivalued function, we take its standard value on the real geometry and analytically continue it along the prescribed deformation, which is assumed not to cross any singularity or branch point. In particular, $r_{\bx,\by}$ is obtained by continuing the positive real square root and, for the admissible deformations considered below, satisfies $\Im r_{\bx,\by}\ge 0$.

\subsection{Complex scaling and integral equations}
\label{sec:complexification_example}
As a moltivating example, we briefly sketch the integral-equation-based \emph{complex scaling} method for the sound-soft acoustic scattering problem on a perturbed half-space \cite{complexification}. We consider the 2-D case with a boundary $\Gamma$ parameterized by a curve $\bgamma:\R\to \R^2$ given by $\bgamma(t)=(t,\gamma_2(t))$, where $\gamma_2(t)=0$ if $|t|>B$ for some $B>0$. We are interested in solving the boundary value problem
\begin{align*}
    \begin{cases}
        (\Delta+\kappa^2)u(\bx)=0, & \bx \in \Omega, \\ 
        u(\bx)=f(\bx), & \bx \in \Gamma,\\
        \lim\limits_{\|\bx\| \to \infty} \sqrt{\|\bx\|}\left(\frac{\partial u(\bx)}{\partial \|\bx\|}-i\kappa u(\bx)\right)=0,
    \end{cases}
\end{align*}
for some known data $f$ defined on $\Gamma$.

A naive integral equation approach represents the solution in the upper half-plane using a double layer potential $u=D_{\Gamma,\kappa}[\sigma]$. Since it automatically satisfies the Helmholtz equation, one only needs to ensure that the boundary conditions are satisfied. Inserting this ansatz into the boundary condition and using the standard jump relation of the double layer potential, we reduce the problem to solving the following BIE
\begin{align}\label{eq:urep_helm_dir}
    -\frac{1}{2}\sigma(\bx)+D_{\Gamma,\kappa}[\sigma](\bx)=f(\bx), \quad \bx \in \Gamma.
\end{align}
For more details on the derivation of the jump relation and the use of the BIE method, we refer the reader to~\cite{kress_2014,colton_kress_inverse,colton_kress}. Unfortunately, for many problems of physical interest, both the data $f$ and the unknown density $\sigma$ decay slowly, typically like $1/\sqrt{|x_1|}$ or~$1/|x_1|^{3/2}$. Numerically, this causes significant errors when the computational domain is naively truncated. In three dimensions, the surface $\Gamma$ is modeled by a perturbation of the $(x_1,x_2)$-plane. In this case, one can show similarly that, for a certain class of `outgoing' data, the density typically decays like $1/\|\bx\|$ as $\|\bx\| \to \infty$.

In the current example, complex scaling addresses this issue by deforming the boundary $\Gamma \subset \mathbb{R}^2$ into a complex curve $\tilde \Gamma \subset \mathbb{C}\times \mathbb{R}$ parameterized by
\begin{align*}
    \tilde\bgamma(t)=(t+i\psi(t),\gamma_2(t)),\qquad
    \tilde\Gamma=\{\tilde\bgamma(t):t\in\R\},
\end{align*}
where $\psi:\R\to \R$ is a smooth monotone function that vanishes on $[-B-\delta,B+\delta]$ for some $\delta>0$ and satisfies $\psi(t)\to\pm\infty$ as $t\to\pm\infty$. The solution is then represented as a double layer potential over the deformed boundary, written as $u=D_{\tilde \Gamma,\kappa}[\tilde \sigma]$. The corresponding BIE is
\begin{align}
\label{eqn:complex_BIE}
    -\frac{1}{2}\tilde \sigma(\bx)+D_{\tilde \Gamma,\kappa}[\tilde \sigma](\bx)=f(\bx), \quad \bx \in \tilde \Gamma,
\end{align}
where the function $f$ is assumed to admit an analytic continuation to $\tilde \Gamma$. On $\tilde \Gamma$, both the data and the density decay exponentially:
\begin{align}
\label{eqn:fast_decay_bound}
    |\tilde \sigma(\bx)|,|f(\bx)|
    \lesssim \frac{e^{-\kappa |\psi(x_1)|}}{\sqrt{|x_1|}},
    \qquad x_1\to\infty.
\end{align}
The exponential decay in \eqref{eqn:fast_decay_bound} enables efficient truncation of the computational domain, with radius proportional to the logarithm of the desired accuracy.

To achieve high accuracy, the discretized BIE system  typically contains a large number of variables, making direct solvers computationally expensive.  Though a compressed representation of the inverse discretized operator can be constructed using a modification of the methods in \cite{doi:10.1137/16M1095949,sushnikova2023fmm}, this method currently lacks theoretical guarantees and in practice incurs significant memory costs, limiting the size of the discretization.  On the other hand, iterative solvers such as GMRES~\cite{Saad-gmres-1986}  are more memory efficient and easier to implement, but naively require the application of the full $N\times N$ dense system matrix.  The complex-coordinate FMM described in this paper can accelerate this computation, reducing the cost of applying the matrix from $\cO(N^2)$ to $\cO(N)$.

\subsection{The classical FMM}
\label{sec:FMMs_intro}
After the integral equation~\eqref{eq:urep_helm_dir} is discretized, numerical evaluation of the double layer operator requires computing sums of the form
\begin{equation}\label{eqn:double_disc}
    u(\bx_i)=\sum_{j=1}^M \partial_{\nu(\by_j)} G_\kappa(\bx_i,\by_j) \sigma(\by_j) w_j,
\end{equation}
where $\bx_1,\ldots,\bx_N$ are {\it target locations} at which $u$ is to be evaluated, $\by_1,\ldots,\by_M$ are {\it source locations} or {\it quadrature nodes}, and $w_1,\ldots,w_M$ are the associated {\it quadrature weights}. We note that, in practice, one frequently adds quadrature corrections \cite{doi:10.1137/S0036142995287847,WuMartinsson2021,doi:10.1137/22M1520372,WuMartinsson2021b,hoskins2023quadraturesingularintegraloperators} to the sum in~\eqref{eqn:double_disc} to ensure high-order accuracy of the integrals when the targets lie on or near $\Gamma$. Since these quadrature corrections are typically sparse, their application requires $\cO(N)$ operations and is not the dominant computational cost. Sums of the form~\eqref{eqn:double_disc} can be computed efficiently using the FMM. To make the presentation as self-contained as possible, we give a brief nontechnical introduction to the classical FMM and its associated terminology.  We refer the reader to \cite{b9199369fcb841c19f902432847cfa34} for  a more comprehensive overview.

For simplicity, we replace~\eqref{eqn:double_disc} by the $M$-body calculation
\begin{equation}
\label{eqn:evaluation_problem}
    u(\bx_i)=\sum_{j=1}^M G_\kappa(\bx_i,\by_j) \sigma_j,
\end{equation}
with densities $\{\sigma_j\}_{j=1}^M \subset \C$. Throughout, $\vec{\sigma}$ denotes the corresponding density vector, with $\|\vec{\sigma}\|_1=\sum_j|\sigma_j|$. The sum \eqref{eqn:double_disc} can be computed by differentiating the appropriate expansions below.

The FMM is based on two observations. First, any solution $u$ of the Helmholtz or Laplace equations in the ball~$\|\bx-\bc\|<R$ admits a \textit{local expansion}
\begin{equation}\label{eqn:local_exp}
    u(\bx) = \sum_{n=0}^\infty L_n f_n(\bx-\bc),
\end{equation}
for a known collection of functions $f_n$ depending only on the PDE. In 2-D, the expansion is such that
\begin{equation}
    u(\bx) = \sum_{n=0}^P L_n f_n(\bx-\bc) + \cO\left( \left(\frac rR\right)^P\right)
\end{equation}
in the ball~$\|\bx- \bc\|<r<R$. In 3-D, the truncation error has the same geometric dependence on the maximum expansion degree $P$, up to dimension-dependent prefactors. In practice, for spherically symmetric kernels each function $f_n$ will be the product of a radial function and an angular basis function. The second observation is that a field generated by sources contained in the ball $\|\by-\bc\|<R$ admits, for $\|\bx-\bc\|>R$, a \textit{multipole expansion}
\begin{equation}\label{eqn:multipole_exp}
    u(\bx) = \sum_{n=0}^\infty M_n g_n(\bx-\bc),
\end{equation}
where again the functions $g_n$ are a known collection of functions depending only on the PDE. Such an expansion applies to exterior solutions satisfying the appropriate outgoing condition.

Given these observations, the FMM proceeds as follows. First, split the sources and targets into a collection of boxes. Away from a given box, the contribution from the sources inside the box to \eqref{eqn:evaluation_problem} can be expressed as a truncated multipole expansion. The contribution of different boxes can then be merged efficiently by analytically \textit{translating} the expansions to a common center and adding their coefficients. We call this translation the \emph{multipole-to-multipole (M2M)} operator. We then build a local expansion for the field generated by sources in this merged box using an equivalent \emph{multipole-to-local (M2L)} operator. Local expansions for children of a box can be built from their parent expansion using a \emph{local-to-local (L2L)} operator. Finally, $u$ is given by the sum of the local expansion and the contribution of nearby sources.

By efficiently managing which boxes to merge and which expansions to translate, the FMM computes the sum \eqref{eqn:evaluation_problem} in $\cO(M+N)$ time. To compute sums of the form~\eqref{eqn:double_disc}, we simply differentiate the multipole expansions. We can also compute derivatives of~$u$ by differentiating the local expansions.

\section{Analytic preliminaries}
\label{sec:analytic-preliminaries}

\subsection{Notation for complex variables}
\label{sec:notations}

The Green's functions, local expansions, and multipole expansions are all expressed in terms of polar or spherical coordinates. In order to proceed, we define the analogous system for complex vectors.

\begin{definition}[Complex polar coordinates]
    Let $\bx=(x_1,x_2)\in \C^2$ with $x_1^2+x_2^2\ne0$. We define the complex polar coordinates $(r_x, \varphi_x) \in \C^2$ by
    \begin{equation*}
    \begin{cases}
         r_x  =\sqrt{x_1^2+x_2^2},  \\
         \cos \varphi_x  = \frac{x_1}{r_x}, \quad \sin\varphi_x = \frac{x_2}{r_x}. 
    \end{cases}
    \end{equation*}
    For $r_x\neq0$ we have $x_1 = r_x \cos \varphi_x$ and $x_2 = r_x \sin \varphi_x$.
\end{definition}
It is clear from the definition that the complex polar coordinates are an analytic continuation of the real ones. Moreover, we also note that with this definition, $e^{i\varphi_x} = \frac{x_1+ix_2}{r_x}$. For convenience, we will often write $\bx = (r_x, \varphi_x) \in \C^2$ in a slight abuse of notation.

\begin{remark}
    It is important to note that for~$\bx\in \C^2$,~$r_x$ may be small even when~$\sqrt{|x_1|^2+|x_2|^2}$ is not. Indeed, if $\bx=(1,i)$, then~$r_x=0$. Below, we impose assumptions that control the imaginary part of~$\bx$ in terms of its real part and, in particular, exclude this degeneracy.
\end{remark}

For $\bx\in\C^3$, we use the following analytic continuation of the standard spherical coordinates.
\begin{definition}[Complex spherical coordinates]
Let $\bx=(x_1,x_2,x_3)\in \C^3$ satisfy $x_1^2+x_2^2+x_3^2\ne0$, and assume either $x_1^2+x_2^2\ne0$ or $x_1=x_2=0$. We first define
$$
    \rho_x=\sqrt{x_1^2+x_2^2+x_3^2}.
$$
If $x_1^2+x_2^2\ne0$, we define $\theta_x$ and $\phi_x$ by
    \begin{equation*}
    \begin{cases}
        \cos \theta_x  = \frac{x_3}{\rho_x}, \quad
        \sin \theta_x = \frac{\sqrt{x_1^2+x_2^2}}{\rho_x},\\ 
        \cos \phi_x = \frac{x_1}{\sqrt{x_1^2+x_2^2}}, \quad  \sin \phi_x = \frac{x_2}{\sqrt{x_1^2+x_2^2}}.
    \end{cases}
    \end{equation*}
If $x_1=x_2=0$, we set $\phi_x=0$ and take $\theta_x=0$ or $\pi$ according as $x_3/\rho_x=1$ or $-1$. In both cases, $x_1 = \rho_x \sin \theta_x \cos \phi_x$, $x_2 = \rho_x \sin \theta_x \sin \phi_x$, and $x_3 = \rho_x \cos \theta_x$.
\end{definition}
For $\bx=0$, we define only $\rho_x=0$ and leave the angular coordinates undefined. When $x_1^2+x_2^2\ne0$, we again have $e^{i\phi_x}=\frac{x_1+ix_2}{\sqrt{x_1^2+x_2^2}}$; on the nonzero $x_3$-axis our convention gives $e^{i\phi_x}=1$. We write $\bx = (\rho_x,\theta_x,\phi_x) \in \C^3$ whenever the spherical coordinates are defined.

\subsection{Properties of classical special functions}
\label{sec:special_functions}
\begin{theorem}[Graf's Addition Formula~\cite{abramowitz+stegun}]
\label{thm:graf2d}
Let $z=\sqrt{u^2+w^2-2uw\cos\alpha}$ with $u,w,\alpha\in\C$. Suppose $\beta\in\C$ satisfies $u-w\cos\alpha=z\cos\beta$ and $w\sin\alpha=z\sin\beta$. For every $m\in\mathbb Z$,
\begin{align}
J_m(z) e^{i m \beta} &= \sum_{n=-\infty}^\infty J_{m+n}(u) J_n(w) e^{i n \alpha}, \label{eqn:addition_for_J} \\
H_m^{(1)}(z) e^{i m \beta} &= \sum_{n=-\infty}^\infty H_{m+n}^{(1)}(u) J_n(w) e^{i n \alpha}, \quad \text{if } |w e^{\pm i \alpha}| < |u|. \label{eqn:addition_for_H}
\end{align}
For $m=0$, when $P$ is sufficiently large, the truncation error satisfies
\begin{equation}
\label{eqn:hankel_error}
    \left \vert H^{(1)}_0(z)-\sum_{n=-P}^P H^{(1)}_{n}(u) J_n(w) e^{in \alpha} \right\vert  \lesssim  \frac{2}{ P \pi (c-1)} \left(\frac{1}{c}\right)^{P},
\end{equation}
where $c=\max\{|w e^{i \alpha}|/|u|,|w e^{-i \alpha}|/|u|\}^{-1}>1$.
\end{theorem}

\begin{proof}
Equations~\eqref{eqn:addition_for_J} and~\eqref{eqn:addition_for_H} are classical results from Graf's addition theorem~\cite{watson1995treatise}. The error estimate~\eqref{eqn:hankel_error} follows from the large-order asymptotic expressions (cf. 10.19.1 and 10.19.2 in \cite{NIST:DLMF})
\begin{equation}
\label{eqn:bessel_asymp}
    J_n(z)\sim \frac{1}{\sqrt{2\pi n}}
    \left(\frac{ez}{2n}\right)^n,\qquad
    H_n^{(1)}(z)\sim -i\sqrt{\frac{2}{\pi n}}
    \left(\frac{2n}{ez}\right)^n,\qquad n\to\infty,
\end{equation}
 the relations $ J_{-n}(z)= (-1)^n J_n(z),\quad H_{-n}^{(1)}(z)= (-1)^n H_n^{(1)}(z), \quad n \in \mathbb{N}$, and the triangle inequality.
\end{proof}

For the 3-D Helmholtz problem, we require a corresponding addition formula for the spherical Bessel and spherical Hankel functions, defined for integers $n\ge0$ by
\begin{equation*}
    j_n(z) = \sqrt{\frac{\pi}{2z}} J_{n+\frac{1}{2}}(z) \quad\text{and}\quad
     h_n^{(1)}(z) = \sqrt{\frac{\pi}{2z}} H^{(1)}_{n+\frac{1}{2}}(z),
\end{equation*}
respectively. For the spherical Hankel function of order zero, the addition formula is given by the following theorem.
\begin{theorem}
Let $z=\sqrt{u^2+w^2-2uw\cos\alpha}$, where $u,w,\alpha\in\C$, $u\ne0$, and $|w e^{\pm i\alpha}|<|u|$. Then
\begin{equation}
\label{eqn:addition_sph_hankel}
    h^{(1)}_0(z) = \sum_{n=0}^\infty (2n+1) h_n^{(1)}(u) j_n(w) P_n(\cos \alpha),
\end{equation}
where $P_n$ is the Legendre polynomial of degree $n$. For fixed $u\ne0$ and $w$, and all sufficiently large $P$, the truncation error satisfies
\begin{equation*}
    \left \vert  h^{(1)}_0(z)-  \sum_{n=0}^P (2n+1) h_n^{(1)}(u) j_n(w) P_n(\cos \alpha)\right \vert \lesssim \frac{1}{|u|(c-1)} \left(\frac{1}{c}\right)^P,
\end{equation*}
where $c=\max\{|w e^{i\alpha}|/|u|,|w e^{-i\alpha}|/|u|\}^{-1}>1$.
\end{theorem}
\begin{proof}
Formula~\eqref{eqn:addition_sph_hankel} is a special case of Gegenbauer's addition formula (cf. 9.1.80 in \cite{abramowitz+stegun}). We first establish the following bound on Legendre polynomials:
\begin{equation}
\label{eqn:lege_bound}
    |P_n(\cos\alpha)| \le \max\{|e^{in\alpha}|,|e^{-in\alpha}|\}, \quad \text{ for } \alpha \in \C.
\end{equation} 
To see this, we use the integral representation $P_n(\cos\alpha) = \frac{1}{\pi} \int_0^\pi (\cos\alpha+i\sin\alpha \cos\phi)^n\, \mathrm{d}\phi$, which corresponds to equation 14.12.8 in \cite{NIST:DLMF} for $m=0$. Taking the modulus and applying the triangle inequality yields $|P_n(\cos\alpha)| \le \frac{1}{\pi} \int_0^\pi |\cos\alpha+i\sin\alpha \cos\phi|^n \, \mathrm{d}\phi$. Using the identities $\cos\alpha=\frac{1}{2}e^{i\alpha}+\frac{1}{2}e^{-i\alpha}$ and $\sin\alpha=\frac{1}{2i}e^{i\alpha}-\frac{1}{2i}e^{-i\alpha}$, we obtain
\begin{align*}
     \left\vert \cos\alpha+i\sin\alpha\cos\phi \right\vert &= \left\vert \frac{1+\cos\phi}{2} e^{i\alpha} + \frac{1-\cos\phi}{2}e^{-i\alpha}\right\vert \le \frac{1+\cos\phi}{2} \left\vert e^{i\alpha}\right\vert + \frac{1-\cos\phi}{2}\left\vert  e^{-i\alpha}\right\vert\\ 
     &\le \max\{\left\vert  e^{i\alpha}\right\vert,\left\vert  e^{-i\alpha}\right\vert\},
 \end{align*}
from which the required bound on $P_n$ follows. For fixed $u\ne0$ and $w$, the standard large-order asymptotic formulas give
$$
    j_n(w)\sim
    \frac{\sqrt{\pi}\,w^n}{2^{n+1}\Gamma(n+3/2)},
    \qquad
    h_n^{(1)}(u)\sim
    -i\frac{2^n\Gamma(n+1/2)}{\sqrt{\pi}\,u^{n+1}}.
$$
Consequently,
$$
    (2n+1)|h_n^{(1)}(u)j_n(w)P_n(\cos\alpha)|
    \lesssim \frac{1}{|u|}
    \left(
        \max\left\{
        \frac{|w e^{i\alpha}|}{|u|},
        \frac{|w e^{-i\alpha}|}{|u|}
        \right\}
    \right)^n
    =\frac{1}{|u|c^n},
$$
for all sufficiently large $n$. Summing the geometric tail gives the stated estimate.
\end{proof}

\subsubsection{Spherical harmonics}
The \textit{associated Legendre polynomials} of degree $n\ge 0$ and order $0\le m\le n$ are defined as
\begin{equation*}
P_n^m(z) = (-1)^m(1-z^2)^{m/2} \frac{d^m}{dz^m} P_n(z),
\end{equation*}
for $z\in\C$. It is often convenient to use the following normalized functions for signed orders $-n\le m\le n$:
\begin{equation*}
    \bar P^m_n(z) =(-1)^m \sqrt{\frac{2n+1}{4\pi}
    \frac{(n-|m|)!}{(n+|m|)!}}\,P_n^{|m|}(z),
\end{equation*}
which satisfy the orthogonality relation
\begin{equation}
\label{eqn:lenP_ortho}
      \int_{-1}^1 \bar P^m_n(z)\bar P^m_{n'}(z)\,\mathrm{d}z
      =\frac{\delta_{n,n'}}{2\pi}. 
\end{equation}

With the above definitions, the standard spherical harmonics for real variables admit the following natural extension to complex arguments.
\begin{definition}[Spherical harmonics]
Let $n \geq 0$ and $-n \leq m \leq n$. The spherical harmonic $Y_n^m(\theta, \phi)$ is defined as
\begin{equation*}
    Y_n^m(\theta,\phi) =(-1)^m \sqrt{ \frac{(n-|m|)!}{(n+|m|)!}} P_n^{|m|}(\cos\theta) e^{im\phi},
\end{equation*}
where $\theta,\phi\in\C$.
\end{definition}
Thus, we also have
\begin{equation}
\label{eqn:Y_barP_relation}
    Y_n^m(\theta,\phi)
    =\sqrt{\frac{4\pi}{2n+1}}\,
    \bar P_n^m(\cos\theta)e^{im\phi}.
\end{equation}

For the next results it will be convenient to define the following differential operators:
\begin{equation*}
    \partial_\pm = \frac{\partial}{\partial x_1}\pm i \frac{\partial}{\partial x_2}, \quad \partial_{x_3} = \frac{\partial }{\partial x_3}.
\end{equation*}
The next lemma follows directly from these definitions.
\begin{lemma}
\label{lma:harmonic}
If $u$ is a harmonic function of three variables, then
\begin{equation*}
    \partial_+ \partial_-(u) = -\partial_{x_3}^2(u).
\end{equation*}
\end{lemma}

In the construction of the FMM, it is convenient to express the spherical harmonics as derivatives of~$\rho^{-1}$. The formulas below are given for real points in \cite{10.7551/mitpress/5750.001.0001} and extend to complex points by analyticity.
\begin{lemma}
\label{lma:spharmonics}
Let $\bx=(\rho,\theta,\phi)$ with $\rho\ne0$, let $n\ge0$, and let $-n\le m\le n$. Then
\begin{equation*}
    \frac{Y_n^m(\theta,\phi)}{\rho^{n+1}} = \begin{cases}
        A_n^m\cdot  \partial_+^m \partial_{x_3}^{n-m}\left(\frac{1}{\rho}\right) , & m>0,\\
        A_n^0 \cdot \frac{\partial^n}{\partial x_3^n}\left(\frac{1}{\rho}\right)  , &m=0,\\
        A_n^{-m}\cdot  \partial_-^{-m} \partial_{x_3}^{n+m}\left(\frac{1}{\rho}\right), & m<0.
    \end{cases}
\end{equation*}
Here, $A_n^m=(-1)^n/\sqrt{(n-m)!(n+m)!}$.
\end{lemma}

Let $(\theta_1,\phi_1)$ and $(\theta_2,\phi_2)$ be the spherical angles of two vectors in $\C^3$. Then the addition formula
\begin{equation}
\label{eqn:sphaddition}
    P_n(\cos\alpha) =\sum_{m=-n}^n Y_n^{-m}(\theta_1,\phi_1) Y_n^m(\theta_2,\phi_2),
\end{equation}  
holds for $\cos\alpha=\cos\theta_1\cos\theta_2+\sin\theta_1\sin\theta_2\cos(\phi_1-\phi_2)$ \cite{whittaker1996course}.

\section{The 2-D complex-coordinate FMM}
\label{sec:zfmm2d}
In this section, we describe the two-dimensional complex-coordinate FMM for the Laplace and Helmholtz Green's functions. {Although these operators and also the ones in~\cref{sec:zfmm3d}  have the same algebraic structure of the classical FMM due to analytic continuation, their extension to complex coordinates requires new convergence conditions and error estimates.}

\subsection{Geometric assumptions}
\label{sec:geo_conditon_2d}

\begin{assumption}[2-D Assumption]
\label{assump:2d}
{
We assume that the source locations, target locations, and expansion centers are vectors in $\C^2$ satisfying
$
    \Im\bx=\psi(\Re\bx),
$
where $\psi:\R^2\to\R^2$ has the form
$
    \psi(x_1,x_2)=\bigl(\psi_1(x_1),\psi_2(x_2)\bigr).
$
For some $L_\psi<1$, each smooth function $\psi_j:\R\to\R$, $j=1,2$, satisfies
\begin{equation}
  \label{eqn:Lip2d}
    0\le \psi_j(t)-\psi_j(s)\le L_\psi(t-s),
    \qquad s\le t.
\end{equation}
}
We let $\VpsiTwo\subset \C^2$ denote the set of points satisfying the constraint $\Im \bx = \psi(\Re \bx), \quad \bx \in \C^2$.
\end{assumption}

{
We first provide a pointwise estimate in coordinates with expansion center $\bc=\Re\bc+i\psi(\Re\bc)$. In these coordinates, the complexification is
\begin{equation}
\label{eqn:recentered_psi}
    \widetilde\psi(\bu)
    =\psi(\bu+\Re\bc)-\psi(\Re\bc).
\end{equation}
It has the same componentwise monotonicity and Lipschitz constant as $\psi$ and satisfies $\widetilde\psi(0)=0$. With a slight abuse of notation, we still write $\psi$ for this recentered map  when necessary.

For $0\le L_\psi<1$ and $0\le r<R$, we define
\begin{equation}
\label{eqn:distortion_constant}
  s_{\pm}(L_\psi)=\frac{\sqrt{L_\psi^2+4}\pm L_\psi}{2}, \qquad
  \chi(L_\psi)=\frac{s_+(L_\psi)}{s_-(L_\psi)}
  =\frac{2+L_\psi^2+L_\psi\sqrt{L_\psi^2+4}}{2},
\end{equation}
and
\begin{equation}
\label{eqn:convergence_factor}
    C_{L_\psi,r,R}:=\chi(L_\psi)\frac{r}{R}.
\end{equation}
Here, $r/R$ is the classical real-coordinate convergence factor for multipole and local expansions. Later on, we will see that $\chi(L_\psi)$ measures the worst-case distortion caused by complexification. The function $\chi$ is increasing, with $\chi(0)=1$ and $\chi(1)=\frac{3+\sqrt{5}}{2}$. We will show that $C_{L_\psi,r,R}$ is the convergence factor for the complex-coordinate expansions. In addition, we have $s_+(L_\psi)s_-(L_\psi)=1$, $\chi(L_\psi)=s_+(L_\psi)^2$, and the following property holds:
\begin{equation}
  \label{eqn:convergence_factor_equivalence}
  C_{L_\psi,r,R}<1
  \quad\Longleftrightarrow\quad
  L_\psi<\frac{R-r}{\sqrt{Rr}},
  \qquad 0<r<R.
\end{equation}
}

{
\begin{lemma}
\label{lma:pointwise_distortion_2d}
Let $\psi:\R^2\to\R^2$ be the complexification scheme satisfying Assumption~\ref{assump:2d}. If $\psi(0)=0$, then for $\bz\in\VpsiTwo$, we have
\begin{equation}
\label{eq:Apm_uniform_bounds}
    s_-(L_\psi)\|\Re\bz\|
    \le |z_1\pm iz_2|
    \le s_+(L_\psi)\|\Re\bz\|.
\end{equation}
\end{lemma}
}

\begin{proof}
{
Let $\bz=\bp+i\psi(\bp)$, where $\bp\in\R^2$. For $j=1,2$, we define
\begin{equation*}
    \gamma_j=
    \begin{cases}
        \psi_j(p_j)/p_j, & p_j\ne0,\\
        0, & p_j=0.
    \end{cases}
\end{equation*}
Since $\psi(0)=0$, we have $0\le\gamma_j\le L_\psi$ and $\psi_j(p_j)=\gamma_jp_j$. Therefore,
$z_1=(1+i\gamma_1)p_1, z_2=(1+i\gamma_2)p_2.$
Equivalently, these identities can be written as
\begin{align}
\label{eq:Apm_identity}
  \begin{pmatrix}
    \Re(z_1+iz_2)
    \\ 
    \Im(z_1+iz_2)
  \end{pmatrix}
  &=A_+(\gamma_1,\gamma_2)\bp,
  & A_+(\gamma_1,\gamma_2)
  &=\begin{pmatrix}
  1 & -\gamma_2\\ 
  \gamma_1 & 1
  \end{pmatrix},\\
  \begin{pmatrix}
    \Re(z_1-iz_2)\\ \Im(z_1-iz_2)
  \end{pmatrix}
  &=A_-(\gamma_1,\gamma_2)\bp,
  &A_-(\gamma_1,\gamma_2)
  &=\begin{pmatrix}
  1 & \gamma_2\\ 
  \gamma_1 & -1
  \end{pmatrix}.
\end{align}
The two matrices have the same singular values, which are
\begin{equation}
\label{eqn:singular_vals_A}
  \sigma_{\max}(A_\pm)
  =\frac{\sqrt{4+(\gamma_1+\gamma_2)^2}
  +|\gamma_1-\gamma_2|}{2},\qquad
  \sigma_{\min}(A_\pm)
  =\frac{\sqrt{4+(\gamma_1+\gamma_2)^2}
  -|\gamma_1-\gamma_2|}{2}.
\end{equation}
Without loss of generality, we let $\gamma_1\ge\gamma_2$. For fixed $\gamma_1$, the first expression in \eqref{eqn:singular_vals_A} decreases with $\gamma_2$, since
\begin{equation*}
  \frac{\partial}{\partial\gamma_2}
  \left(\sqrt{4+(\gamma_1+\gamma_2)^2}
  +\gamma_1-\gamma_2\right)
  =\frac{\gamma_1+\gamma_2}
  {\sqrt{4+(\gamma_1+\gamma_2)^2}}-1<0.
\end{equation*}
Therefore, for fixed $\gamma_1$, it achieves its maximum at $\gamma_2=0$. At this endpoint,
\begin{equation*}
  \sigma_{\max}(A_\pm(\gamma_1,0))
  =\frac{\sqrt{4+\gamma_1^2}+\gamma_1}{2}=s_+(\gamma_1),
\end{equation*}
which is increasing in $\gamma_1$. Similarly,
\begin{equation*}
  \sigma_{\min}(A_\pm(\gamma_1,0))
  =\frac{\sqrt{4+\gamma_1^2}-\gamma_1}{2}=s_-(\gamma_1),
\end{equation*}
which is decreasing in $\gamma_1$. Hence the maximum and minimum are $s_+(L_\psi)$ and $s_-(L_\psi)$, respectively, and both occur at $(\gamma_1,\gamma_2)=(L_\psi,0)$. Applying these two singular-value bounds in~\eqref{eq:Apm_identity} proves~\eqref{eq:Apm_uniform_bounds}.
}
\end{proof}

Let  $\bx=(r_x,\varphi_x)$ and $\by=(r_y,\varphi_y)\in \VpsiTwo$ be the target and source locations, respectively, expressed in polar coordinates. We denote their difference by $\bx-\by = (r_{xy},\varphi_{xy})$. It is not hard to verify that
\begin{equation*}
    r_{xy} = \sqrt{r_x^2+r_y^2-2r_xr_y\cos(\varphi_y-\varphi_x)}. 
\end{equation*}
In what follows, we will use Graf's addition formula \eqref{eqn:addition_for_H} with $z=r_{xy},u=r_x,w=r_y$ and $\alpha=\varphi_y-\varphi_x$. We saw earlier that a sufficient condition for the convergence of \eqref{eqn:addition_for_H} is $|r_y e^{\pm i\alpha}|/|r_x|<1$. The same condition is sufficient for the Laplace expansions. The following lemma shows that it follows from sufficient separation of $\Re\bx$ and $\Re\by$.

{
\begin{lemma}
\label{lma:bound_for_2d}
Let $0<r<R$ and let $\bx,\by\in\VpsiTwo$ with $\psi(0)=0$ satisfy
\begin{equation}
\label{eqn:real_separate}
    \|\Re \bx\| \ge R, \qquad \|\Re \by\| \le r.
\end{equation}
If
\begin{equation}
\label{eqn:2dLipcond}
    L_\psi < \frac{R-r}{\sqrt{Rr}},
\end{equation}
then
\begin{equation}
  \label{eqn:cplx_separate}
   \left| \frac{r_y e^{\pm i(\varphi_y-\varphi_x)}}{r_x} \right| \le C_{L_\psi,r,R}<1.
\end{equation}
\end{lemma}
}

\begin{proof}
{
First, we have
\begin{equation*}
  \left|\frac{r_y e^{\pm i(\varphi_y-\varphi_x)}}{r_x} \right| =\frac{|y_1\pm iy_2|}{|x_1\pm ix_2|}.
\end{equation*}
Applying Lemma~\ref{lma:pointwise_distortion_2d} to $\bx$ and $\by$ gives
\begin{equation*}
  \left|\frac{r_y e^{\pm i(\varphi_y-\varphi_x)}}{r_x}\right|
  \le\frac{s_+(L_\psi)}{s_-(L_\psi)}\frac{\|\Re\by\|}{\|\Re\bx\|} =\chi(L_\psi)\frac{\|\Re\by\|}{\|\Re\bx\|} \le \chi(L_\psi)\frac{r}{R}
  = C_{L_\psi,r,R},
\end{equation*}
and~\eqref{eqn:convergence_factor_equivalence}, together with~\eqref{eqn:2dLipcond}, shows that $C_{L_\psi,r,R}<1$.
}
\end{proof}

\begin{remark}
    At a fixed origin, Lemma~\ref{lma:bound_for_2d} remains valid if the source and target use two different recentered componentwise maps, provided both maps vanish at the origin and satisfy the same bound $L_\psi$. For shifted-center translation operators, however, the source, target, and relevant expansion center must be generated by the same map, or an additional compatibility condition between the maps is required. For clarity, we use a single map $\psi$ throughout.
\end{remark}

\subsection{The 2-D complex-coordinate Laplace FMM}
\label{sec:fmm_lap2d}

For the 2-D Laplace equation, the analytic continuation of the Green's function is given by
\begin{equation*}
    G(\bx,\by) =\frac{1}{2\pi} \log \frac{1}{r_{xy}},\qquad
    \bx,\by\in\VpsiTwo,\quad \bx\ne\by.
\end{equation*}
Let the source locations be $\{\by_{j}=(r_{y_j},\varphi_{y_j})\}_{j=1}^N \subset \VpsiTwo$, with corresponding densities $\{\sigma_j\}_{j=1}^N \subset \C$. Dropping an overall constant of~$-(2\pi)^{-1}$ (likewise for the other kernels), we focus on evaluating the potential 
\begin{equation}
\label{eqn:lap_2d_pot}
    u(\bx) = \sum_{j=1}^N \log (r_{xy_j}) \sigma_j, \quad
    \bx\in \VpsiTwo\setminus\{\by_1,\ldots,\by_N\}, \quad
    \text{ where }  \bx-\by_j=(r_{xy_j},\varphi_{xy_j}).
\end{equation}

\subsubsection{Multipole and Local Expansions}
The separation of variables formula we require for the 2-D Laplace kernel is based on the following addition formula.
\begin{lemma}[Addition Formula]
\label{lma:addtion_2d_laplace}
Under Assumption~\ref{assump:2d}, let $\bx=(r_x,\varphi_x)$ and $\by=(r_y,\varphi_y)$ lie in $\VpsiTwo$. If $|r_y e^{\pm i(\varphi_x-\varphi_y)}|<|r_x|$, then
\begin{equation}\label{eqn:lap2d_multi}
    \log r_{xy} = \log r_x-\frac{1}{2} \sum_{n=1}^\infty \frac{1}{n} \frac{r_y^n e^{-in\varphi_y}}{r_x^n} e^{in\varphi_x}-\frac{1}{2}\sum_{n=1}^\infty \frac{1}{n} \frac{r_y^n e^{in\varphi_y}}{r_x^n} e^{-in\varphi_x},
\end{equation}
where $r_{xy}=\sqrt{(x_1-y_1)^2+(x_2-y_2)^2}=\sqrt{r_x^2+r_y^2-2r_xr_y\cos(\varphi_x-\varphi_y)}$. For $P\ge1$, the truncation error satisfies
\begin{equation*}
   \left \vert \log(r_{xy})- \left(\log r_x-\frac{1}{2} \sum_{n=1}^P \frac{1}{n} \frac{r_y^n e^{-in\varphi_y}}{r_x^n} e^{in\varphi_x}-\frac{1}{2}\sum_{n=1}^P \frac{1}{n}  \frac{r_y^n e^{in\varphi_y}}{r_x^n} e^{-in\varphi_x}\right) \right\vert 
   \le  \frac{1}{c-1} \left(\frac{1}{c}\right)^P,
\end{equation*}
where
\begin{equation*}
    c=\max\left\{
    \frac{|r_y e^{i(\varphi_x-\varphi_y)}|}{|r_x|},
    \frac{|r_y e^{-i(\varphi_x-\varphi_y)}|}{|r_x|}
    \right\}^{-1}>1.
\end{equation*}
\end{lemma}

\begin{proof}
We set $u=r_y/r_x$ and $\alpha=\varphi_y-\varphi_x$. Equation~\eqref{eqn:lap2d_multi} is well known for real points and is derived by rewriting
\begin{equation*}
\log r_{xy}=\log r_x+\frac12\log(1+u^2-2u\cos\alpha).
\end{equation*}
The assumption on the points gives a convergent analytic series, so the result also holds for complex points. The error bound follows directly from standard estimates for truncated geometric series.

\end{proof}

In this subsection, for a center $\bc\in\C^2$, we  write $\bx-\bc=(r_{xc},\varphi_{xc})$. Let $M=\{M_n^\pm\}_{n\ge 0}$ and $L=(\{L_n^\pm\}_{n\ge 0}$ be the multipole and local expansion coefficients respectively. For every nonnegative integer $P$, we define the degree-$P$ multipole and local expansions of~\eqref{eqn:lap_2d_pot} centered at  $\bc$ by
\begin{align}
\label{eqn:mul_expansion_lap2d}
    u^{\mathrm{mp}}_{\bc,P}[M](\bx)
    &:=M_0\log r_{xc}
      +\sum_{n=1}^P\frac{M_n^+e^{in\varphi_{xc}}
      +M_n^-e^{-in\varphi_{xc}}}{r_{xc}^n},\\
\label{eqn:loc_expansion_lap2d}
    u^{\mathrm{loc}}_{\bc,P}[L](\bx)
    &:=L_0+\sum_{n=1}^P
      \left(L_n^+e^{in\varphi_{xc}}
      +L_n^-e^{-in\varphi_{xc}}\right)r_{xc}^n.
\end{align}

\begin{theorem}[Multipole expansion]
\label{thm:form_multipole_lap2d}
Under Assumption~\ref{assump:2d}, let $0<r<R$, let the source locations $\{\by_j\}_{j=1}^N\subset\VpsiTwo$ satisfy $\|\Re\by_j\|\le r$ for $j=1,\ldots,N$, and let $\bx=(r_x,\varphi_x)\in\VpsiTwo$ satisfy $\|\Re\bx\|\ge R$. Assume that $L_\psi<\frac{R-r}{\sqrt{Rr}}$, so that $C_{L_\psi,r,R}<1$. We define the following expansion coefficients
\begin{equation}
\label{eqn:lap2d_mp_coef}
    M_0 = \sum_{j=1}^N \sigma_j, \quad 
    M_n^\pm = -\frac{1}{2n} \sum_{j=1}^N r_{y_j}^n e^{\mp i n \varphi_{y_j}} \sigma_j, \quad n \geq 1.
\end{equation}
Then $u(\bx)=\lim\limits_{P\to\infty}u^{\mathrm{mp}}_{\boldsymbol 0,P}[M](\bx)$, and, for $P\ge1$,
\begin{equation}
\label{eqn:l2d_mp_err}
    \left|u(\bx)-u^{\mathrm{mp}}_{\boldsymbol 0,P}[M](\bx)\right|
    \leq \|\vec{\sigma}\|_1\frac{(C_{L_\psi,r,R})^{P+1}}{1-C_{L_\psi,r,R}}.
\end{equation}
\end{theorem}
\begin{proof}
Applying Lemma~\ref{lma:addtion_2d_laplace} to each $\log(r_{xy_j})$ term in~\eqref{eqn:lap_2d_pot} and collecting the coefficient of each basis function gives~\eqref{eqn:lap2d_mp_coef}. Lemma~\ref{lma:bound_for_2d} and the triangle inequality give~\eqref{eqn:l2d_mp_err}.
\end{proof}

Using a similar argument, we can also prove the convergence of local expansions.
\begin{theorem}[Local expansion]
\label{thm:form_local_lap2d}
Under Assumption~\ref{assump:2d}, let $0<r<R$, let the source locations $\{\by_j\}_{j=1}^N\subset\VpsiTwo$ satisfy $\|\Re\by_j\|\ge R$ for $j=1,\ldots,N$, and let $\bx\in\VpsiTwo$ satisfy $\|\Re\bx\|\le r$. Assume that
$
    L_\psi<\frac{R-r}{\sqrt{Rr}},
$
so that $C_{L_\psi,r,R}<1$. We define the following expansion coefficients
\begin{equation}
\label{eqn:lap2d_loc_coef}
    L_0 = \sum_{j=1}^N \log r_{y_j} \sigma_j, \quad 
    L_n^\pm = -\frac{1}{2n} \sum_{j=1}^N r_{y_j}^{-n} e^{\mp i n \varphi_{y_j}} \sigma_j, \quad n \geq 1.
\end{equation}
Then $u(\bx)=\lim\limits_{P\to\infty}u^{\mathrm{loc}}_{\boldsymbol 0,P}[L](\bx)$, and, for $P\ge1$,
\begin{equation*}
    \left|u(\bx)-u^{\mathrm{loc}}_{\boldsymbol 0,P}[L](\bx)\right|
    \leq \|\vec{\sigma}\|_1\frac{(C_{L_\psi,r,R})^{P+1}}{1-C_{L_\psi,r,R}}.
\end{equation*}
\end{theorem}

\subsubsection{Translation operators}
The multipole-to-multipole (M2M), multipole-to-local (M2L), and local-to-local (L2L) translation operators can be derived using the following addition formulas. The proofs of these formulas are based on binomial expansions and are omitted. 
\begin{lemma}
\label{lma:addition_formula_lap2d_trans}
For $\bx=(r_x,\varphi_x), \by =(r_y,\varphi_y) \in \VpsiTwo$, let $\bx-\by=(r_{xy}, \varphi_{xy})$. If $|r_y e^{\pm i(\varphi_x-\varphi_y)}|<|r_x|$, then for $n\ge 1$,
\begin{equation}
\label{eqn:firstaddition_lap2d}
\frac{e^{\pm in\varphi_{xy}}}{r_{xy}^n} =\sum_{k=0}^\infty \binom{n+k-1}{k} r_y^ke^{\mp ik\varphi_y}  \frac{ e^{\pm i(n+k)\varphi_x}}{r_x^{n+k}}.
\end{equation}
If $|r_x e^{\pm i(\varphi_x-\varphi_y)}|<|r_y|$, then for $n\ge 1$,
\begin{equation}
\label{eqn:secondaddition_lap2d}
   \frac{e^{\pm in\varphi_{xy}}}{r_{xy}^n} = (-1)^n \sum_{k=0}^\infty \binom{n+k-1}{k} \frac{ e^{\pm i(n+k)\varphi_y}}{r_y^{n+k}} r_x^k e^{\mp ik\varphi_x}  .
\end{equation}
For any $\bx,\by$ and $n\ge1$, the following identity holds:
\begin{equation}
\label{eqn:thirdaddition_lap2d}
  e^{\pm in\varphi_{xy}}r_{xy}^n = \sum_{k=0}^n \binom{n}{k} (-1)^k  r_y^k e^{\pm ik\varphi_y} r_x^{n-k} e^{\pm i(n-k)\varphi_x}.
\end{equation}
\end{lemma}

\begin{theorem}[M2M]
\label{thm:lap2d_m2m}
Under Assumption~\ref{assump:2d}, let $r>0$ and $\bx_0,\by_1,\ldots,\by_N\in\VpsiTwo$ satisfy $\|\Re\by_j-\Re\bx_0\|<r$. Let $R>r+\|\Re\bx_0\|$ and assume that $L_\psi<\frac{R-r-\|\Re\bx_0\|}{\sqrt{R(r+\|\Re\bx_0\|)}}$, so that $C_{L_\psi,r+\|\Re\bx_0\|,R}<1$. For densities $\{\sigma_j\}_{j=1}^N$, we define $M=(M_0,\{M_n^\pm\}_{n\ge1})$ by~\eqref{eqn:lap2d_mp_coef}, using the polar coordinates of $\by_j-\bx_0$, where $\bx_0=(r_0,\varphi_0)$. The translated multipole expansion coefficients centered at zero are
\begin{equation}
\label{eqn:M2M_lap2d}
    \tilde M_0 = M_0,\quad
    \tilde M_k^\pm = -\frac{1}{2k}r_0^k e^{\mp i k \varphi_0}M_0+\sum_{n=1}^k  \binom{k-1}{k-n} r_0^{k-n} e^{\mp i(k-n)\varphi_0} M_n^{\pm},\qquad k\ge1.
\end{equation}
For $\bx\in\VpsiTwo$ with $\|\Re\bx\|>R$, $u(\bx)=\lim\limits_{P\to\infty}u^{\mathrm{mp}}_{\boldsymbol 0,P}[\tilde M](\bx)$. Moreover, for $P\ge1$,
\begin{equation}
\label{eqn:M2M_lap2d_err}
    \left|u(\bx)-u^{\mathrm{mp}}_{\boldsymbol 0,P}[\tilde M](\bx)\right|
    \le \|\vec{\sigma}\|_1\frac{(C_{L_\psi,r+\|\Re\bx_0\|,R})^{P+1}}
    {1-C_{L_\psi,r+\|\Re\bx_0\|,R}}.
\end{equation}
\end{theorem}
\begin{proof}
We verified the translated multipole expansion coefficients  from~\eqref{eqn:lap2d_mp_coef}. For each fixed $k\ge 1$, we expand $r_{y_j}^k e^{\mp ik\varphi_{y_j}}$ using $\by_j=\bx_0+(\by_j-\bx_0)$ and the finite binomial formula. Multiplying by $-\frac{\sigma_j}{2k}$ and summing over the sources gives~\eqref{eqn:M2M_lap2d}, since $\frac{n}{k}{k \choose n}={k-1 \choose k-n}$ for $1\le n\le k$. The zeroth coefficient remains unchanged. This calculation gives exactly the multipole coefficients defined by (4.17) with expansion center at the origin, using only finite sums.

Finally, since $\|\Re y_j\|<r+\|\Re \bx_0\|$, we apply Theorem~\ref{thm:form_multipole_lap2d} with $r$ replaced by $r+\|\Re \bx_0\|$, to get the convergence and error bound. 
\end{proof}

\begin{theorem}[M2L]
\label{thm:lap2d_m2l}
Under Assumption~\ref{assump:2d}, let $0<r<R$, $\bx_0=(r_0,\varphi_0)\in\VpsiTwo$, and $\by_1,\ldots,\by_N\in\VpsiTwo$ satisfy $\|\Re\by_j-\Re\bx_0\|<r$. Assume that $\|\Re\bx_0\|>R+r$ and $L_\psi<\frac{R-r}{\sqrt{Rr}}$, so that $C_{L_\psi,r,R}<1$. For densities $\{\sigma_j\}_{j=1}^N$, we define $M=(M_0,\{M_n^\pm\}_{n\ge1})$ by~\eqref{eqn:lap2d_mp_coef}, using the polar coordinates of $\by_j-\bx_0$. The coefficients of the converted local expansion centered at zero are
\begin{equation}
\label{eqn:M2L_lap2d}
\begin{split}
    \tilde L_0 = M_0\log r_0+\sum_{n=1}^\infty (-1)^n \frac{e^{in\varphi_0}}{r_0^n}M_n^+ +\sum_{n=1}^{\infty} (-1)^n \frac{e^{-in\varphi_0}}{r_0^n}M_n^-, \\
    \tilde L_k^\pm =-\frac{1}{2k} \frac{e^{\mp ik \varphi_0}}{r_0^k} M_0+ \sum_{n=1}^\infty (-1)^n  \binom{n+k-1}{k} \frac{e^{\mp i(n+k)\varphi_0}}{r_0^{n+k}} M_n^{\mp},\qquad k\ge1.
\end{split}
\end{equation}
For $\bx\in\VpsiTwo$ with $\|\Re\bx\|<r$, $u(\bx)=\lim\limits_{P\to\infty}u^{\mathrm{loc}}_{\boldsymbol 0,P}[\tilde L](\bx)$. Moreover, for $P\ge1$,
\begin{equation}
\label{eqn:M2L_lap2d_err}
    \left|u(\bx)-u^{\mathrm{loc}}_{\boldsymbol 0,P}[\tilde L](\bx)\right|
    \le \|\vec{\sigma}\|_1\frac{(C_{L_\psi,r,R})^{P+1}}{1-C_{L_\psi,r,R}}.
\end{equation}
\end{theorem}
\begin{proof}
We verify each converted local expansion coefficient in~\eqref{eqn:lap2d_loc_coef}. We apply Lemma~\ref{lma:pointwise_distortion_2d} at zero and $\bx_0$, which gives
\begin{equation}
\label{eqn:point_estimate_m2l_2d}
    \left\vert \frac{r_{y_j x_0} e^{\pm i(\varphi_{y_jx_0} -\varphi_0)}}{r_0}\right\vert \le \chi(L_\psi) \frac{\|\Re \by_j - \Re \bx_0\|}{\|\Re \bx_0\|}<C_{L_\psi,r,R}<1.
\end{equation}
For each fixed $k\ge 1$, the binomial expansion gives 
\begin{equation}
    \frac{e^{\mp ik\varphi_{y_j}}}{r_{y_j}^k}= \sum_{n=0}^\infty (-1)^n {k+n-1 \choose n}\frac{e^{\mp i(k+n)\varphi_0}}{r_0^{k+n}} r_{y_jx_0}^n e^{\pm in\varphi_{y_jx_0}}. 
\end{equation}
We substitute this identity into~\eqref{eqn:lap2d_loc_coef} and use the multipole coefficients defined by~\eqref{eqn:lap2d_mp_coef} centered at $\bx_0$ to obtain the nonzero-th order coefficients in~\eqref{eqn:M2L_lap2d}, since $\frac{n}{k}{k+n-1 \choose n}={k+n-1 \choose k}$ for $n\ge 1$. The zero-th order coefficient follows the same way by expanding $\log r_{y_j}$ about $\bx_0$ using Lemma~\ref{lma:addtion_2d_laplace}.  Each series is absolutely convergent for the fixed local expansion coefficient, so only the finite sum over sources is interchanged with a convergent series. Therefore~\eqref{eqn:M2L_lap2d} gives precisely the coefficients in~\eqref{eqn:lap2d_loc_coef}. 

Finally, since $\| \Re y_j \|>R$ and $\|\Re \bx\|<r$, so applying Theorem~\ref{thm:form_local_lap2d} gives the error bound. 
\end{proof}

The following theorem provides a method for shifting the center of a truncated local expansion. The proof is an immediate consequence of the addition formula \eqref{eqn:thirdaddition_lap2d}.
\begin{theorem}[L2L]
\label{thm:lap2d_l2l}
Let $P\ge0$, let $\bx_0=(r_0,\varphi_0)\in\C^2$, and let $L=(L_0,\{L_n^\pm\}_{n=1}^P)$. Then
$
    u^{\mathrm{loc}}_{\bx_0,P}[L](\bx)
    =u^{\mathrm{loc}}_{\boldsymbol 0,P}[\bar L](\bx),
$
where the shifted local expansion coefficients
\begin{equation}
\label{eqn:L2L_lap2d}
\begin{split}
\bar L_0 = L_0+\sum_{n=1}^P (-1)^nr_0^n e^{in\varphi_0}L_n^+ + \sum_{n=1}^P (-1)^n r_0^n e^{-in\varphi_0}L_n^-\\
 \bar L_k^\pm =\sum_{n= k}^P (-1)^{n-k} \binom{n}{k} r_0^{n-k} e^{\pm i(n-k)\varphi_0} L_n^\pm,\qquad 1\le k\le P.
 \end{split}
\end{equation}
\end{theorem}

\subsection{The  2-D complex-coordinate Helmholtz FMM}
\label{sec:fmm_helm2d}
The analytically continued Green's function is
\begin{equation*}
G_\kappa(\bx,\by) = \frac{i}{4} H_0^{(1)}(\kappa r_{xy}),\qquad
\bx,\by\in\VpsiTwo,\quad \bx\ne\by.
\end{equation*}
Let the source locations $\{\by_j=(r_{y_j},\varphi_{y_j})\}_{j=1}^N$ lie in $\VpsiTwo$, with associated densities $\{\sigma_j\}_{j=1}^N$. We consider calculating the potential 
\begin{equation}
\label{eqn:helm_2d_pot}
    u(\bx) = \sum_{j=1}^N H_0^{(1)}(\kappa r_{xy_j}) \sigma_j,
    \quad \bx\in \VpsiTwo\setminus\{\by_1,\ldots,\by_N\},
    \quad \text{ where }  \bx-\by_j=(r_{xy_j},\varphi_{xy_j}). 
\end{equation}

\subsubsection{Multipole and local expansions}
For a center $\bc\in\C^2$, we write $\bx-\bc=(r_{xc},\varphi_{xc})$. Let $M=(M_n)_{n\in\mathbb Z}$ and $L=(L_n)_{n\in\mathbb Z}$ be the multipole and local expansion coefficients respectively. For every nonnegative integer $P$, the degree-$P$ multipole and local expansions of~\eqref{eqn:helm_2d_pot} centered at $\bc$ are
\begin{align}
\label{eqn:mp_expansion_helm2d}
    u^{\mathrm{mp}}_{\bc,P}[M](\bx)
    &:=\sum_{n=-P}^P M_n H_n^{(1)}(\kappa r_{xc})
      e^{in\varphi_{xc}},\\
    u^{\mathrm{loc}}_{\bc,P}[L](\bx)
    &:=\sum_{n=-P}^P L_n J_n(\kappa r_{xc})
      e^{in\varphi_{xc}}.
\end{align}

\begin{theorem}[Multipole expansion]
\label{thm:form_mp_helm2d}
Under Assumption~\ref{assump:2d}, let $0<r<R$, let the source locations $\{\by_j\}_{j=1}^N\subset\VpsiTwo$ satisfy $\|\Re\by_j\|<r$ for $j=1,\ldots,N$, and let $\bx\in\VpsiTwo$ satisfy $\|\Re\bx\|>R$. Assume that $L_\psi<\frac{R-r}{\sqrt{Rr}}$, so that $C_{L_\psi,r,R}<1$. Define $M_n=\sum_{j=1}^N J_n(\kappa r_{y_j}) e^{-in\varphi_{y_j}}\sigma_j$, $n\in\mathbb Z$. Then $u(\bx)=\lim\limits_{P\to\infty}u^{\mathrm{mp}}_{\boldsymbol 0,P}[M](\bx)$, and, for sufficiently large $P$,
\begin{equation}
\label{eqn:h2d_err_mp}
    \left|u(\bx)-u^{\mathrm{mp}}_{\boldsymbol 0,P}[M](\bx)\right|
    \lesssim \frac{2\|\vec{\sigma}\|_1}{P\pi}
    \frac{(C_{L_\psi,r,R})^{P+1}}{1-C_{L_\psi,r,R}}.
\end{equation}
\end{theorem}
\begin{proof}
Applying Graf's addition formula in Theorem~\ref{thm:graf2d} to each $H_0^{(1)}(\kappa r_{xy_j})$ term in~\eqref{eqn:helm_2d_pot} and collecting the coefficient of each basis function gives $M_n$ for $n\in\mathbb Z$. Lemma~\ref{lma:bound_for_2d} and the triangle inequality then give the result.
\end{proof}

A similar result holds for the local expansion of the Helmholtz kernel.
\begin{theorem}[Local expansion]
\label{thm:local_helm}
Under Assumption~\ref{assump:2d}, let $0<r<R$, let the source locations $\{\by_j\}_{j=1}^N\subset\VpsiTwo$ satisfy $\|\Re\by_j\|>R$ for $j=1,\ldots,N$, and let $\bx\in\VpsiTwo$ satisfy $\|\Re\bx\|<r$. Assume that $L_\psi<\frac{R-r}{\sqrt{Rr}}$, so that $C_{L_\psi,r,R}<1$. We define $L_n=\sum_{j=1}^N H_n^{(1)}(\kappa r_{y_j}) e^{-in\varphi_{y_j}}\sigma_j$, $n\in\mathbb Z$. Then $u(\bx)=\lim\limits_{P\to\infty}u^{\mathrm{loc}}_{\boldsymbol 0,P}[L](\bx)$, and, for sufficiently large $P$,
\begin{equation*}
    \left|u(\bx)-u^{\mathrm{loc}}_{\boldsymbol 0,P}[L](\bx)\right|
    \lesssim \frac{2\|\vec{\sigma}\|_1}{P\pi}
    \frac{(C_{L_\psi,r,R})^{P+1}}{1-C_{L_\psi,r,R}}.
\end{equation*}
\end{theorem}

\subsubsection{Translation operators}
The translation operators M2M, M2L, and L2L follow directly from Graf's addition formula in Theorem~\ref{thm:graf2d}, and their truncation error bounds reduce to the multipole and local expansion theorems above.

\begin{theorem}[M2M]
\label{thm:helm2d_m2m}
Under Assumption~\ref{assump:2d}, let $r>0$ and $\bx_0,\by_1,\ldots,\by_N\in\VpsiTwo$ satisfy $\|\Re\by_j-\Re\bx_0\|<r$. Let $R>r+\|\Re\bx_0\|$ and assume that $L_\psi<\frac{R-r-\|\Re\bx_0\|}{\sqrt{R(r+\|\Re\bx_0\|)}}$, so that $C_{L_\psi,r+\|\Re\bx_0\|,R}<1$. For densities $\{\sigma_j\}_{j=1}^N$, we define $M_n$, $n\in\mathbb Z$, by the formula in Theorem~\ref{thm:form_mp_helm2d}, using the polar coordinates of $\by_j-\bx_0$, where $\bx_0=(r_0,\varphi_0)$. The translated multipole expansion coefficients centered at zero are
\begin{equation}
\label{eqn:M2M_helm2d}
    \tilde M_m = \sum_{n=-\infty}^\infty J_{m-n}(\kappa r_0) e^{-i(m-n)\varphi_0} M_n,\qquad m\in\mathbb Z.
\end{equation}
For $\bx\in\VpsiTwo$ with $\|\Re\bx\|>R$, $u(\bx)=\lim\limits_{P\to\infty}u^{\mathrm{mp}}_{\boldsymbol 0,P}[\tilde M](\bx)$. For sufficiently large $P$,
\begin{equation*}
    \left|u(\bx)-u^{\mathrm{mp}}_{\boldsymbol 0,P}[\tilde M](\bx)\right|
    \lesssim \frac{2\|\vec{\sigma}\|_1}{P\pi}
    \frac{(C_{L_\psi,r+\|\Re\bx_0\|,R})^{P+1}}
    {1-C_{L_\psi,r+\|\Re\bx_0\|,R}}.
\end{equation*}
\end{theorem}
\begin{proof}
For each $m$, using Graf's addition formula~\eqref{eqn:addition_for_J} applied to $\by_j=\bx_0+(\by_j-\bx_0)$, we have 
\begin{equation}
\label{eqn:m2m_bessel_expand}
    J_m(\kappa r_{y_j}) e^{-im\varphi_{y_j}} = \sum_{n=-\infty}^\infty J_{m-n}(\kappa r_0) e^{-i(m-n)\varphi_0} J_n(\kappa r_{y_j x_0}) e^{-in\varphi_{y_j x_0}}
\end{equation}
which converges absolutely by the large-order decay of the Bessel functions. We multiply~\eqref{eqn:m2m_bessel_expand} by $\sigma_j$ and summing over the sources therefore gives~\eqref{eqn:M2M_helm2d} using the definition of multipole expansion coefficients $M_n$ with centered $\bx_0$. Therefore, the  translated coefficients are precisely the multipole coefficients in Theorem~\ref{thm:form_mp_helm2d} centered at zero. 

Since $\|\Re \by_j\|<r+\|\Re \bx_0\|$, we can use Theorem~\ref{thm:form_mp_helm2d} with $r$ replaced by $r+\|\Re \bx_0\|$ to conclude the convergence and error bound.
\end{proof}

\begin{theorem}[M2L]
\label{thm:helm2d_m2l}
Under Assumption~\ref{assump:2d}, let $0<r<R$, $\bx_0=(r_0,\varphi_0)\in\VpsiTwo$, and $\by_1,\ldots,\by_N\in\VpsiTwo$ satisfy $\|\Re\by_j-\Re\bx_0\|<r$. Assume that $\|\Re\bx_0\|>R+r$ and $L_\psi<\frac{R-r}{\sqrt{Rr}}$, so that $C_{L_\psi,r,R}<1$. For densities $\{\sigma_j\}_{j=1}^N$, we define $M_n$, $n\in\mathbb Z$, by the formula in Theorem~\ref{thm:form_mp_helm2d}, using the polar coordinates of $\by_j-\bx_0$. The coefficients of the converted local expansion centered at zero are
\begin{equation}
\label{eqn:M2L_helm2d}
    L_m = \sum_{n=-\infty}^\infty H^{(1)}_{m-n}(\kappa r_0)
    e^{-i(m-n)\varphi_0} M_n,\qquad m\in\mathbb Z.
\end{equation}
For $\bx\in\VpsiTwo$ with $\|\Re\bx\|<r$, $u(\bx)=\lim\limits_{P\to\infty}u^{\mathrm{loc}}_{\boldsymbol 0,P}[L](\bx)$. For sufficiently large $P$,
\begin{equation*}
    \left|u(\bx)-u^{\mathrm{loc}}_{\boldsymbol 0,P}[L](\bx)\right|
    \lesssim \frac{2\|\vec{\sigma}\|_1}{P\pi}
    \frac{(C_{L_\psi,r,R})^{P+1}}{1-C_{L_\psi,r,R}}.
\end{equation*}
\end{theorem}
\begin{proof}
For fixed $m$, we verify the local expansion coefficient in Theorem~\eqref{thm:local_helm}.  We use the estimate~\eqref{eqn:point_estimate_m2l_2d} used in Theorem~\ref{thm:lap2d_m2l} and the Graf's addition formula~\eqref{eqn:addition_for_H} applied to $\by_j=(\by_j-\bx_0)+\bx_0$, to obtain 
\begin{equation}
\label{eqn:hankel_addition_m2l}
    H_m^{(1)}(\kappa r_{y_j}) e^{-im\varphi_{y_j}}=\sum_{n=-\infty}^\infty H_{m-n}^{(1)}(\kappa r_0) e^{-i(m-n)\varphi_0} J_n(\kappa r_{y_j x_0}) e^{-in\varphi_{y_ix_0}}.
\end{equation}
We multiply~\eqref{eqn:hankel_addition_m2l} by $\sigma_j$ and sum over the sources to yield~\eqref{eqn:M2L_helm2d}, where we also use the definition of local expansion coefficient $L_n$ centered at $\bx_0$. Therefore, the translated coefficients are the   local expansion coefficients in Theorem~\ref{thm:local_helm}. 

Since $\|\Re\by_j\|>R$ and $\|\Re \bx\|<r$, Theorem~\ref{thm:local_helm} implies the error bound. 
\end{proof}

\begin{theorem}[L2L]
\label{thm:helm2d_l2l}
Under Assumption~\ref{assump:2d}, let $R>0$ and $\bx_0,\by_1,\ldots,\by_N\in\VpsiTwo$ satisfy $\|\Re\by_j-\Re\bx_0\|>R$. Let $r>0$ satisfy $R>r+\|\Re\bx_0\|$ and assume that $L_\psi<\frac{R-\|\Re\bx_0\|-r} {\sqrt{R(r+\|\Re\bx_0\|)}}$, so that $C_{L_\psi,r+\|\Re\bx_0\|,R}<1$. Since $\frac{r}{R-\|\Re\bx_0\|} \le\frac{r+\|\Re\bx_0\|}{R}$, we also have $C_{L_\psi,r,R-\|\Re\bx_0\|}<1$. For densities $\{\sigma_j\}_{j=1}^N$, we define $L_n$, $n\in\mathbb Z$, by the formula in Theorem~\ref{thm:local_helm}, using the polar coordinates of $\by_j-\bx_0$, where $\bx_0=(r_0,\varphi_0)$. The translated local expansion coefficients are
\begin{equation}
\label{eqn:L2L_helm2d}
    \tilde L_m=\sum_{n=-\infty}^{\infty}J_{m-n}(\kappa r_0)e^{-i(m-n)\varphi_0}L_n,\qquad m\in\mathbb Z.
\end{equation}
For $\bx\in\VpsiTwo$ with $\|\Re\bx\|<r$, $u(\bx)=\lim\limits_{P\to\infty}u^{\mathrm{loc}}_{\boldsymbol 0,P}[\tilde L](\bx)$. For sufficiently large $P$,
\begin{equation*}
    \left|u(\bx)-u^{\mathrm{loc}}_{\boldsymbol 0,P}[\tilde L](\bx)\right|
    \lesssim\frac{2\|\vec{\sigma}\|_1}{P\pi}
    \frac{(C_{L_\psi,r,R-\|\Re\bx_0\|})^{P+1}}
    {1-C_{L_\psi,r,R-\|\Re\bx_0\|}}.
\end{equation*}
\end{theorem}
\begin{proof}
For fixed $m$, we can apply Lemma~\ref{lma:pointwise_distortion_2d} at zero and $\bx_0$ to get 
\begin{equation}
    \left \vert \frac{r_0 e^{\pm i(\varphi_0-\varphi_{y_j x_0})}}{r_{y_jx_0}}\right\vert  \le \chi(L_\psi) \frac{\|\Re \bx_0\|}{\|\Re \by_j - \Re \bx_0\|}<C_{L_\psi,r+\|\Re\bx_0\|,R}<1.
\end{equation}
We then apply Graf's addition formula~\eqref{eqn:addition_for_H} to $\by_j=(\by_j-\bx_0)+\bx_0$ to obtain
\begin{equation}
\label{eqn:hankel_addition_l2l}
    H_m^{(1)}(\kappa r_{y_j}) e^{-im\varphi_{y_j}}=\sum_{n=-\infty}^\infty J_{m-n}(\kappa r_0) e^{-i(m-n)\varphi_0} H_n^{(1)}(\kappa r_{y_j x_0}) e^{-in\varphi_{y_ix_0}}.
\end{equation}
Similarly, we can multiply $\sigma_j$ and sum over sources to get the translated local coefficients. 

Finally, since $\|\Re \by_j\|>R-\|\Re \bx_0\|$ and $C_{L_\psi,r,R-\|\Re\bx_0\|}<1$, we can apply Theorem~\ref{thm:local_helm} with $R$ replaced by $R-\|\Re\bx_0\|$ to get the error bound. 
\end{proof}

\section{The 3-D complex-coordinate FMM}
\label{sec:zfmm3d}
In this section, we develop the three-dimensional complex-coordinate FMM for the Laplace and Helmholtz kernels.

\subsection{Geometric assumptions}
We denote the three-dimensional analogue of $\VpsiTwo$ by $\VpsiThree$.

\begin{assumption}[3-D Assumption]
\label{assump:3d}
{
We assume that the source locations, target locations, and expansion centers are vectors in $\C^3$ satisfying
$
    \Im\bx=\psi(\Re\bx),
$
where $\psi:\R^3\to\R^3$ has the form
$
    \psi(\bx)
    =\bigl(\psi_1(x_1),\psi_2(x_2),\psi_3(x_3)\bigr).
$
For some $L_\psi<1$, each smooth function $\psi_j:\R\to\R$, $j=1,2,3$, satisfies
$$
    0\le \psi_j(t)-\psi_j(s)\le L_\psi(t-s),
    \qquad s\le t.
$$
}
\end{assumption}
Let $\VpsiThree\subset \mathbb{C}^3$ denote the set of points satisfying the constraint $\Im  \bx  = \psi(\Re{\bx}),  \bx \in \mathbb{C}^3$. As in the two-dimensional analysis, whenever an expansion center is placed at the origin, the coordinates and the complexification map are understood to have been recentered at that center. In particular, the recentered map satisfies $\psi(0)=0$ and the same componentwise bounds as the original map.
{
 The three-dimensional result will use the following matrix map. For $\bz=(z_1,z_2,z_3)\in\C^3$, we define $\mathcal P$\footnote{The map $\mathcal P$ is the Pauli-matrix representation $\mathcal P(\bz)=z_1\sigma_1+z_2\sigma_2+z_3\sigma_3$, where $\sigma_1,\sigma_2,\sigma_3$ are the Pauli matrices~\cite{sakurai2021modern}.} by
\begin{equation}
\label{eqn:pauli_map}
    \mathcal P(\bz)
    :=
    \begin{pmatrix}
        z_3 & z_1-iz_2\\
        z_1+iz_2 & -z_3
    \end{pmatrix}.
\end{equation}
For any $\bz_1,\bz_2\in\C^3$, direct multiplication gives
\begin{equation}
\label{eqn:pauli_identities}
    \mathcal P(\bz)^2=(\bz\cdot\bz)I,
    \qquad
    \mathcal P(\bz_1)\mathcal P(\bz_2)
    =(\bz_1\cdot\bz_2)I+i\mathcal P(\bz_1\times\bz_2).
\end{equation}
If $\bz\cdot\bz\ne0$, the first identity gives
\begin{equation}
\label{eqn:pauli_inverse}
    \mathcal P(\bz)^{-1}
    =\frac{\mathcal P(\bz)}{\bz\cdot\bz}.
\end{equation}
Moreover,
\begin{equation}
  \label{eqn:pauli_det}
  \det(\mathcal P(\bz))=-\bz\cdot\bz=-\rho_z^2,
  \qquad |\det(\mathcal P(\bz))|=|\rho_z|^2.
\end{equation}
}

{
\begin{lemma}
\label{lma:pointwise_distortion_3d}
Let $\psi:\R^3\to\R^3$ be the complexification scheme satisfying Assumption~\ref{assump:3d}. Assume that $\psi(0)=0$. Then, for $\bz\in \VpsiThree$, we have
\begin{align}
\label{eqn:pauli_uniform_bounds}
    &s_-(L_\psi)\|\Re\bz\|
    \le \sigma_{\min}(\mathcal P(\bz))
    \le \sigma_{\max}(\mathcal P(\bz))
    \le s_+(L_\psi)\|\Re\bz\|,\\
&\label{eqn:rho_lower_bound}
    |\rho_z|\ge s_-(L_\psi)\|\Re\bz\|
    =\frac{\|\Re \bz\|}{\sqrt{\chi(L_\psi)}}.
\end{align}
\end{lemma}
}

\begin{proof}
{
Let $\bz=\bp+i\psi(\bp)$. If $\bp=0$, then $\bz=0$ and the result is immediate. Hence, assume that $0\ne\bp\in\R^3$. For $j=1,2,3$, we define
\begin{equation*}
    \gamma_j=
    \begin{cases}
        \psi_j(p_j)/p_j, & p_j\ne0,\\
        0, & p_j=0.
    \end{cases}
\end{equation*}
Since $\psi(0)=0$, componentwise monotonicity and the Lipschitz condition imply
\begin{equation*}
    0\le\gamma_j\le L_\psi,\qquad
    \psi_j(p_j)=\gamma_jp_j,\qquad
    \|\psi(\bp)\|\le L_\psi\|\bp\|.
\end{equation*}
We first prove the spectral bounds in~\eqref{eqn:pauli_uniform_bounds}. We set
$S=\|\bz\|^2=\|\bp\|^2+\|\psi(\bp)\|^2$ and $\bt=\bp\times\psi(\bp).$

Since $\mathcal P(\bz)^*=\mathcal P(\overline{\bz})$, we can calculate
\begin{equation}
\label{eqn:pauli_H_matrix}
    H(\bz):=\mathcal P(\bz)^*\mathcal P(\bz)
    =\begin{pmatrix}
        S-2t_3 & -2t_1+2it_2\\
        -2t_1-2it_2 & S+2t_3
      \end{pmatrix}.
\end{equation}
The eigenvalues of $H(\bz)$ are the squared singular values of $\mathcal P(\bz)$. Since
$\operatorname{tr}H(\bz)=2S$ and $\det H(\bz)=S^2-4\|\bt\|^2,$
we obtain
\begin{equation}
\label{eqn:pauli_singular_values}
\begin{split}
    \sigma_{\max}(\mathcal P(\bz))^2=\|\bp\|^2+\|\psi(\bp)\|^2
    +2\|\bp\times\psi(\bp)\|, \\ 
    \sigma_{\min}(\mathcal P(\bz))^2 =\|\bp\|^2+\|\psi(\bp)\|^2
    -2\|\bp\times\psi(\bp)\|.
\end{split}
\end{equation}

We define the following quantities
\begin{equation*}
  \mu = \frac{\bp\cdot \psi(\bp)}{\|\bp\|^2}, \qquad \delta = \frac{\|\bp\times \psi(\bp)\|}{\|\bp\|^2}, \qquad \nu = \frac{\|\psi(\bp)\|^2}{\|\bp\|^2}.
\end{equation*}

Since $0\le\gamma_j\le L_\psi$, we have
\begin{equation}
\label{eqn:mu_nu_delta_bounds}
    \nu\le L_\psi\mu,\qquad
    \delta^2=\nu-\mu^2\le\mu(L_\psi-\mu),\qquad
    0\le\delta\le\frac{L_\psi}{2}<1.
\end{equation}
Here, the last inequality follows from $\mu(L_\psi-\mu)\le L_\psi^2/4$. It follows from~\eqref{eqn:pauli_singular_values} that
\begin{align}
\frac{\sigma_{\max}(\mathcal P(\bz))^2}{\|\bp\|^2}
&=1+\nu+2\delta
 \le 1+L_\psi\mu+2\sqrt{\mu(L_\psi-\mu)},
\label{eqn:pauli_sigma_upper}\\
\frac{\sigma_{\min}(\mathcal P(\bz))^2}{\|\bp\|^2}
&=1+\mu^2+\delta^2-2\delta
 \ge 1+L_\psi\mu-2\sqrt{\mu(L_\psi-\mu)}.
\label{eqn:pauli_sigma_lower}
\end{align}
For the second inequality, we used the fact that $\delta^2-2\delta$ is decreasing on $[0,1]$.

If $L_\psi=0$, the desired bounds follow immediately. For $L_\psi>0$, we set $t=\mu/L_\psi\in[0,1]$ and $\boldsymbol w=(\sqrt{1-t},\sqrt t)^T$. Then the two right-hand sides above are
\begin{equation*}
    \boldsymbol w^T
    \begin{pmatrix} 
    1 & L_\psi\\
    L_\psi & 1+L_\psi^2
    \end{pmatrix} 
    \boldsymbol w,
    \quad
    \boldsymbol w^T
    \begin{pmatrix}
    1 & -L_\psi
    \\
    -L_\psi & 1+L_\psi^2
    \end{pmatrix}
    \boldsymbol w.
\end{equation*}
Both matrices have eigenvalues $s_-(L_\psi)^2$ and $s_+(L_\psi)^2$. Since $\|\boldsymbol w\|=1$, inequalities \eqref{eqn:pauli_sigma_upper} and \eqref{eqn:pauli_sigma_lower} imply \eqref{eqn:pauli_uniform_bounds}.

Finally, the product of the two singular values is $|\det\mathcal P(\bz)|=|\rho_z|^2$. The lower bound in \eqref{eqn:pauli_uniform_bounds}  implies \eqref{eqn:rho_lower_bound}.
}
\end{proof}

For a target location $\bx=(\rho_x,\theta_x,\phi_x)$ and a source location $\by=(\rho_y,\theta_y,\phi_y)$, we denote their difference by $\bx-\by=(\rho_{xy},\theta_{xy},\phi_{xy})$. It is straightforward to verify that
\begin{equation*}
    \rho_{xy} = \sqrt{\rho_x^2+\rho_y^2-2\rho_x\rho_y\cos(\alpha)}, 
\end{equation*}
where $\cos\alpha=\cos\theta_x\cos\theta_y+ \sin\theta_x\sin\theta_y\cos(\phi_x-\phi_y)$. The assumptions for the validity of~\eqref{eqn:addition_sph_hankel} are satisfied if $|\rho_y e^{\pm i\alpha}|/|\rho_x|<1$. Equivalently,
\begin{equation*}
    |\bx\cdot\by\pm i\sqrt{(\bx\times \by)\cdot (\bx\times \by)}|<|\rho_x|^2.
\end{equation*}
As in the 2-D case, the following lemma shows that sufficient separation of $\Re\bx$ and $\Re\by$ implies the convergence condition in~\eqref{eqn:addition_sph_hankel}.
\begin{lemma}
\label{lma:bound_for_3d}
{
Let $0<r<R$ and let $\bx,\by\in\VpsiThree$ with $\psi(0)=0$ satisfy
\begin{equation}
\label{eqn:real_separate_3d}
    \|\Re \bx\|\ge R, \qquad \|\Re \by\|\le r.
\end{equation}
If
\begin{equation}
\label{eqn:3dLipcond}
    L_\psi<\frac{R-r}{\sqrt{Rr}},
\end{equation}
then
\begin{equation}
\label{eqn:3dcond}
    |\rho_y e^{\pm i\alpha}|/|\rho_x|
    \le C_{L_\psi,r,R}<1.
\end{equation}}
\end{lemma}
\begin{proof}
{
By Lemma~\ref{lma:pointwise_distortion_3d}, we have
\begin{equation*}
    \sigma_{\min}(\mathcal P(\bx))
    \ge s_-(L_\psi)\|\Re \bx \|>0,
\end{equation*}
so $\mathcal P(\bx)$ is invertible. The identities in~\eqref{eqn:pauli_identities} therefore give
\begin{equation*}
    \mathcal P(\bx)^{-1}\mathcal P(\by)
    =\frac{(\bx\cdot\by)I+i\mathcal P(\bx\times\by)}{\bx\cdot\bx}
\end{equation*}
whose eigenvalues are $\rho_y e^{\pm i\alpha}/\rho_x$.
The modulus of an eigenvalue is bounded by the matrix operator norm. Hence,
\begin{equation}
\label{eqn:pauli_norm_reduction}
    \left|\frac{\rho_y e^{\pm i\alpha}}{\rho_x}\right|
    \le \|\mathcal P(\bx)^{-1}\|_2
        \|\mathcal P(\by)\|_2.
\end{equation}

We apply Lemma~\ref{lma:pointwise_distortion_3d} to $\bx$ and $\by$ in~\eqref{eqn:pauli_norm_reduction}, which gives
\begin{equation*}
    \left|\frac{\rho_y e^{\pm i\alpha}}{\rho_x}\right|
    \le \frac{s_+(L_\psi)}{s_-(L_\psi)}
       \frac{\|\Re\by\|}{\|\Re\bx\|}
    \le C_{L_\psi,r,R}.
\end{equation*}
Finally,~\eqref{eqn:convergence_factor_equivalence}, together with~\eqref{eqn:3dLipcond}, implies that $C_{L_\psi,r,R}<1$.
}
\end{proof}

\subsection{The 3-D complex-coordinate Laplace FMM}
\label{sec:fmm_lap3d}
For the 3-D Laplace equation, the analytically continued Green's function is
\begin{equation*}
    G(\bx,\by)=\frac{1}{4\pi \rho_{xy}},\qquad
    \bx,\by\in\VpsiThree,\quad \bx\ne\by.
\end{equation*}
We denote the source locations by
$
    \by_j=(\rho_{y_j},\theta_{y_j},\phi_{y_j})\in\VpsiThree,
     j=1,\ldots,N,
$
with corresponding densities $\{\sigma_j\}_{j=1}^N$. We consider the  potential
\begin{equation}
\label{eqn:lap_3d_pot}
    u(\bx) = \sum_{j=1}^N \frac{\sigma_j}{\rho_{xy_j}},\qquad
    \bx\in\VpsiThree\setminus\{\by_1,\ldots,\by_N\},
\end{equation}
where $\bx-\by_j=(\rho_{xy_j},\theta_{xy_j},\phi_{xy_j})$ for each $j=1,\ldots,N$.

\subsubsection{Multipole and local expansions}
\begin{lemma}
\label{lma:addition_3d_lap}
Under Assumption~\ref{assump:3d}, let $\bx=(\rho_x,\theta_x,\phi_x)$ and $\by=(\rho_y,\theta_y,\phi_y)$ lie in $\VpsiThree$, and let $\alpha\in\C$ satisfy $\cos\alpha=\cos\theta_x\cos\theta_y+\sin\theta_x\sin\theta_y\cos(\phi_x-\phi_y)$. Set $\rho_{xy}=\sqrt{\rho_x^2+\rho_y^2-2\rho_x\rho_y\cos\alpha}$. If $|\rho_y e^{\pm i\alpha }|<|\rho_x|$, then
\begin{equation}
\label{eqn:3dlaplace_addition}
    \frac{1}{\rho_{xy}}  =  \sum_{n=0}^\infty \frac{\rho_y^n}{\rho_x^{n+1}} P_n(\cos\alpha) .
\end{equation}
For $P\ge0$, the truncation error satisfies
\begin{equation}\label{eqn:trunc_err}
    \left \vert  \frac{1}{\rho_{xy}} - \sum_{n=0}^P
    \frac{\rho_y^n}{\rho_x^{n+1}} P_n(\cos\alpha) \right\vert
    \le \frac{1}{(c-1)|\rho_x|}
    \left(\frac{1}{c}\right)^P,
\end{equation}
where $c=\max\{|\rho_y e^{i\alpha}|/|\rho_x|, |\rho_y e^{-i\alpha}|/|\rho_x|\}^{-1}>1$.
\end{lemma}

\begin{proof}
Let $\mu=\rho_y/\rho_x$. Then $\rho_{xy}=\rho_x\sqrt{1-2\mu\cos\alpha+\mu^2}$. Since $1-2\mu\cos\alpha+\mu^2 =(1-\mu e^{i\alpha})(1-\mu e^{-i\alpha})$, we can use the Taylor expansion of the generating function for the Legendre polynomials. This gives
\begin{equation*}
    \frac{1}{\sqrt{1-2\mu\cos\alpha +\mu^2}} = \sum^{\infty}_{n=0} P_n(\cos\alpha) \mu^n,
\end{equation*}
which gives \eqref{eqn:3dlaplace_addition}. To prove the bound on the truncation error, we observe that for any $\alpha \in \C$ and any $n \in \mathbb{N}$,
     \begin{equation}\label{eqn:pn_bound}
         |P_n(\cos\alpha)|\le \max\{|e^{in\alpha}|,|e^{-in\alpha}|\},
     \end{equation}
which follows from the standard identity $P_n(\cos\alpha) = \frac{1}{\pi} \int_0^\pi (\cos\alpha+i\sin\alpha \cos\phi)^n\, \mathrm{d}\phi$. The error bound \eqref{eqn:trunc_err} now follows immediately from \eqref{eqn:pn_bound} and the triangle inequality.
\end{proof}

For a center $\bc\in\C^3$, we write $\bx-\bc=(\rho_{xc},\theta_{xc},\phi_{xc})$. Let $M=\{M_n^m|n\ge0,\ |m|\le n\}$ and $L=\{L_n^m|n\ge0,\ |m|\le n\}$. For every nonnegative integer $P$, we define the multipole and local expansions of~\eqref{eqn:lap_3d_pot} centered at $\bc$ by 
\begin{align}
\label{eqn:Lap3dformmp}
    u^{\mathrm{mp}}_{\bc,P}[M](\bx)
    &:=\sum_{n=0}^P\sum_{m=-n}^n
      \frac{M_n^m}{\rho_{xc}^{n+1}}
      Y_n^m(\theta_{xc},\phi_{xc}),\\
    u^{\mathrm{loc}}_{\bc,P}[L](\bx)
    &:=\sum_{n=0}^P\sum_{m=-n}^n
      L_n^m\rho_{xc}^n
      Y_n^m(\theta_{xc},\phi_{xc}).
\end{align}

\begin{theorem}[Multipole expansion]
\label{thm:form_multiple_lap3d}
Under Assumption~\ref{assump:3d}, let $0<r<R$ and assume that $L_\psi<\frac{R-r}{\sqrt{Rr}}$, so that $C_{L_\psi,r,R}<1$. Let the source locations $\{\by_j\}_{j=1}^N\subset\VpsiThree$ satisfy $\|\Re\by_j\|\le r$ for $j=1,\ldots,N$, and let $\bx\in\VpsiThree$ satisfy $\|\Re\bx\|\ge R$. Define
$M_n^m =  \sum_{j=1}^N  \rho_{y_j}^n
    Y_n^{-m}(\theta_{y_j}, \phi_{y_j})\sigma_j,$
    for $n\ge0, |m|\le n.$
Then $u(\bx)=\lim\limits_{P\to\infty} u^{\mathrm{mp}}_{\boldsymbol 0,P}[M](\bx)$, and, for $P\ge0$,
\begin{equation}
\label{eqn:lap3d_err_mp}
   \left|u(\bx)-u^{\mathrm{mp}}_{\boldsymbol 0,P}[M](\bx)\right|
   \le \frac{\|\vec{\sigma}\|_1\sqrt{\chi(L_\psi)}}{R}
   \frac{(C_{L_\psi,r,R})^{P+1}}{1-C_{L_\psi,r,R}}.
\end{equation}
\end{theorem}
\begin{proof}
The multipole expansion follows from Lemma~\ref{lma:addition_3d_lap} and \eqref{eqn:sphaddition}. By Lemma~\ref{lma:bound_for_3d}, the ratio of the geometric tail is bounded by $C_{L_\psi,r,R}<1$. Moreover, \eqref{eqn:rho_lower_bound} gives $|\rho_x|\ge R/\sqrt{\chi(L_\psi)}$. Therefore, the truncation error is bounded by
\begin{equation*}
    \frac{\|\vec{\sigma}\|_1\sqrt{\chi(L_\psi)}}{R}
    \sum_{n=P+1}^\infty C_{L_\psi,r,R}^{n}
    =\frac{\|\vec{\sigma}\|_1\sqrt{\chi(L_\psi)}}{R}
    \frac{(C_{L_\psi,r,R})^{P+1}}{1-C_{L_\psi,r,R}}.
\end{equation*}
\end{proof}

The derivation of the local expansion and the related truncation error follow in a similar manner. The result is summarized in the theorem below.
\begin{theorem}[Local expansion]
\label{thm:form_local_lap3d}
Under Assumption~\ref{assump:3d}, let $0<r<R$ and assume that $L_\psi<\frac{R-r}{\sqrt{Rr}}$, so that $C_{L_\psi,r,R}<1$. Let the source locations $\{\by_j\}_{j=1}^N\subset\VpsiThree$ satisfy $\|\Re\by_j\|\ge R$ for $j=1,\ldots,N$, and let $\bx\in\VpsiThree$ satisfy $\|\Re\bx\|\le r$. Define $L_n^m=\sum_{j=1}^N\rho_{y_j}^{-n-1} Y_n^{-m}(\theta_{y_j},\phi_{y_j})\sigma_j$ for $n\ge0$ and $|m|\le n$. Then $u(\bx)=\lim\limits_{P\to\infty}u^{\mathrm{loc}}_{\boldsymbol 0,P}[L](\bx)$, and, for $P\ge0$,
\begin{equation*}
\left|u(\bx)-u^{\mathrm{loc}}_{\boldsymbol 0,P}[L](\bx)\right|
\le \frac{\|\vec{\sigma}\|_1\sqrt{\chi(L_\psi)}}{R}
\frac{(C_{L_\psi,r,R})^{P+1}}{1-C_{L_\psi,r,R}}.
\end{equation*}
\end{theorem}

\subsubsection{Translation operators}

With Lemmas~\ref{lma:harmonic} and~\ref{lma:spharmonics}, it is straightforward to show that the following addition formulas hold. The proofs follow the same reasoning as Theorems 3.5.1, 3.5.2, and 3.5.3 in \cite{10.7551/mitpress/5750.001.0001} and are omitted.

\begin{lemma}[Addition formula]
\label{lma:addtion_lap3d}
Under Assumption~\ref{assump:3d}, let $\bx_0=(\rho_0,\theta_0,\phi_0)\in\VpsiThree$ and $\bx=(\rho_x,\theta_x,\phi_x)\in\VpsiThree$, and write $\bx-\bx_0=(\rho',\theta',\phi')$. Let $j\ge0$ and $|k|\le j$. The following addition formulas hold.
\begin{enumerate}
    \item If $0<\|\Re \bx_0\|<\|\Re \bx\|$ and
    $
        L_\psi<\frac{\|\Re\bx\|-\|\Re\bx_0\|}
        {\sqrt{\|\Re\bx\|\,\|\Re\bx_0\|}},
    $
    then
\begin{equation}\label{eqn:firstaddtion_lap3d}
    \frac{Y_j^k(\theta',\phi')}{(\rho')^{j+1}}
    =\sum_{n=0}^\infty \sum_{m=-n}^n \frac{J_m^{k} A_n^m A_j^k  }{A_{n+j}^{m+k}} \cdot \frac{\rho_0^n}{\rho^{n+j+1}_x} Y_n^{-m}(\theta_0,\phi_0)Y_{n+j}^{m+k}(\theta_x,\phi_x),
\end{equation}
where the quantities $A_n^m$ are defined in Lemma~\ref{lma:spharmonics} and
\begin{equation}
\label{eqn:first_J}
    J_m^k=\begin{cases}
        (-1)^{\min(|k|,|m|)} & \text{if } m\cdot k<0; \\
        1 & \text{otherwise.}
    \end{cases}
\end{equation}
\item If $0<\|\Re \bx\|<\|\Re \bx_0\|$ and
$
    L_\psi<\frac{\|\Re\bx_0\|-\|\Re\bx\|}
    {\sqrt{\|\Re\bx_0\|\,\|\Re\bx\|}},
$
then
\begin{equation}\label{eqn:secondaddtion_lap3d}
    \frac{Y_j^k(\theta',\phi')}{(\rho')^{j+1}}
    =\sum_{n=0}^\infty \sum_{m=-n}^n \frac{\widehat{J}_m^{k}(j) A_n^m A_j^k  }{A_{n+j}^{m-k}} \cdot \frac{\rho_x^n}{\rho_0^{n+j+1}} Y_{n+j}^{k-m}(\theta_0,\phi_0)Y_{n}^{m}(\theta_x,\phi_x),
\end{equation}
where
\begin{equation}
\label{eqn:second_J}
    \widehat{J}_m^k(j)=\begin{cases}
        (-1)^j(-1)^{\min(|k|,|m|)} & \text{if } m\cdot k>0; \\
        (-1)^j & \text{otherwise.}
    \end{cases}
\end{equation}
\item We have
\begin{equation}\label{eqn:thirdaddtion_lap3d}
    Y_j^k(\theta',\phi') (\rho')^j = \sum_{n=0}^j
    \sum_{\substack{m=-n\\ |k-m|\le j-n}}^n
    \frac{\widetilde{J}_{n,m}^k A_n^m A_{j-n}^{k-m} }{A_j^k} \cdot Y_n^m(\theta_0,\phi_0) Y_{j-n}^{k-m}(\theta_x,\phi_x) \rho_0^n\rho_x^{j-n},
\end{equation}
where
\begin{equation}
\label{eqn:third_J}
    \widetilde{J}_{n,m}^{k} = 
    \begin{cases}
        (-1)^n (-1)^m & \text{ if } m\cdot k\le0; \\ 
        (-1)^n (-1)^{k-m} & \text{ if } m\cdot k>0 \text{ and } |k|<|m|;\\
        (-1)^n & \text{otherwise}.
    \end{cases}
\end{equation}
\end{enumerate}
\end{lemma}

\begin{theorem}[M2M]
Under Assumption~\ref{assump:3d}, let $r>0$ and $\bx_0,\by_1,\ldots,\by_N\in\VpsiThree$ satisfy $\|\Re\by_j-\Re\bx_0\|<r$. Let $R>r+\|\Re\bx_0\|$ and assume that $L_\psi<\frac{R-r-\|\Re\bx_0\|}{\sqrt{R(r+\|\Re\bx_0\|)}}$, so that $C_{L_\psi,r+\|\Re\bx_0\|,R}<1$. For densities $\{\sigma_j\}_{j=1}^N$, we define $M_n^m$ by the formula in Theorem~\ref{thm:form_multiple_lap3d}, for $n\ge0$ and $|m|\le n$, using the spherical coordinates of $\by_j-\bx_0$, where $\bx_0=(\rho_0,\theta_0,\phi_0)$. The translated multipole expansion coefficients centered at zero is
\begin{align}
\label{eqn:L3dM2M}
    \tilde M_j^k = \sum_{n=0}^j
    \sum_{\substack{m=-n\\ |k-m|\le j-n}}^{n}
    \frac{ J_m^{k-m} A_{n}^{m} A_{j-n}^{k-m}
    \rho_0^{n}Y_{n}^{-m}(\theta_0,\phi_0)}{A_{j}^k}
    M_{j-n}^{k-m}. 
\end{align}
Here, $j\ge0$, $|k|\le j$, and $J_m^{k-m}$ is defined by~\eqref{eqn:first_J}. For $\bx\in\VpsiThree$ with $\|\Re\bx\|>R$, $u(\bx)=\lim\limits_{P\to\infty}u^{\mathrm{mp}}_{\boldsymbol 0,P}[\tilde M](\bx)$. For $P\ge1$,
\begin{align}
   \left|u(\bx)-u^{\mathrm{mp}}_{\boldsymbol 0,P}[\tilde M](\bx)\right|
    \le \frac{\|\vec{\sigma}\|_1\sqrt{\chi(L_\psi)}}{R}
    \frac{(C_{L_\psi,r+\|\Re\bx_0\|,R})^{P+1}}
    {1-C_{L_\psi,r+\|\Re\bx_0\|,R}}.
\label{eqn:M2M_lap3d_error}
\end{align}
\end{theorem}
\begin{proof}
We verify the translated multipole expansion coefficients using the definition in Theorem~\ref{thm:form_multiple_lap3d}. For each fixed $j\ge0$ and $|k|\le j$, we apply the finite addition formula~\eqref{eqn:thirdaddtion_lap3d} to $\by_\ell=(\by_\ell-\bx_0)-(-\bx_0)$. Using $A_n^{-m}=A_n^m$, and $(-1)^n\widetilde{J}_{n,-m}^{-k}=J_m^{k-m}$, we obtain
\begin{equation}
\label{eqn:m2m_sph_harmonic_expand}
    \rho_{y_\ell}^jY_j^{-k}(\theta_{y_\ell},\phi_{y_\ell})  =\sum_{n=0}^j\sum_{\substack{m=-n\\ |k-m|\le j-n}}^n  \frac{J_m^{k-m}A_n^mA_{j-n}^{k-m}}{A_j^k}\rho_0^nY_n^{-m}(\theta_0,\phi_0)\rho_{y_\ell x_0}^{j-n}Y_{j-n}^{m-k}(\theta_{y_\ell x_0},\phi_{y_\ell x_0}).
\end{equation}
We multiply~\eqref{eqn:m2m_sph_harmonic_expand} by $\sigma_\ell$ and sum over the sources to get~\eqref{eqn:L3dM2M}, using the definition of the multipole expansion coefficients $M_{j-n}^{k-m}$ centered at $\bx_0$. This  gives  the multipole expansion coefficients in Theorem~\ref{thm:form_multiple_lap3d} with expansion center at the origin. 

Finally, since $\|\Re\by_\ell\|<r+\|\Re\bx_0\|$, we apply Theorem~\ref{thm:form_multiple_lap3d} with $r$ replaced by $r+\|\Re\bx_0\|$ to get the convergence and error bound.
\end{proof}

\begin{theorem}[M2L]
\label{thm:lap3d_m2l}
Under Assumption~\ref{assump:3d}, let $0<r<R$, $\bx_0=(\rho_0,\theta_0,\phi_0)\in\VpsiThree$, and $\by_1,\ldots,\by_N\in\VpsiThree$ satisfy $\|\Re\by_j-\Re\bx_0\|<r$. Assume that $\|\Re\bx_0\|>R+r$ and $L_\psi<\frac{R-r}{\sqrt{Rr}}$, so that $C_{L_\psi,r,R}<1$. For densities $\{\sigma_j\}_{j=1}^N$, we define $M_n^m$ by the formula in Theorem~\ref{thm:form_multiple_lap3d}, for $n\ge0$ and $|m|\le n$, using the spherical coordinates of $\by_j-\bx_0$. The  coefficients of the converted local expansion centered at zero is
\begin{align}
\label{eqn:L3dM2L}
     L_j^k = \sum_{n=0}^\infty \sum_{m=-n}^{n}
     \frac{ \widehat{J}_k^m(n) A_{n}^{m} A_{j}^{k}
     Y_{j+n}^{m-k}(\theta_0,\phi_0)}
     {A_{j+n}^{m-k} \rho_0^{j+n+1}} M_{n}^{m}. 
\end{align}
Here, $j\ge0$, $|k|\le j$, and $\widehat{J}_k^m(n)$ is defined by~\eqref{eqn:second_J}. For $\bx\in\VpsiThree$ with $\|\Re\bx\|<r$, $u(\bx)=\lim\limits_{P\to\infty}u^{\mathrm{loc}}_{\boldsymbol 0,P}[L](\bx)$. For $P\ge1$,
\begin{align*}
\left|u(\bx)-u^{\mathrm{loc}}_{\boldsymbol 0,P}[L](\bx)\right|
\le \frac{\|\vec{\sigma}\|_1\sqrt{\chi(L_\psi)}}{R}
\frac{(C_{L_\psi,r,R})^{P+1}}{1-C_{L_\psi,r,R}}.
\end{align*}
\end{theorem}

\begin{proof}
For each fixed $j\ge0$ and $|k|\le j$, we verify the local expansion coefficient in Theorem~\ref{thm:form_local_lap3d}. We apply Lemma~\ref{lma:pointwise_distortion_3d} at zero and $\bx_0$, which gives
\begin{equation}
\label{eqn:point_estimate_m2l_3d}
    \|\mathcal P(\bx_0)^{-1}\|_2\,
    \|\mathcal P(\bx_0-\by_\ell)\|_2
    \le\chi(L_\psi)\frac{\|\Re\by_\ell-\Re\bx_0\|}{\|\Re\bx_0\|}
    <C_{L_\psi,r,R}<1.
\end{equation}
By the same argument as in Lemma~\ref{lma:bound_for_3d}, this estimate gives convergence of the first addition formula~\eqref{eqn:firstaddtion_lap3d} applied to $\by_\ell=\bx_0-(\bx_0-\by_\ell)$. Using $A_j^{-k}=A_j^k$, and $(-1)^nJ_m^{-k}=\widehat{J}_k^m(n)$, we obtain
\begin{equation}
\label{eqn:m2l_sph_harmonic_expand}
    \frac{Y_j^{-k}(\theta_{y_\ell},\phi_{y_\ell})}{\rho_{y_\ell}^{j+1}}
    =\sum_{n=0}^\infty\sum_{m=-n}^n
    \frac{\widehat{J}_k^m(n)A_n^mA_j^kY_{j+n}^{m-k}(\theta_0,\phi_0)}
    {A_{j+n}^{m-k}\rho_0^{j+n+1}}\rho_{y_\ell x_0}^n
    Y_n^{-m}(\theta_{y_\ell x_0},\phi_{y_\ell x_0}).
\end{equation}
We multiply~\eqref{eqn:m2l_sph_harmonic_expand} by $\sigma_\ell$ and sum over the sources to get~\eqref{eqn:L3dM2L}, using the definition of the multipole expansion coefficients $M_n^m$ centered at $\bx_0$. 

Finally, since $\|\Re\by_\ell\|>R$ and $\|\Re\bx\|<r$, we apply Theorem~\ref{thm:form_local_lap3d} to get the  error bound.
\end{proof}

The following theorem provides a method for shifting the center of a truncated local expansion. The proof is an immediate consequence of the addition formula \eqref{eqn:thirdaddtion_lap3d}.
\begin{theorem}[L2L]
Let $P\ge0$, let $\bx_0=(\rho_0,\theta_0,\phi_0)\in\C^3$, and let $L=\{L_n^m:0\le n\le P,\ |m|\le n\}$. Then
\begin{equation*}
    u^{\mathrm{loc}}_{\bx_0,P}[L](\bx)
    =u^{\mathrm{loc}}_{\boldsymbol 0,P}[\tilde L](\bx),
\end{equation*}
where the translated local expansion coefficients
\begin{align}
\label{eqn:L3dL2L}
    \tilde L_j^k =  \sum_{n=j}^P
    \sum_{\substack{m=-n\\ |m-k|\le n-j}}^n
    \frac{\widetilde{J}_{n-j,m-k}^{m} A_{n-j}^{m-k}A_j^k
    Y_{n-j}^{m-k}(\theta_0,\phi_0)\rho_0^{n-j}}
    {A_n^m} L_n^m.
\end{align}
Here, $0\le j\le P$, $|k|\le j$, and $\widetilde{J}_{n-j,m-k}^{m}$ is defined by~\eqref{eqn:third_J}.
\end{theorem}

\subsubsection{Translation operators along the \texorpdfstring{$z$}{z}-direction}
For coefficient arrays restricted to degree at most $P$, a direct implementation of the translation operators costs $\cO(P^4)$. If the offset $\bx_0$ lies on the $z$-axis, the cost is reduced to $\cO(P^3)$ by the following formulas.
\begin{lemma}
\label{lma:zshift}
For $\bx_0=(\rho_0,\theta_0,0)$, with $\rho_0\in\C$ and $\theta_0\in\{0,\pi\}$, the exact M2M and L2L maps restricted to degree $P$, together with the degree-$P$ M2L input-output projection, reduce, for $0\le j\le P$ and $|k|\le j$, to
\begin{align}
\label{eqn:L3dM2M_zshift}
\tilde M_j^k & =  \sum_{n=0}^{j-|k|} \frac{A_n^0A_{j-n}^k Y^0_n(\theta_0,0)\rho_0^n }{A_{j}^k } M_{j-n}^k,\\
\label{eqn:L3dM2L_zshift}
L_{j,P}^k &= \sum_{n=|k|}^{P}
\frac{(-1)^{n+|k|} A_n^kA_j^k
Y^0_{j+n}(\theta_0,0)}{A_{j+n}^0 \rho_0^{j+n+1}}M_n^k,\\
\label{eqn:L3dL2L_zshift}
\tilde L_j^k &= \sum_{n=j}^P
\frac{(-1)^{n-j}A_{n-j}^0A_j^k
Y^0_{n-j}(\theta_0,0) \rho_0^{n-j}}{A_n^k} L_n^k. 
\end{align}
Here, $\rho_0\ne0$ in the M2L formula. The M2M and L2L formulas are the exact restrictions of \eqref{eqn:L3dM2M} and~\eqref{eqn:L3dL2L}. The M2L formula gives the degree-$P$ local expansion coefficients $L_{j,P}^k$ generated by a degree-$P$ input multipole; it is therefore a finite projection of~\eqref{eqn:L3dM2L}.
\end{lemma}

\begin{proof}
When $\theta_0=0$ or $\pi$, $\cos\theta_0=\pm1$. The factor $(1-x^2)^{|m|/2}$ makes every spherical harmonic of nonzero order vanish on the $x_3$-axis. Substituting the surviving zero-order terms into~\eqref{eqn:L3dM2M} and~\eqref{eqn:L3dL2L} gives the M2M and L2L formulas. Substitution into~\eqref{eqn:L3dM2L}, followed by truncation of the input at degree $P$, gives the M2L formula.
\end{proof}

Based on Lemma~\ref{lma:zshift}, a more efficient procedure, known as the ``\textit{point-and-shoot}'' method, was proposed in \cite{white96} for the 3-D Laplace FMM with real coordinates. It has an overall cost of $\cO(P^3)$ and is carried out in the following steps:
\begin{enumerate}
    \item[(1)]  Rotate the multipole or local expansion so that the offset between centers lies on the $z$-axis. This rotation involves a rotation around the $z$-axis, which costs $\cO(P^2)$, followed by a rotation around the $y$-axis, which costs $\cO(P^3)$.
    \item[(2)] Apply the translation operators in~\eqref{eqn:L3dM2M_zshift}--\eqref{eqn:L3dL2L_zshift} along the $z$-axis, which requires $\cO(P^3)$ work.
    \item[(3)] Reverse the  rotation steps in (1), which costs $\cO(P^2) + \cO(P^3)$.
\end{enumerate}
\begin{definition}[Complex special orthogonal group]
The 3-D complex special orthogonal group $\mathcal{SO}(3,\C)$ is defined as 
\begin{equation*}
    \mathcal{SO}(3,\C)=\{R\in\C^{3\times 3} | RR^T = I_{3\times 3}, \,\, \det(R) = 1\}.   
\end{equation*}
Rotating a complex function $f:\C^3 \to \C$ by a complex rotation $R \in \mathcal{SO}(3,\C)$ is defined by
\[
    (R \circ f)(\bx) = f(R^T \bx),\qquad
    R\in\mathcal{SO}(3,\C).
\]
\end{definition}

\subsubsection{Rotation around the \texorpdfstring{$z$}{z}-axis and the \texorpdfstring{$y$}{y}-axis}
\label{sec:rotation_operator}
To rotate around the $z$-axis, let the target center be $\bc \in \C^3$. If $c_1=c_2=0$, this step is unnecessary. Otherwise, for offsets between centers satisfying Assumption~\ref{assump:3d}, we have $c_1^2+c_2^2\ne0$, and we may choose $\beta\in\C$ so that
\begin{equation*}
    \cos\beta=\frac{c_1}{\sqrt{c_1^2+c_2^2}},\qquad
    \sin\beta=\frac{c_2}{\sqrt{c_1^2+c_2^2}},
\end{equation*}
and apply the following matrix in $\mathcal{SO}(3,\C)$
\begin{equation*}
    R_z(-\beta) = \begin{pmatrix}
        \cos\beta & \sin\beta & 0 \\ 
        -\sin\beta & \cos\beta & 0 \\
       0 & 0 & 1
    \end{pmatrix}
\end{equation*}
to the coordinates. Since this rotation affects only the azimuthal component, the coefficient transformation is diagonal:
\begin{equation}
    \tilde M_n^m = M_n^m e^{im\beta}, \qquad
    \tilde L_n^m = L_n^m e^{im\beta},
\end{equation}
for all $0\le n\le P$ and $|m|\le n$. The rotated center is $R_z(-\beta)\bc=(\sqrt{c_1^2+c_2^2},0,c_3)^T$.

After the $z$-rotation, write the center as $\bc=(c_1,0,c_3)^T$. The rotation around the $y$-axis uses
\begin{equation*}
R_y(\alpha) = \begin{pmatrix}
    \cos\alpha & 0 & -\sin\alpha \\ 
    0 & 1 & 0 \\ 
    \sin\alpha & 0 & \cos\alpha
\end{pmatrix},
\end{equation*}
where
$
    \cos\alpha=\frac{c_3}{\sqrt{c_1^2+c_3^2}},
    \sin\alpha=\frac{c_1}{\sqrt{c_1^2+c_3^2}}.
$
Thus, $R_y(\alpha)\bc=(0,0,\sqrt{c_1^2+c_3^2})^T$. For real angles, the degree-$n$ coefficients transform by the small Wigner-$d$ matrix. Its entries are finite combinations of powers of $\sin(\alpha/2)$ and $\cos(\alpha/2)$ and hence extend analytically to complex $\alpha$~\cite{chirikjian2016harmonic}. Rather than evaluating this matrix explicitly, we use the numerical rotation method of \cite{GIMBUTAS20095621}. Let $\{\tilde M_n^m\}$ denote the coefficients after rotation. For $\bx,\by\in\C^3$ satisfying $\by=R_y(\alpha)\bx$, they are defined by
\begin{equation}
\label{eqn:y_rotation_target}
   \sum_{m=-n}^n M_n^m \bar P_n^m(\cos\theta_x)e^{im \phi_x} =  \sum_{m=-n}^n  \tilde M_n^m \bar P_n^m(\cos \theta_y) e^{im\phi_y},
\end{equation}
for each $n=0,\ldots,P$. Here, the radial functions can be omitted because they are invariant under rotation. We define
\begin{equation}
\label{eqn:tar_fun_rot}
    F_n(\theta_y,\phi_y)=\sum_{m=-n}^n  \tilde M_n^m \bar P_n^m(\cos \theta_y) e^{im\phi_y}. 
\end{equation}
The key observation is that if we choose a fixed $\theta_y$, then for each $n$, the function $F_n(\theta_y,\cdot)$ in~\eqref{eqn:tar_fun_rot} is expressed in a Fourier basis. We can therefore recover its Fourier coefficients efficiently using the FFT. To avoid the zeros of $\bar P_n^m(0)$, we also use the quantities $D_n^m:=\left.\frac{d}{d\theta} \bar P_n^m(\cos\theta)\right|_{\theta=\pi/2}$. To obtain the coefficients $\{\tilde M_n^m\}$, for each $0\le n\le P$ and $-P\le j\le P$, we define $\phi_j=\frac{2\pi j}{2P+1}$ and  evaluate
\begin{align}
    f_{n,j} & = F_n(\frac{\pi}{2},\phi_j)=\sum_{m=-n}^n \tilde M_n^m \bar P_n^m(0) e^{i m \phi_j}, \\  g_{n,j}& =\frac{\partial}{\partial \theta} F_n(\frac{\pi}{2},\phi_j)=\sum_{m=-n}^n \tilde M_n^m D_n^m e^{i m \phi_j}.
\end{align}
Although we do not know $\{\tilde M_n^m\}$, these values can be obtained from the expansion in the $(\theta_x, \phi_x)$ coordinates using \eqref{eqn:y_rotation_target} and the known coefficients $\{M_n^m\}$, with a total cost of $\mathcal{O}(P^3)$. Using the orthogonality of the Fourier basis $\{e^{im\phi}\}$ on the grid $\{\phi_j\}_{j=-P}^P$, we can compute
\begin{equation}
\label{eqn:mp_y_rot}
    \tilde M_n^m = \frac{1}{2P+1}  \frac{\sum_{j=-P}^P \left( f_{n,j} \bar P_n^m(0) +g_{n,j}D_n^m \right) e^{-im \phi_j} }{(\bar P_n^m(0))^2+(D_n^m)^2} . 
\end{equation}
The use of both $F_n$ and its derivative prevents the denominator in~\eqref{eqn:mp_y_rot} from being zero. All rotated coefficients can be computed efficiently using $\mathcal O(P)$ one-dimensional FFTs on arrays of length $2P+1$, at a cost of $\mathcal O(P^2\log P)$. Thus, the total cost of rotating $\{M_n^m\}$ to $\{\tilde M_n^m\}$ is $\mathcal{O}(P^3)$.

\subsection{The 3-D complex-coordinate Helmholtz FMM}
\label{sec:fmm_helm3d}
For the 3-D Helmholtz equation with $\kappa>0$, the complexified Green's function is
\begin{equation*}
G_\kappa(\bx,\by) = \frac{i\kappa}{4\pi} h_0^{(1)}(\kappa \rho_{xy})
= \frac{e^{i\kappa \rho_{xy}}}{4\pi \rho_{xy}},\qquad
\bx,\by\in\VpsiThree,\quad \bx\ne\by.
\end{equation*}
We focus on evaluating the following potential:
\begin{equation}
\label{eqn:helm_3d_pot}
    u(\bx)=\sum_{j=1}^N \frac{e^{i\kappa \rho_{xy_j}}}{\rho_{xy_j}}\sigma_j
    =i\kappa\sum_{j=1}^N h_0^{(1)}(\kappa\rho_{xy_j})\sigma_j,
    \qquad \bx\in\VpsiThree\setminus\{\by_1,\ldots,\by_N\},
\end{equation}
where the source locations $\{\by_j=(\rho_{y_j},\theta_{y_j},\phi_{y_j})\}_{j=1}^N$ lie in $\VpsiThree$, with associated densities $\{\sigma_j\}_{j=1}^N$. Here, $\bx-\by_j=(\rho_{xy_j},\theta_{xy_j},\phi_{xy_j})$ for each $j=1,\ldots,N$.

\subsubsection{Multipole and local expansions}

For a center $\bc\in\C^3$, we write $\bx-\bc=(\rho_{xc},\theta_{xc},\phi_{xc})$. Let $M=\{M_n^m:n\ge0,\ |m|\le n\}$ and $L=\{L_n^m:n\ge0,\ |m|\le n\}$ be the multipole and local expansion coefficients. For every nonnegative integer $P$, we define the multipole and local expansions of~\eqref{eqn:helm_3d_pot} centered at $\bc$ by 
\begin{align*}
    u^{\mathrm{mp}}_{\bc,P}[M](\bx)
    &:=\sum_{n=0}^P\sum_{m=-n}^n M_n^m
      h_n^{(1)}(\kappa\rho_{xc})
      Y_n^m(\theta_{xc},\phi_{xc}),\\
    u^{\mathrm{loc}}_{\bc,P}[L](\bx)
    &:=\sum_{n=0}^P\sum_{m=-n}^n L_n^m
      j_n(\kappa\rho_{xc})
      Y_n^m(\theta_{xc},\phi_{xc}).
\end{align*}

\begin{theorem}[Multipole expansion]
Under Assumption~\ref{assump:3d}, let $0<r<R$ and assume that $L_\psi<\frac{R-r}{\sqrt{Rr}}$, so that $C_{L_\psi,r,R}<1$. Let the source locations $\{\by_j\}_{j=1}^N\subset\VpsiThree$, with associated densities $\{\sigma_j\}_{j=1}^N$, satisfy $\|\Re\by_j\|\le r$ for $j=1,\ldots,N$, and let $\bx\in\VpsiThree$ satisfy $\|\Re\bx\|\ge R$. We define the following expansion coefficients
\begin{equation*}
    M_n^m =i\kappa(2n+1) \sum_{j=1}^N  j_n(\kappa \rho_{y_j}) Y_n^{-m}(\theta_{y_j}, \phi_{y_j})\sigma_j,
    \qquad n\ge0,\quad |m|\le n.
\end{equation*}
Then $u(\bx)=\lim\limits_{P\to\infty} u^{\mathrm{mp}}_{\boldsymbol 0,P}[M](\bx)$. For sufficiently large $P$,
\begin{equation}
\label{eqn:h3d_mp_err}
    \left|u(\bx)-u^{\mathrm{mp}}_{\boldsymbol 0,P}[M](\bx)\right|
    \lesssim \frac{\|\vec{\sigma}\|_1\sqrt{\chi(L_\psi)}}{R}
    \frac{(C_{L_\psi,r,R})^{P+1}}{1-C_{L_\psi,r,R}}.
\end{equation}
\end{theorem}

\begin{proof}
The multipole expansion formula arises from the addition formula for spherical Hankel function~\eqref{eqn:addition_sph_hankel} combined with the spherical-harmonic addition formula~\eqref{eqn:sphaddition}. By Lemma~\ref{lma:bound_for_3d}, the ratio of the resulting geometric tail is bounded by $C_{L_\psi,r,R}<1$. The error bound then follows by summing this tail and using \eqref{eqn:rho_lower_bound}.
\end{proof}

We can prove a similar result regarding expressing the potential $u$ as a local expansion, which is summarized below.
\begin{theorem}[Local expansion]
Under Assumption~\ref{assump:3d}, let $0<r<R$ and assume that $L_\psi<\frac{R-r}{\sqrt{Rr}}$, so that $C_{L_\psi,r,R}<1$. Let the source locations $\{\by_j\}_{j=1}^N\subset\VpsiThree$, with associated densities $\{\sigma_j\}_{j=1}^N$, satisfy $\|\Re\by_j\|\ge R$ for $j=1,\ldots,N$, and let $\bx\in\VpsiThree$ satisfy $\|\Re\bx\|\le r$. We define the following expansion coefficients
\begin{equation*}
    L_n^m =i\kappa(2n+1)\sum_{j=1}^N h_n^{(1)}(\kappa \rho_{y_j}) Y_n^{-m}(\theta_{y_j},\phi_{y_j})\sigma_j,
    \qquad n\ge0,\quad |m|\le n.
\end{equation*}
Then $u(\bx)=\lim\limits_{P\to\infty} u^{\mathrm{loc}}_{\boldsymbol 0,P}[L](\bx)$. For sufficiently large $P$,
\begin{equation*}
    \left|u(\bx)-u^{\mathrm{loc}}_{\boldsymbol 0,P}[L](\bx)\right|
    \lesssim \frac{\|\vec{\sigma}\|_1\sqrt{\chi(L_\psi)}}{R}
    \frac{(C_{L_\psi,r,R})^{P+1}}{1-C_{L_\psi,r,R}}.
\end{equation*}
\end{theorem}

\subsubsection{Translation operators}
For the 3-D Helmholtz FMM, translation operators can be derived from the spherical-wave addition theorems~\cite{wideband3d}. A direct implementation costs $\mathcal{O}(P^4)$. As in the Laplace case, we use the ``point-and-shoot'' method to reduce the cost to $\mathcal{O}(P^3)$. Since rotation is independent of the radial part of an expansion, the operators in Section~\ref{sec:rotation_operator} apply here as well. It suffices to describe translation along the $z$-axis. We use M2L as an illustrative example; M2M and L2L are implemented similarly.

Without loss of generality, assume the multipole expansion is centered at the origin and the new center $\bc = (0, 0, c_3)$ lies on the $z$-axis, where $c_3 \in \C$. In this subsection, we absorb the factor $\sqrt{4\pi/(2n+1)}$ in~\eqref{eqn:Y_barP_relation} into the degree-$n$ multipole and local expansion coefficients and continue to denote the rescaled coefficients by $M_n^m$ and $L_n^m$. We seek a truncated local expansion such that, for $\by=\bx-\bc$, equivalently $\bx=\by+\bc$,
\begin{equation}
\label{eqn:noncoaxial}
   \sum_{n=0}^P\sum_{m=-n}^n M_n^m h_n^{(1)}(\kappa \rho_x) \bar P_n^m(\cos\theta_x) e^{im\phi_x}\approx  \sum_{n=0}^P \sum_{m=-n}^n L_n^m j_n(\kappa \rho_y)  \bar P_n^m(\cos \theta_y) e^{im\phi_y}. 
\end{equation}
Under this translation, we always have $\phi_x = \phi_y$. By the orthogonality of the Fourier basis $\{e^{im\phi}\}$, the problem reduces to finding $\{L_n^m\}$ such that
\begin{equation}
\label{eqn:coaxial}
    \sum_{n\ge |m|}^P M_n^m h_n^{(1)}(\kappa \rho_x) \bar P_n^m(\cos\theta_x) \approx  \sum_{n\ge |m|}^P  L_n^m j_n(\kappa \rho_y)  \bar P_n^m(\cos \theta_y), \quad  \text{ for any} -P\le m \le P.
\end{equation}
For $P\ge1$, let $(w_j,x_j)$, $j=1,\ldots,N_{\rm quad}$, denote the Gauss--Legendre rule with $N_{\rm quad}=\lfloor2.5P\rfloor$. For each $-P\le m\le P$, we define
\begin{equation}
\label{eqn:target_fun_M2L}
    F_{m}(\rho,\theta) =  \sum_{n\ge |m|}^P  L_n^m j_n(\kappa \rho)  \bar P_n^m(\cos \theta). 
\end{equation}
The observation is that if we choose a fixed $\rho$, then for each $m$, the function $F_m(\rho,\cdot)$ in~\eqref{eqn:target_fun_M2L} is expanded in the normalized associated Legendre polynomial basis. The coefficients (multiplied by the constant $j_n(\kappa \rho)$) can be recovered by projecting $F_m(\rho,\cdot)$ onto each normalized associated Legendre polynomial. We evaluate the following matrices
\begin{align}
\label{eqn:fun_sample}
    f_{m,j} & = F_{m}(\rho,\arccos(x_j))=\sum_{n\ge |m|}^P  L_n^m j_n(\kappa \rho)  \bar P_n^m(x_j) , \\
\label{eqn:der_sample}
    g_{m,j} & = \frac{\partial}{\partial \rho}F_{m}(\rho,\arccos(x_j))=\sum_{n\ge |m|}^P  L_n^m \kappa j_n'(\kappa \rho)  \bar P_n^m(x_j),
\end{align}
for some $\rho>0$, $-P\le m\le P$, and $j=1,2,\ldots,N_{\rm quad}$. The sampling sphere is chosen to lie in a region where the input multipole expansion converges and the translated field is analytic. Although we do not know $\{L_n^m\}$ in~\eqref{eqn:target_fun_M2L}, we can evaluate $F_m$ and its derivative from the known multipole expansion coefficients $\{M_n^m\}$ in the $(\rho_x,\theta_x)$ coordinates via~\eqref{eqn:coaxial}. Projecting these sampled values onto the normalized associated Legendre basis gives the degree-$P$ local approximation
\begin{equation}
\label{eqn:helm3d_M2L_z_axis}
    L_n^m =2\pi  \frac{\sum_{j=1}^{N_{\rm quad}} w_j \bar P_n^m(x_j) \left (j_n(\kappa \rho) f_{m,j}+\kappa j_n'(\kappa \rho) g_{m,j}\right)}{j_n(\kappa \rho)^2+\kappa^2 j_n'(\kappa \rho)^2},
\end{equation}
for $0\le n\le P$ and $-n\le m\le n$. Since $\kappa\rho$ is real, the use of both $F_m$ and its derivative prevents the denominator in~\eqref{eqn:helm3d_M2L_z_axis} from vanishing: $j_n$ and $j_n'$ cannot vanish simultaneously. When the samples are evaluated from a truncated multipole expansion,~\eqref{eqn:helm3d_M2L_z_axis} reconstructs the degree-$P$ projection of those approximate data rather than the exact infinite local expansion. Evaluating the samples and evaluating the projected coefficients each cost $\mathcal{O}(P^3)$. The quadrature order is chosen to resolve products of the angular basis functions through degree $P$. However, \eqref{eqn:helm3d_M2L_z_axis} can suffer from catastrophic errors when $\kappa$ is large. See \Cref{sec:catastrophic} for a detailed discussion of this issue.

\begin{remark}
{
In contrast to the 2-D case, establishing truncation error estimates for the 3-D Helmholtz translation operators based on the point-and-shoot method remains an open problem even for the real-coordinate case.  Although truncation-error estimates are available for analytic 3-D Helmholtz translation operators in real coordinates~\cite{https://doi.org/10.1002/num.23148}, they do not directly apply to the point-and-shoot method used here. Deriving such estimates and extending them to complex coordinates would require a separate analysis and is beyond the scope of this work.
}
\end{remark}

\section{Details of the multi-level algorithm}
\label{sec:tree_algorithm}
In this section, we provide implementation details of our complex-coordinate FMM,  which follows closely  the adaptive  versions of the classical FMM~\cite{doi:10.1137/0909044,CHENG1999468}. For clarity of exposition, we illustrate the details of the 2-D case, which naturally generalize to the 3-D case.

\subsection{Hierarchical tree structure}
Although the point locations lie in $\C^2$, which is a four-dimensional space, the intrinsic geometry is only two-dimensional, as the imaginary components are continuous and monotonic functions of the real components. Therefore, instead of constructing a hierarchical tree in $\C^2$, we need only construct an adaptive tree of boxes on the real parts of the points, then complexify the box centers. We will say that a point belongs to a box if its real part lies in the box.

We assume, without loss of generality, that the real parts of all point locations lie within a square box in $\R^2$. The $\ell = 0$ level of the tree will consist of that box. We will recursively split boxes until there are no more than~$N_s$ points in each leaf box. To recursively construct level $\ell + 1$ from level $\ell$, we subdivide each box at level $\ell$ that contains more than $N_s$ points into four equally sized child boxes in $\R^2$. This procedure continues until a finest level $L$ is reached at which every box contains at most $N_s$ points. Within this tree, every box that contains more than $N_s$ points is called a parent box (non-leaf box), while it is called a leaf box if it contains $N_s$ points or fewer.  

All the boxes at level~$\ell$ will have the same width, which we denote~$w_\ell$, and every box~$b$ has a real center, which we denote~$\bx_b\in\R^2$. During the formation of the tree, we impose a \emph{level-restriction} condition: for any pair~$b_1$ and~$b_2$ of childless boxes at levels $\ell_1$ and $\ell_2$, respectively, if $|\ell_1-\ell_2|>1$, then $\|\bx_{b_1}-\bx_{b_2}\|_\infty>\frac{q}{2}(w_{\ell_1}+w_{\ell_2})$.
Once the tree is constructed, we complexify the box centers  in order to define  multipole and local expansions associated with those boxes.  Ideally, for a box $b$ with a real center $\bx_b$,   we would determine its complex center using the same complexification scheme $\psi$ as for the source and target locations:
\begin{equation}
\label{eqn:complexify_center_analytic}
    \tilde \bx_b = \bx_b+i\psi(\bx_b).
\end{equation}
This requires knowing the function $\psi$. Let $\{\tilde\bx_j\}_{j=1}^{n_b}\subset\C^2$ be the complex point locations assigned to box~$b$, and let $\bx_j=\Re\tilde\bx_j$. We use the average of their imaginary parts to define the complexified center:
\begin{equation}
\label{eqn:complexify_center_average}
\tilde \bx_b
=\bx_b+\frac{i}{n_b}\sum_{j=1}^{n_b}\Im\tilde\bx_j,
\end{equation}
where $n_b$ is the number of points in box~$b$. If $\tilde\bx_j=\bx_j+i\psi(\bx_j)$ and $r_b=\max_{1\le j\le n_b}\|\bx_j-\bx_b\|$, then the componentwise Lipschitz condition gives
$
 \left\|\frac{1}{n_b}\sum_{j=1}^{n_b}\Im\tilde\bx_j
 -\psi(\bx_b)\right\|
 \le \frac{L_\psi}{n_b}\sum_{j=1}^{n_b}\|\bx_j-\bx_b\|
 \le L_\psi r_b.
$
Thus,~\eqref{eqn:complexify_center_average} approximates the exact center in~\eqref{eqn:complexify_center_analytic} using only the point locations. The convergence analysis applies directly to the exact center, while the averaged center can be the practical choice used when $\psi$ is unavailable.

\subsection{List generation}
\label{sec:list_generation}
For any nonnegative integer $j$, we say that boxes $a$ and $b$ are \emph{well separated} by exactly $j$ boxes if their real centers $\bx_a,\bx_b\in\R^2$ satisfy
$\|\bx_a-\bx_b\|_\infty = (j+1)w_\ell.$
We define the \emph{$q$-colleagues} of the box $b$ to be the union of the boxes that are well separated from $b$ by at most $q$ boxes. For any box $b$, we use this notion to define four lists of associated boxes.

\emph{List 1} of a box $b$ is empty if the box $b$ is a parent box. If the box $b$ is childless, then List 1 consists of $b$ itself, its childless $q$-colleagues, all childless children $c$ of its $q$-colleagues that satisfy $\|\bx_b-\bx_c\|_\infty\le \frac{q}{2} (w_\ell + w_{\ell+1})$, and all childless $q$-colleagues $a$ of its parent that satisfy $\|\bx_b-\bx_a\|_\infty \le \frac{q}{2} (w_\ell+w_{\ell-1})$.

\emph{List 2} of a box $b$ is formed by all the children of the $q$-colleagues of the box $b$'s parent that are well separated by at least $q+1$ boxes.

\emph{List 3} of a box $b$ is empty if $b$ is a parent box. If the box $b$ is childless, then List 3 consists of all  descendants of $b$'s $q$-colleagues $c$ such that $\|\bx_b-\bx_c\|_\infty > \frac{q}{2} (w_\ell + w_{\ell+1})$.

\emph{List 4} of a box $b$ consists of all boxes $c$ such that $b$ is in List 3 of $c$.

\begin{figure}[ht]
    \centering
    \includegraphics[width=0.75\textwidth]{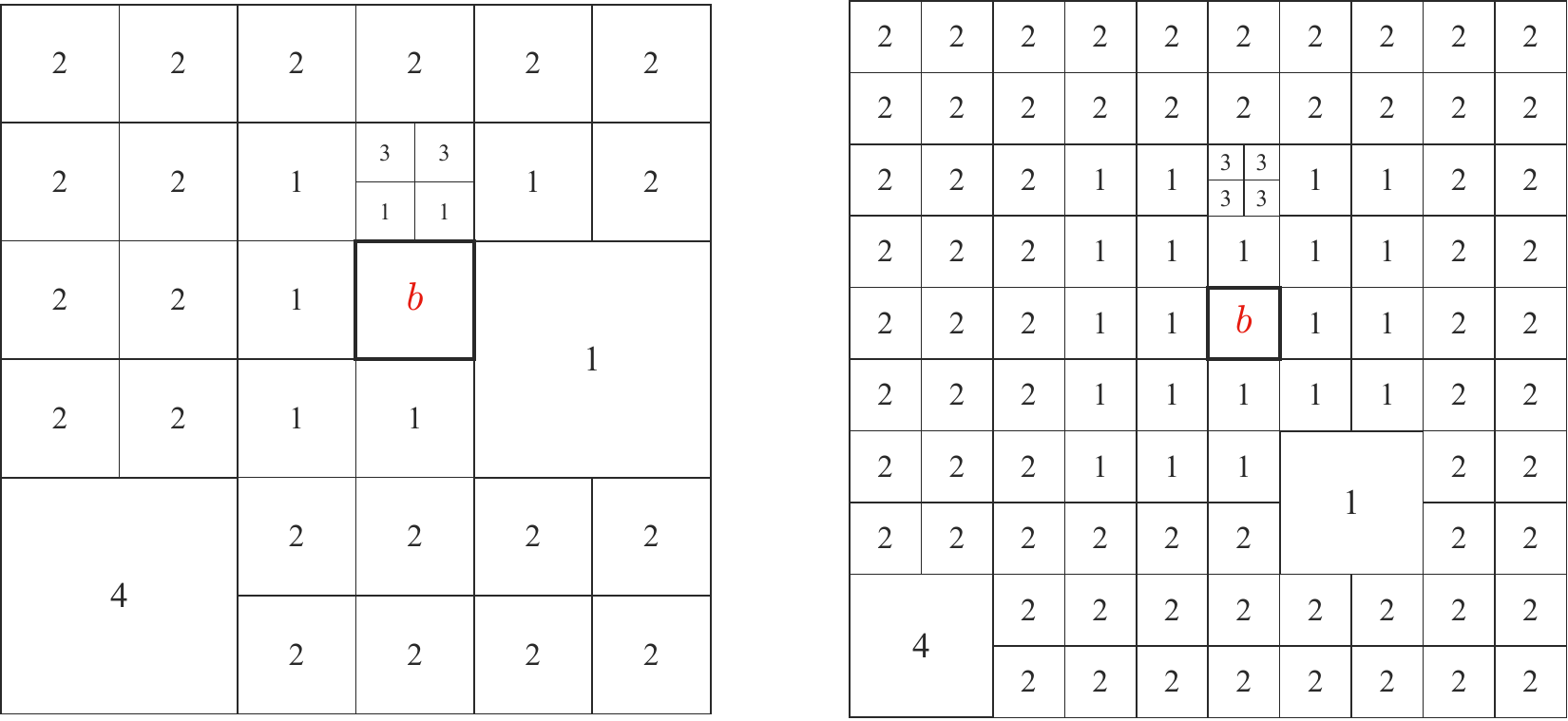}
\caption{ Real parts of a box $b$ and its Lists 1--4 for $q=1$ (left) and $q=2$ (right).
} 
\label{fig:isep}
\end{figure}

\subsection{Determining the expansion degrees}
\label{sec:nterms}
For a prescribed target tolerance $\epsilon>0$, we estimate an appropriate maximum expansion degree $P$ using the multipole error bounds in~\eqref{eqn:l2d_mp_err}, \eqref{eqn:h2d_err_mp}, \eqref{eqn:lap3d_err_mp}, and~\eqref{eqn:h3d_mp_err}.
{
Let $w>0$ be the box width at a given level. The real parts of the points in a box are contained in a ball centered at the box center with radius $r=\frac{\sqrt{2}}{2}w$ in 2-D and $r=\frac{\sqrt{3}}{2}w$ in 3-D. For a well-separated interaction, the smallest possible real distance from the expansion center of the source box to a target location is $R=\left(q+\frac{1}{2}\right)w$. Motivated by the error bounds, we choose $P$  to be the smallest integer satisfying
\begin{equation}
\label{eqn:nterms_criteria_2d}
    \chi(L_\psi)^{P+1}
    |g_{P+1}(R)f_{P+1}(r)|\le \epsilon.
\end{equation}
When $q=1$, the upper bounds on $L_\psi$ imposed by $C_{L_\psi,r,R}<1$ are approximately $0.77$ in 2-D and $0.56$ in 3-D. For $q\ge2$, Lemmas~\ref{lma:bound_for_2d} and~\ref{lma:bound_for_3d} guarantee $C_{L_\psi,r,R}<1$ for every $L_\psi<1$ in both dimensions. }

For the Laplace kernel,  $f_n(z)=z^n$ and $g_n(z)=z^{-n}$ in 2-D,  and $f_n(z)=z^n$ and $g_n(z)=z^{-(n+1)}$ in 3-D. In either dimension, the selected degree is the same at every level. For the Helmholtz kernel,  $f_n(z)=J_n(\kappa z)$ and $g_n(z)=H_n^{(1)}(\kappa z)$ in 2-D, and $f_n(z)=j_n(\kappa z)$ and $g_n(z)=h_n^{(1)}(\kappa z)$ in 3-D. In this case, the required number of terms generally varies across levels through $\kappa w$.

\subsection{Full algorithm}
\label{sec:algo_step}
For example, the complex 2-D Laplace FMM algorithm proceeds through the following steps:
\begin{enumerate}
    \item[(1)] Loop from level $\ell=1$ to level $\ell=L$. For each childless box $b$, form its multipole expansion centered at $\tilde \bx_b$ using~\eqref{eqn:lap2d_mp_coef}.
    \item[(2)] Loop from level $\ell=L-1$ down to level $\ell=1$. For each parent box $b$, aggregate its multipole expansion centered at $\tilde \bx_b$ by combining the multipole expansions of its child boxes $c$, each centered at $\tilde \bx_c$, via~\eqref{eqn:M2M_lap2d}.
    \item[(3)] For every box $b$, loop over its List 1.  For each box $a$ in  List 1, evaluate the interactions directly between the target locations in $b$ and the source locations in $a$.
    \item[(4)] For every box $b$, loop over its List 2. For each box $a$ in List 2, convert $a$'s multipole expansion centered at $\tilde \bx_a$ to $b$'s local expansion centered at $\tilde \bx_b$ using~\eqref{eqn:M2L_lap2d}.
    \item[(5)] For every box, evaluate the multipole expansions from the boxes $b$ in its List 3, each centered at $\tilde \bx_b$, using~\eqref{eqn:mul_expansion_lap2d}.
    \item[(6)] For every box, evaluate contributions to its local expansion from the source locations in the boxes of its List 4 using~\eqref{eqn:lap2d_loc_coef}.
    \item[(7)] Loop from level $\ell=1$ to level $\ell=L-1$. For each parent box $b$, shift its local expansion centered at $\tilde \bx_b$ to the local expansions of its child boxes $c$, each centered at $\tilde \bx_c$, using~\eqref{eqn:L2L_lap2d}.
    \item[(8)] For each childless box, evaluate its local expansion at all target locations using~\eqref{eqn:loc_expansion_lap2d}.
\end{enumerate}
The algorithms for the other three complex-coordinate FMMs discussed in this paper follow the same structure, with their corresponding expansions and translation operators. Their computational costs are the same as those of the classical FMM: $\mathcal{O}(P^2N)$ in two dimensions and $\mathcal{O}(P^3N)$ in three dimensions. Detailed complexity analyses can be found, for example, in \cite{10.7551/mitpress/5750.001.0001,CHENG1999468,GREENGARD1987325}. {The memory cost of the method is of the same order as that of the real-coordinate FMM, since in both cases the expansion coefficients are complex-valued, regardless of whether complexification is used. The only additional memory requirement is the storage of complex point locations.}

\subsection{Catastrophic cancellation}
\label{sec:catastrophic}
In the standard real-coordinate FMM, $q=1$ is typically used by default. For the complex-coordinate FMMs considered here, the choice of $q$ depends on the geometry and may need to be larger for numerical stability. In the 3-D complex-coordinate Helmholtz FMM, we observe that when the problem size exceeds approximately 25 wavelengths, choosing $q=1$ leads to significant catastrophic cancellation errors in the M2L operators. Increasing to $q=2$ ensures stability up to approximately $50$ wavelengths. A similar phenomenon occurs in the 2-D complex-coordinate Helmholtz FMM with $q=1$ once the problem size exceeds roughly 150 wavelengths.

More generally, efficient Helmholtz FMM implementations typically switch to high-frequency variants based on far-field signatures once the root box at level $0$ exceeds roughly 16 wavelengths in size~\cite{crutchfield2006remarks,wideband3d}. Consequently, obtaining efficient high-frequency complex-coordinate FMMs will likely require analogous techniques, which may also alleviate the catastrophic cancellation issue.  {A practical way to prevent catastrophic cancellation is to first evaluate the potentials using the FMM at a relatively small sample of target points and compare the results with direct evaluation. If the desired accuracy is not achieved, one can increase $q$ and repeat the test.}

\section{Numerical Results}
\label{sec:numerical_results}

We  define the following auxiliary functions, which are used for complex deformation:
\begin{align}
    \xi(t) &= \frac{1}{2} \left(t\erfc(t) -\frac{e^{-t^2}}{\sqrt{\pi}}\right),
    \label{eqn:xi_deformation}\\
    \psi_{a,b,t_0}(t)&=a\left(\xi(b(t+t_0))-\xi(-b(t-t_0))\right).
    \label{eqn:psi_deformation}
\end{align}

Our complex-coordinate FMMs are implemented in Fortran 77 with the optimization flag \texttt{-O3} and OpenMP parallelism. The direct methods for comparison are also implemented in Fortran 77 with the flag \texttt{-O3}  and OpenMP.  The numerical experiments were performed on a 16-core Linux machine  with a 12th Gen Intel\textsuperscript{\textregistered} Core\textsuperscript{TM} i9-12900 CPU running at 2.40 GHz and 128 GB of memory.

\subsection{Numerical Results for 2-D}

\subsubsection{Time comparison}
\label{sec:time_comparison}

\begin{figure}[ht]
    \centering
 \begin{tikzpicture}[scale=1]
     \node[inner sep=0] at (0,0) {\includegraphics[width=0.5\textwidth, trim=0 200pt 0 200pt, clip]{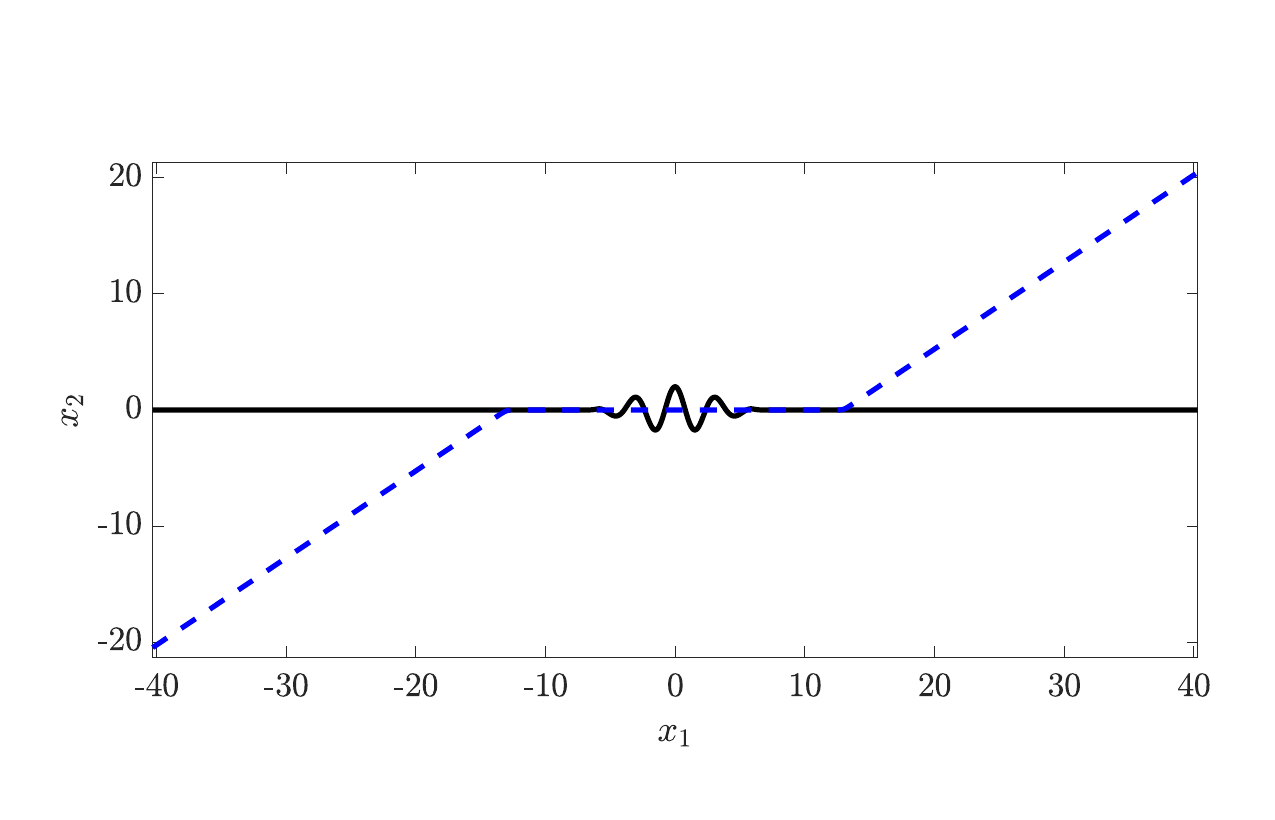}};
      \node[inner sep=0] at (8,0) {\includegraphics[width=0.5\textwidth]{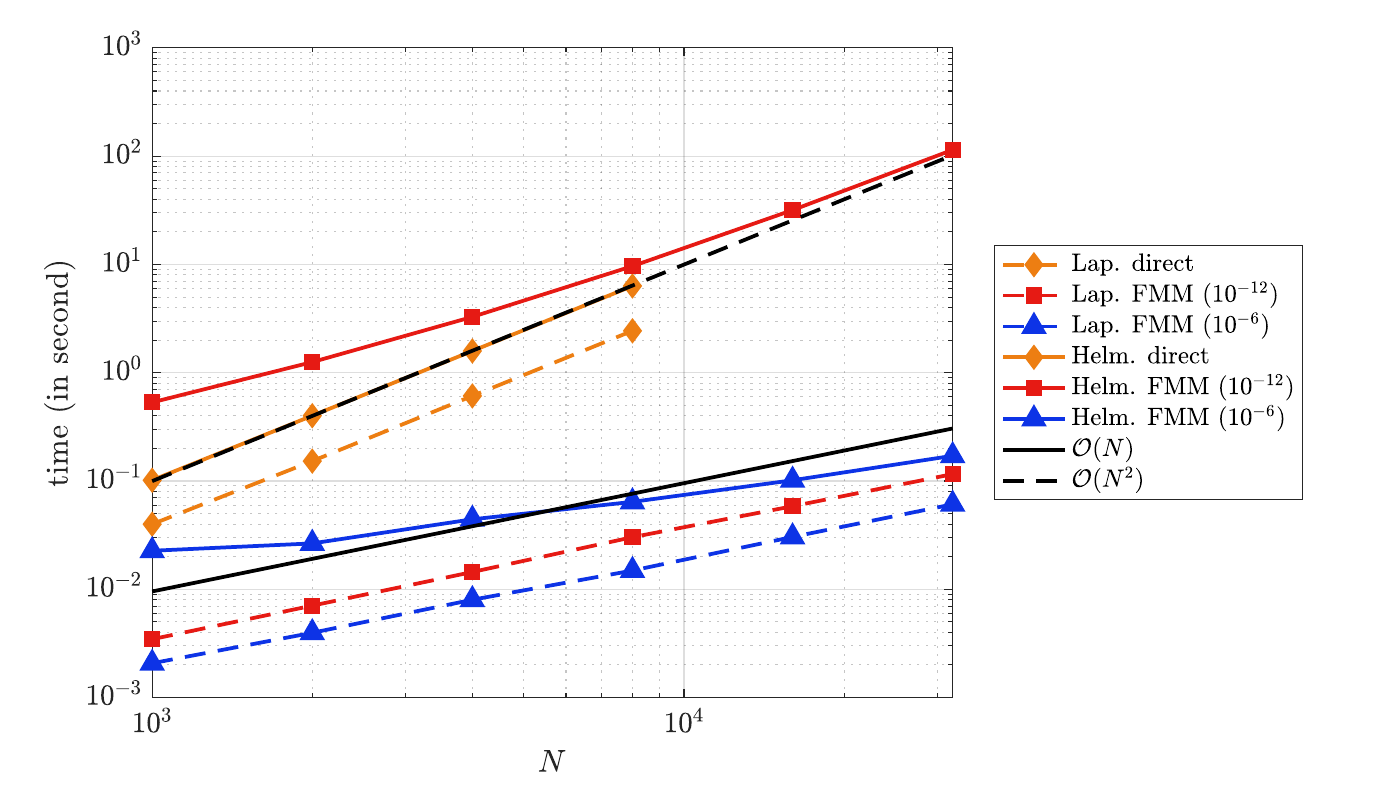}};
\end{tikzpicture}
\caption{\textbf{Left:} The interface used in~\cref{sec:time_comparison}, showing the real interface $\Gamma$ (black) and $\Im \tilde{x}_1$ as a function of $\Re x_1$ (blue dashed). \textbf{Right:} Timings for direct summation and the 2-D complex-coordinate Laplace and Helmholtz FMMs at relative tolerances $10^{-12}$ and $10^{-6}$ for $N$ sources and targets on $\tilde{\Gamma}$ defined by~\eqref{eqn:wobble_2D_cplx}.}
\label{fig:2d_speed_test}
\end{figure}

We define
$
    w(t)=2-\erfc(2(t+6))-\erfc(2(6-t)),
$
and parameterize the real interface $\Gamma$ by
\begin{equation}
\label{eqn:wobble_2D}
    \bgamma(t)=
    \begin{bmatrix}
        t\\
        e^{-t^2/16}\cos(2t)w(t)
    \end{bmatrix},
    \qquad t\in[-40,40].
\end{equation}
The complexified interface $\widetilde\Gamma$ is parameterized by
\begin{equation}
\label{eqn:wobble_2D_cplx}
    \widetilde\bgamma(t)=
    \begin{bmatrix}
        t+i\psi_{1/4,3,13}(t)\\
        e^{-t^2/16}\cos(2t)w(t)
    \end{bmatrix},
    \qquad t\in [-40,40].
\end{equation}
The discretizations is done by the panel-based method in chunkIE~\cite{chunkie}.

After discretizing $\tilde{\Gamma}$, we randomly sample $N$ source (and target) locations from the resulting quadrature nodes and $N$ densities $\{\sigma_j\}_{j=1}^N$ from a standard complex Gaussian distribution. In~\cref{fig:2d_speed_test}, we also report the time required to evaluate the single layer potentials ($\kappa=2\pi$ for the Helmholtz kernel) using the direct methods and the complex-coordinate FMMs with relative errors $10^{-12}$ and $10^{-6}$, where the relative error is defined as
\begin{align}
\label{eqn:relerr}
    \frac{\|\vec u-\vec u_0\|}{\|\vec u_0\|+\|\vec \sigma\|},
\end{align}
where $\vec u$ is the potential computed by the FMMs and $\vec u_0$ is the potential computed by the direct methods. The linear-time scaling of the complex-coordinate FMMs is also demonstrated in~\cref{fig:2d_speed_test}.

\begin{figure}[ht]
    \centering
    \begin{tikzpicture}[scale=1]
        \node[inner sep=0] at (0,0)
        {\includegraphics[width=0.5\textwidth]{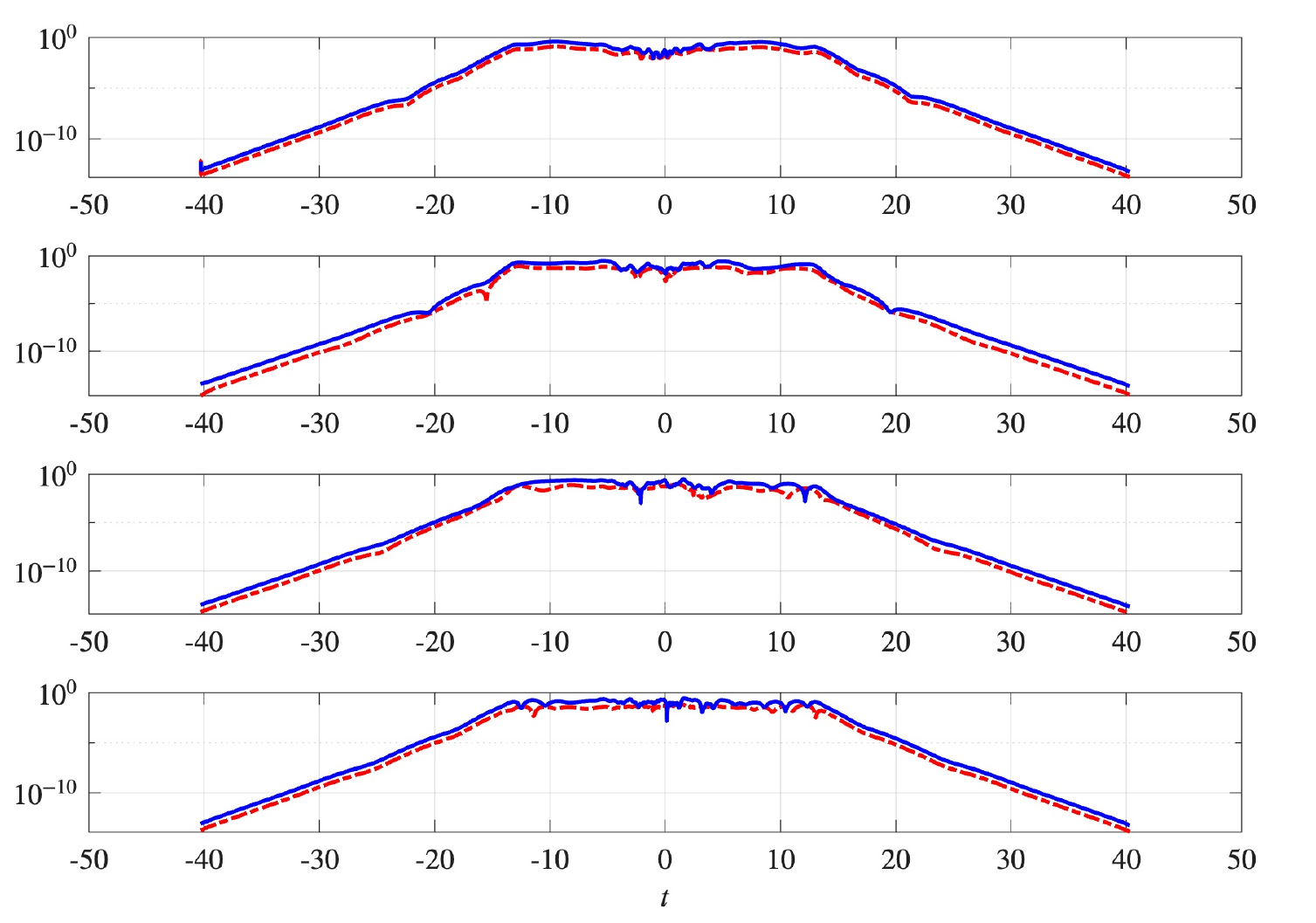}};
        \node[inner sep=0] at (7.8,0)
        {\includegraphics[width=0.55\textwidth]{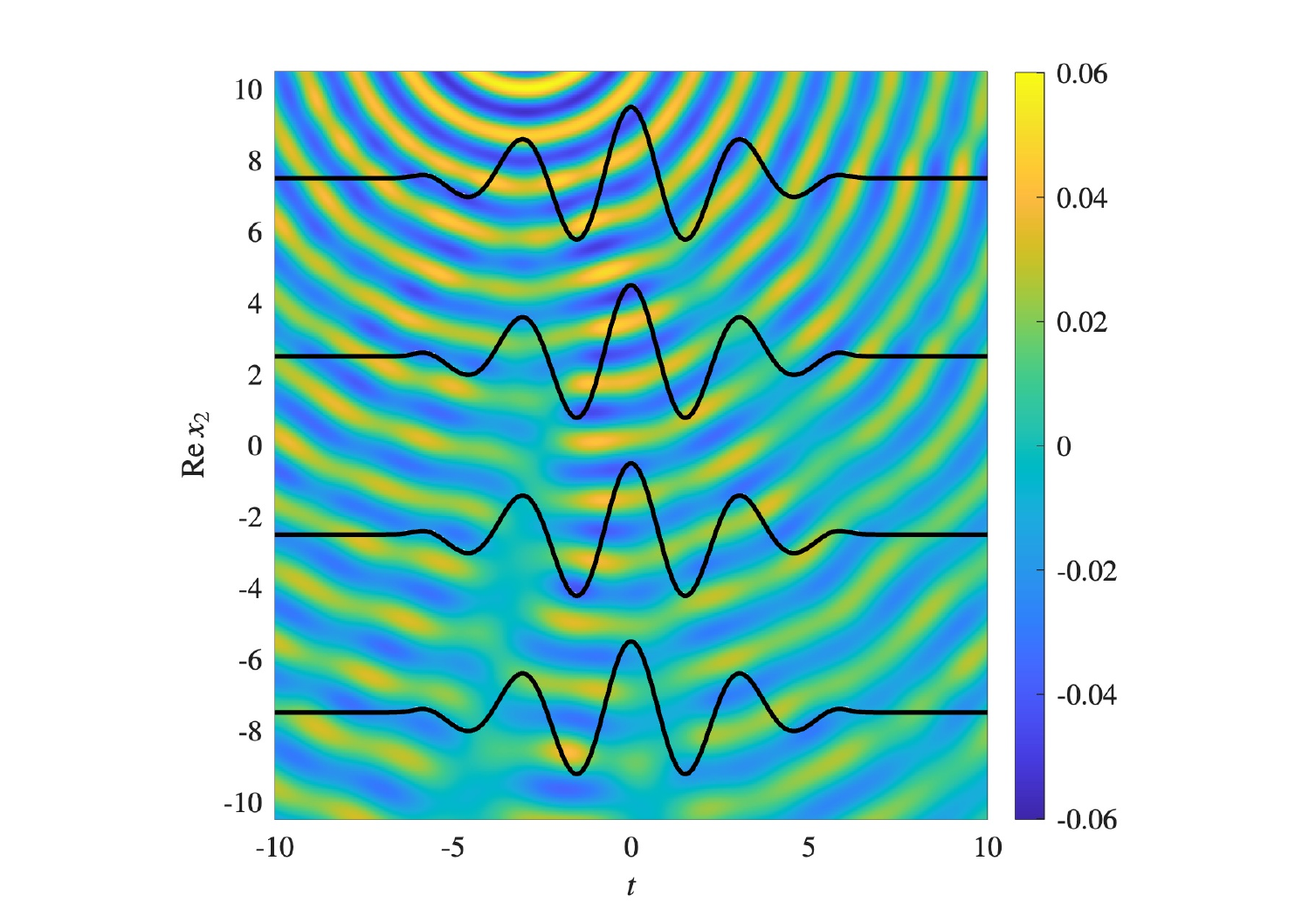}};
    \end{tikzpicture}
    \caption{\textbf{Left:} Magnitudes of the BIE densities $\tau_j$ (red) and $\sigma_j$ (blue) for the 2-D transmission problem, ordered from the bottom interface to the top. \textbf{Right:} Real part of the computed total field, with the black curves showing the interfaces.}
    \label{fig:helm2d_layered_solutions}
\end{figure}

\subsubsection{Transmission problem}
\label{sec:layered_media_2d}
{
Let $\Omega_1,\ldots,\Omega_{N_L}$ denote $N_L$ homogeneous media, ordered from top to bottom, and let $\Gamma_j$ be the interface that separates $\Omega_j$ from $\Omega_{j+1}$ for $j=1,\ldots,N_L-1$.  The unit normal $\nu_j$ on $\Gamma_j$ is oriented from $\Omega_j$ into $\Omega_{j+1}$.  Given an incident field $u^{inc}$, we seek the scattered field $u_1$ in $\Omega_1$ and the transmitted fields $u_2,\ldots,u_{N_L}$ satisfying the PDE:
\begin{equation}
\label{eqn:layered_transmission_PDE}
\begin{cases}
    (\Delta+\kappa_j^2)u_j(\bx)=0,
        & \bx\in\Omega_j,\quad j=1,\ldots,N_L,\\
    u_j(\bx)-u_{j+1}(\bx)
        =-\delta_{j1}u^{inc}(\bx),
        & \bx\in\Gamma_j,\quad j=1,\ldots,N_L-1,\\
    \partial_{\nu_j}u_j(\bx)-\partial_{\nu_j}u_{j+1}(\bx)
        =-\delta_{j1}\partial_{\nu_j}u^{inc}(\bx),
        & \bx\in\Gamma_j,\quad j=1,\ldots,N_L-1,
\end{cases}
\end{equation}
where $\delta_{j1}$ is the Kronecker delta. The total field is $u^{inc}+u_1$ in $\Omega_1$ and $u_j$ in $\Omega_j$ for $j=2,\ldots,N_L$.

Let $\tilde\Gamma_j$ denote the complex deformation of $\Gamma_j$. We represent the fields by
\begin{equation}
\label{eqn:layered_transmission_representation}
    u_j(\bx)=\sum_{m=1}^{N_L-1}
    \left(D_{\tilde\Gamma_m,\kappa_j}[\tau_m](\bx)
    -S_{\tilde\Gamma_m,\kappa_j}[\sigma_m](\bx)\right),
    \qquad \bx\in\Omega_j,
\end{equation}
where $\tau_m$ and $\sigma_m$ are defined on $\tilde\Gamma_m$. For target points on $\tilde\Gamma_j$, define the matrix of layer operators
\begin{equation*}
    \mathcal T_{\ell}^{j,m}
    =\left.
    \begin{bmatrix}
        D_{\tilde\Gamma_m,\kappa_\ell}
            &-S_{\tilde\Gamma_m,\kappa_\ell}\\
        \partial_{\nu_j}D_{\tilde\Gamma_m,\kappa_\ell}
            &-\partial_{\nu_j}S_{\tilde\Gamma_m,\kappa_\ell}
    \end{bmatrix}
    \right|_{\tilde\Gamma_j}.
\end{equation*}
The jump relations give the second-kind BIE system
\begin{equation}
\label{eqn:layered_transmission_BIE}
    -\begin{bmatrix}\tau_j\\ \sigma_j\end{bmatrix}
    +\sum_{m=1}^{N_L-1}
    \left(\mathcal T_j^{j,m}-\mathcal T_{j+1}^{j,m}\right)
    \begin{bmatrix}\tau_m\\ \sigma_m\end{bmatrix}
    =-\delta_{j1}
    \begin{bmatrix}
        u^{inc}|_{\tilde\Gamma_j}\\
        \partial_{\nu_j}u^{inc}|_{\tilde\Gamma_j}
    \end{bmatrix},
    \qquad j=1,\ldots,N_L-1.
\end{equation}

In our example, we choose $N_L=5$, wavenumbers $\kappa_1=1.5\pi$, $\kappa_2=1.375\pi$, $\kappa_3=1.25\pi$, $\kappa_4=1.125\pi$, and $\kappa_5=\pi$. The four interfaces are vertical translates of the interface in~\eqref{eqn:wobble_2D}, with $\Gamma_j=\Gamma+(0,h_j)$, $h_j=5\left(\frac{N_L}{2}-j\right)$, and $j=1,\ldots,N_L-1$. Their complex deformations are $\tilde\Gamma_j=\tilde\Gamma+(0,h_j)$, where $\tilde\Gamma$ is defined in~\eqref{eqn:wobble_2D_cplx}. We choose the incident field to be a point source located at $\bx_0=(-3,12.5)$.

Each complexified interface is discretized with $11{,}680$ quadrature nodes, giving $46{,}720$ point locations and $93{,}440$ unknown densities. We solve~\eqref{eqn:layered_transmission_BIE} using GMRES accelerated by the 2-D complex-coordinate Helmholtz FMM with $q=3$ and $\epsilon=10^{-14}$. We set the GMRES tolerance to $10^{-12}$. GMRES converges in $53$ iterations and takes $76.97$ seconds. All density functions decay below $10^{-12}$ when the complexified interfaces are truncated. The density magnitudes and the computed total field are shown in~\cref{fig:helm2d_layered_solutions}.

}

\subsection{Numerical Results in 3-D}

\subsubsection{Time comparison}

\begin{figure}[ht]
    \centering
 \begin{tikzpicture}[scale=1]
     \node[inner sep=0] at (0,0) {\includegraphics[width=0.4\textwidth]{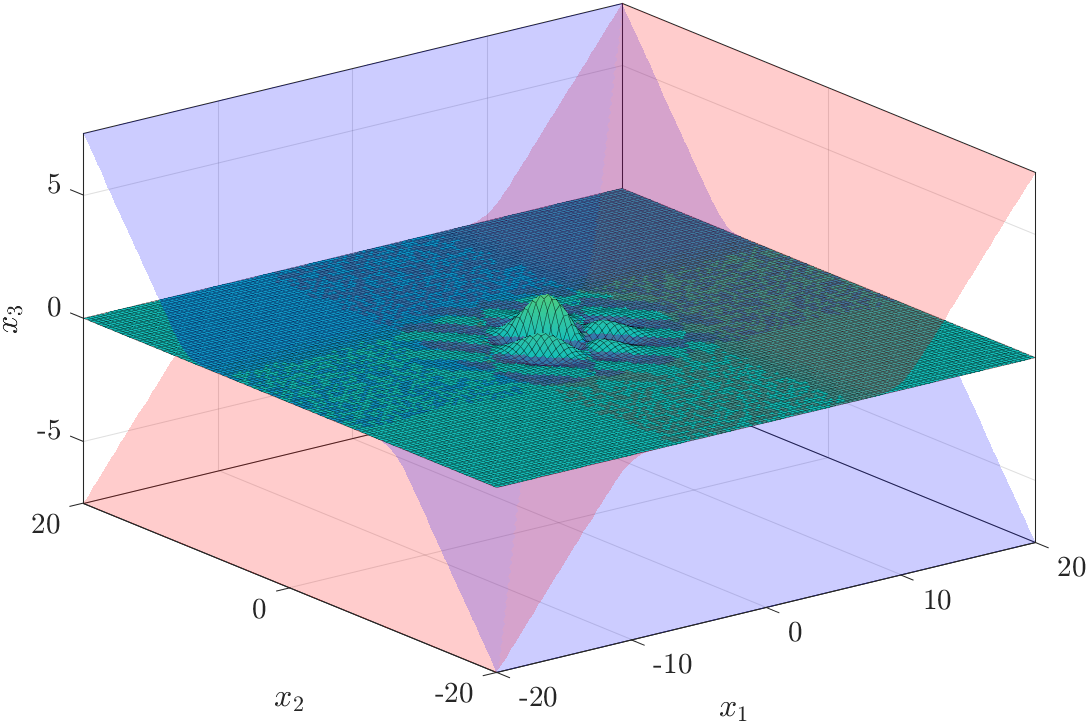}};
      \node[inner sep=0] at (7.5,0) {\includegraphics[width=0.5\textwidth]{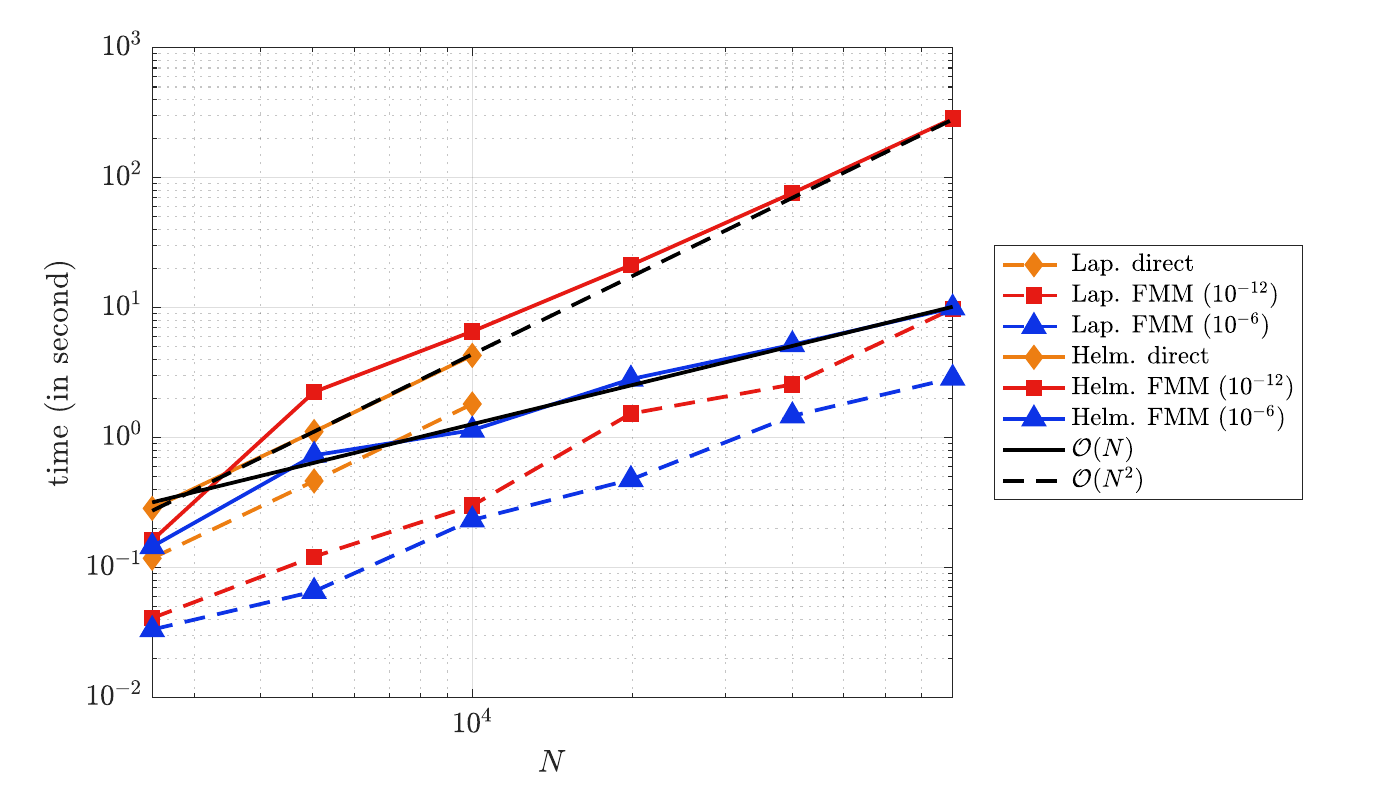}};
\end{tikzpicture}
\caption{\textbf{Left:} The real interface $\Gamma$ and its complex deformation $\tilde{\Gamma}$; the transparent red and blue sheets show $\Im x_1$ and $\Im x_2$, respectively. \textbf{Right:} Timings for direct summation and the 3-D complex-coordinate Laplace and Helmholtz FMMs at relative tolerances $10^{-12}$ and $10^{-6}$ for $N$ sources and targets on $\tilde{\Gamma}$ defined by~\eqref{eqn:complex_wobble}.}
\label{fig:3d_speed_test}
\end{figure}

Next, we consider an interface $\Gamma$ in three dimensions parameterized by
\begin{equation}
\label{eqn:wobble}
    \bgamma(t_1,t_2) = \begin{bmatrix}
        t_1\\
        t_2\\
        e^{-(t_1^2+t_2^2)/8} [\cos(1.9t_1+0.95 t_2)+\sin(t_1+1.55t_2)]
    \end{bmatrix}, \quad (t_1,t_2) \in [-20,20]^2.
\end{equation}
 The complex deformation of the interface is given by
\begin{equation}
\label{eqn:complex_wobble}
    \tilde \bgamma(t_1,t_2) = \begin{bmatrix}
        t_1+i\psi_{a,b,t_0}(t_1)\\
        t_2+i\psi_{a,b,t_0}(t_2)\\
        e^{-(t_1^2+t_2^2)/8} [\cos(1.9t_1+0.95 t_2)+\sin(t_1+1.55t_2)]
    \end{bmatrix}, \quad (t_1,t_2) \in [-20,20]^2,
\end{equation}
where $a=\frac{1}{2}$, $b=\frac{3}{4}$, and $t_0=10$. The resulting complexified interface is illustrated in~\cref{fig:3d_speed_test}. To discretize the complexified surfaces, we use the seventh-order corrected trapezoidal rule in~\cite{hoskins2023quadraturesingularintegraloperators}. In the same figure, we also demonstrate the linear-time scaling of the evaluation of the single layer potentials for the Laplace and Helmholtz kernels using the complex-coordinate FMMs at relative tolerances $10^{-12}$ and $10^{-6}$, compared with the quadratic-time scaling of the direct methods.

\subsubsection{Linearized water wave}
\label{sec:water_wave}
Let $\Omega$ be an infinite fluid domain bounded above by a free surface $\Gamma_t$ and below by a rigid bottom surface $\Gamma_b$. Given a frequency parameter $\alpha>0$, let $\kappa>0$ be the positive solution of $\kappa\tanh\kappa=\alpha$, and write $r_h=\sqrt{x_1^2+x_2^2}$. For a plane-wave incident velocity potential $u^{inc}$, we compute the scattered field $u$ by solving the Laplace boundary value problem
\begin{equation}
\label{eqn:water_wave_3d}
\begin{cases}
    \Delta u(\bx)=0, \quad & \bx \in  \Omega\\ 
    (\partial_{\nu(\bx)} -\alpha) u(\bx) = -(\partial_{\nu(\bx)} -\alpha)u^{inc}(\bx), \quad & \bx \in \Gamma_t\\ 
    \partial_{\nu(\bx)} u(\bx) =-\partial_{\nu(\bx)} u^{inc}(\bx), \quad & \bx \in \Gamma_b \\
    \displaystyle\lim\limits_{r_h\to\infty}\sqrt{r_h}\left(\partial_{r_h}u(\bx)-i\kappa u(\bx)\right)=0.&
\end{cases}
\end{equation}
The normal vector $\nu$ on $\Gamma=\Gamma_t \cup \Gamma_b$ is chosen to point away from~$\Omega$.

\begin{figure}[ht]
    \centering
 \begin{tikzpicture}[scale=0.9]
      \node[inner sep=0] at (0,0) {\includegraphics[width=0.4\textwidth,trim=5pt 88pt 5pt 88pt,clip]{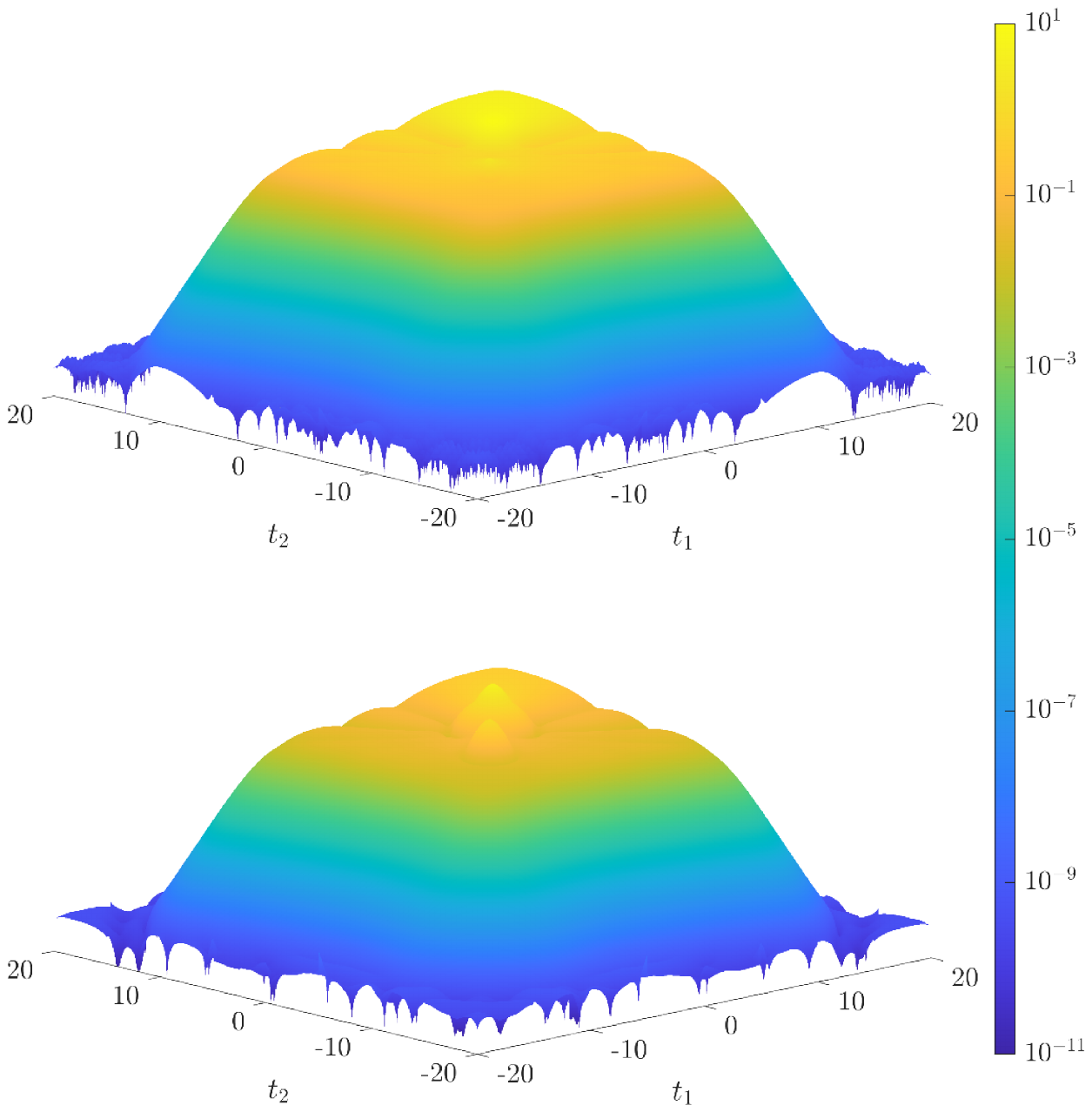}};
      \node[inner sep=0] at (8,0) {\includegraphics[width=0.5\textwidth,trim=14pt 178pt 14pt 178pt,clip]{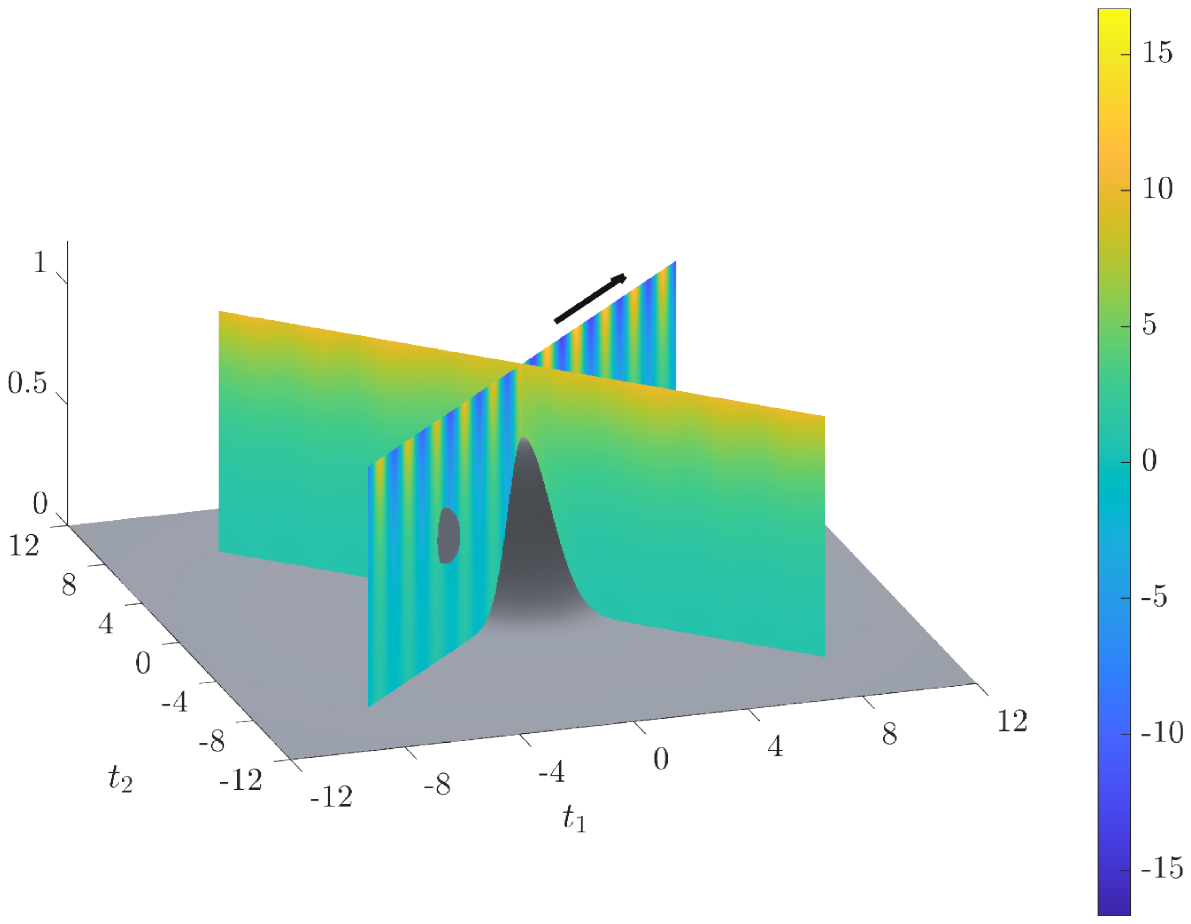}};
\end{tikzpicture}
\caption{\textbf{Left}: Magnitude of the BIE solution on top and bottom surfaces for the 3-D linearized water wave problem. \textbf{Right}: The real part of the computed total field at two cross sections, where the black arrow shows the direction of the incident field.}
    \label{fig:waterwave3d_soln}
\end{figure}

Let $\tilde \Gamma_t$ and $\tilde \Gamma_b$ denote suitable complex deformations of the real interfaces $\Gamma_t$ and $\Gamma_b$, respectively. Since $u$ is harmonic in $\Omega$, Green's identity gives
$
    u(\bx) = S_{\tilde \Gamma}[\partial_\nu u](\bx)-D_{\tilde \Gamma}[u](\bx), \qquad \bx \in \Omega,
$
with $\tilde \Gamma = \tilde\Gamma_t \cup \tilde\Gamma_b$. Taking the limit of this representation on $\tilde\Gamma$ and using the standard jump relations for the double layer potential gives
$
    \frac{1}{2}u(\tilde \bx) = S_{\tilde\Gamma}[\partial_\nu u](\tilde \bx)-D_{\tilde\Gamma}[u](\tilde \bx), \qquad \tilde \bx\in\tilde\Gamma.
$
Substituting the boundary conditions from~\eqref{eqn:water_wave_3d} into this relation leads to the following BIE
\begin{equation}
\label{eqn:waterwaveBIE}
\frac{1}{2}u(\tilde \bx)+D_{\tilde \Gamma} [u](\tilde \bx)-  \alpha  S_{\tilde \Gamma_t}[u|_{\tilde \Gamma_t}](\tilde \bx) =S_{\tilde \Gamma_t} [f|_{\tilde \Gamma_t}](\tilde \bx)+S_{\tilde \Gamma_b} [g|_{\tilde \Gamma_b}](\tilde \bx),\quad \tilde \bx\in \tilde \Gamma,
\end{equation}
where $f(\tilde \bx)=-(\partial_{\nu(\tilde \bx)} -\alpha)u^{inc}(\tilde \bx)$ and $g(\tilde \bx)=-\partial_{\nu(\tilde \bx)} u^{inc}(\tilde \bx)$.
In our numerical example, we parameterize $\Gamma_b$ as a Gaussian bump
\begin{equation*}
    \bgamma(t_1,t_2) = \begin{bmatrix}
        t_1\\
        t_2\\
        0.7 \exp\left(-\frac{t_1^2+t_2^2}{1.5^2}\right)
    \end{bmatrix}, \quad (t_1,t_2) \in [-20,20]^2,
\end{equation*}
and we choose $\Gamma_t$ to be flat, i.e., $\Gamma_t=\{x_3=1\}$. We use the same complexification scheme in~\eqref{eqn:complex_wobble} to deform both interfaces and obtain $\tilde \Gamma_t$ and $\tilde \Gamma_b$. We discretize each complexified surface using $1001^2$ quadrature nodes, resulting in $2{,}004{,}002$ unknowns. We solve a scattering problem with $\alpha=3$ and the incident field $u^{inc}(\bx)=\frac{\cosh(\kappa x_3)}{\cosh\kappa}e^{i\kappa(x_1+x_2)/\sqrt2}$ for $\bx\in\mathbb R^3$. We solve the BIE in~\eqref{eqn:waterwaveBIE} iteratively using our Laplace FMM with $q=2$ and $\epsilon=10^{-6}$ and set the GMRES tolerance to $10^{-8}$. GMRES converges in $240$ iterations, taking $26.50$ hours. \Cref{fig:waterwave3d_soln} illustrates the exponential decay of the BIE solution and the real part of the computed total field.

\subsubsection{Transmission problem}
\label{sec:helm_transmission}
We consider the 3-D analogue of the layered transmission problem in~\cref{sec:layered_media_2d}. Except for the dimension, the PDE and the integral representation have the same form as~\eqref{eqn:layered_transmission_PDE}--\eqref{eqn:layered_transmission_BIE}, with $N_L=2$. We use the interface $\Gamma$ and its complex deformation $\tilde\Gamma$ defined in~\eqref{eqn:wobble} and~\eqref{eqn:complex_wobble}, respectively. The wavenumbers in the upper and lower media are $\kappa_1=1.3\pi$ and $\kappa_2=\pi$. For the incident field $u^{inc}(\bx)=G_{\kappa_1}(\bx,\bx_0)$ with $\bx_0=(2,3,10)$, we discretize $\tilde\Gamma$ with $2501^2$ quadrature nodes, giving $12{,}510{,}002$ unknowns. We solve the BIE using the 3-D Helmholtz FMM with $q=2$ and $\epsilon=10^{-6}$ and set the GMRES tolerance to $10^{-8}$. The solver has a relative error of $\mathcal O(10^{-5})$ in an analytic test. GMRES converges in $26$ iterations, taking $12.78$ hours. \Cref{fig:helm3d_solutions} shows the density magnitudes and the real part of the computed total field.

\begin{figure}[ht]
    \centering
 \begin{tikzpicture}[scale=0.9]
      \node[inner sep=0] at (0,0) {\includegraphics[width=0.4\textwidth,trim=0pt 80pt 0pt 80pt,clip]{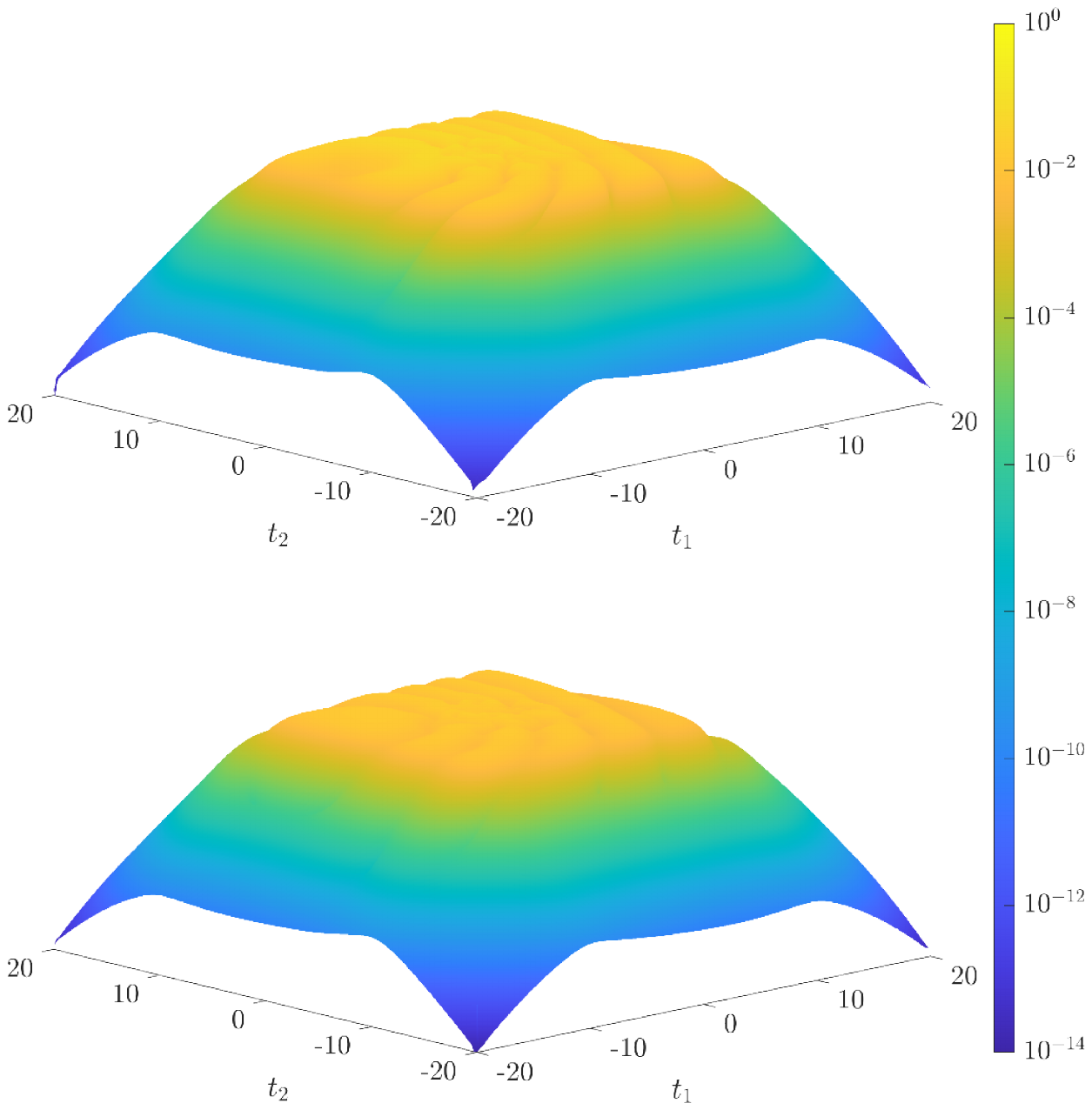}};
      \node[inner sep=0] at (8,0) {\includegraphics[width=0.5\textwidth,trim=10pt 190pt 10pt 190pt,clip]{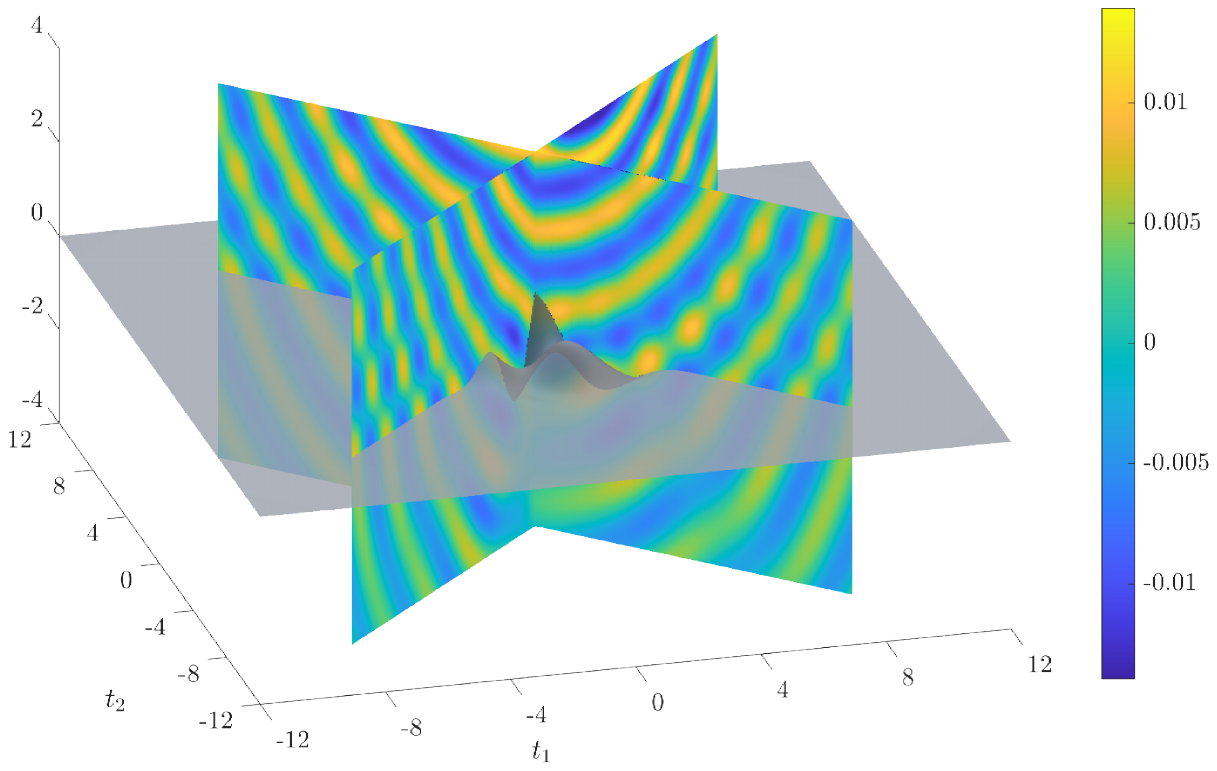}};
\end{tikzpicture}
\caption{\textbf{Left:} Magnitudes of the BIE solution $\tau$ and $\sigma$ for the 3-D transmission problem.  \textbf{Right:} Real part of the computed total field on two cross sections.}
    \label{fig:helm3d_solutions}
\end{figure}
\FloatBarrier

\section{Discussion and Conclusion}
\label{sec:conc}
In this work, we developed a complex-coordinate FMM for both Laplace and Helmholtz kernels in two and three dimensions while maintaining the same linear-time complexity as the classical FMM. The primary novelty of our work lies in building a hierarchical data structure in the ambient dimension of the problem and leveraging analytic continuations of polar and spherical coordinates and addition formulas. We proved truncation error estimates for the multipole and local expansions under Lipschitz assumptions on the imaginary parts expressed as functions of their corresponding real parts. We illustrated the effectiveness of this approach for solving complicated transmission problems and time-harmonic water-wave problems.

 However, there are several open questions that remain to be addressed.  Firstly, for efficient fast multipole methods in the high-frequency regime, an extension of fast diagonalization transform techniques~\cite{ROKHLIN1990414,rokhlin_1993}, which rely on far-field signatures, is required --- this should reduce the cost of applying translation operators from $\cO(P^2)$ to $\cO(P\log P)$ in two dimensions, and from $\cO(P^3)$ to $\cO(P^2\log P)$ in three dimensions. Secondly, the techniques in this work can potentially be extended to other Green's functions that require complex scaling to handle unbounded domains. Finally, this work exploits the analytic properties of complexified kernels, and therefore serves as an important preliminary step toward establishing certain theoretical guarantees for kernel-independent methods.
 
\section*{Acknowledgements}
 The authors would like to thank Charles L. Epstein, Shidong Jiang,  and Peter Nekrasov for many helpful discussions. The Flatiron Institute is a division of the Simons Foundation. The third author was supported in part by a Sloan Research Fellowship. This research was supported in part by grants from the NSF (DMS-2235451) and Simons Foundation (MPS-NITMB-00005320) to the NSF-Simons National Institute for Theory and Mathematics in Biology (NITMB).

\printbibliography

@article{crutchfield2006remarks,
  title={Remarks on the implementation of wideband FMM for the {H}elmholtz equation in two dimensions},
  author={Crutchfield, W and Gimbutas, Z and Greengard, L and Huang, J and Rokhlin, V and Yarvin, N and Zhao, J},
  journal={Contemp. Math.},
  volume={408},
  number={01},
  year={2006}
}

@article{doi:10.1137/S0036142995287847,
author = {Kapur, Sharad and Rokhlin, Vladimir},
title = {High-Order Corrected Trapezoidal Quadrature Rules for Singular Functions},
journal = {SIAM J. Numer. Anal.},
volume = {34},
number = {4},
pages = {1331-1356},
year = {1997}
}

@article{WuMartinsson2021b,
  author    = {Bowei Wu and Per-Gunnar Martinsson},
  title     = {Corrected trapezoidal rules for boundary integral equations in three dimensions},
  journal   = {Numer. Math.},
  year      = {2021},
  volume    = {149},
  pages     = {1025--1071},
}

@article{doi:10.1137/22M1520372,
author = {Wu, Bowei and Martinsson, Per-Gunnar},
title = {A Unified Trapezoidal Quadrature Method for Singular and Hypersingular Boundary Integral Operators on Curved Surfaces},
journal = {SIAM J. Numer. Anal.},
volume = {61},
number = {5},
pages = {2182-2208},
year = {2023},
}

@article{WuMartinsson2021,
  author    = {Bowei Wu and Per-Gunnar Martinsson},
  title     = {Zeta correction: a new approach to constructing corrected trapezoidal quadrature rules for singular integral operators},
  journal   = {Adv. Comput. Math.},
  year      = {2021},
  volume    = {47},
  number    = {45}
}

@article{doi:10.1137/16M1095949,
author = {Minden, Victor and Ho, Kenneth L. and Damle, Anil and Ying, Lexing},
title = {A Recursive Skeletonization Factorization Based on Strong Admissibility},
journal = {Multiscale Model. Simul.},
volume = {15},
number = {2},
pages = {768-796},
year = {2017}
}

@inbook{b9199369fcb841c19f902432847cfa34,
title = "A short course on fast multipole methods",
author = "R Beatson and Leslie Greengard",
year = "1997",
language = "English (US)",
series = "Numerical Mathematics and Scientific Computation",
publisher = "Oxford University Press",
pages = "1--37",
booktitle = "Wavelets, multilevel methods, and elliptic PDEs",
}

@article{https://doi.org/10.1002/num.23148,
author = {Meng, Wenhui},
title = {Error bound of the multilevel fast multipole method for {3-D} scattering problems},
journal = {Numer. Methods Partial Differential Equations},
volume = {40},
number = {6},
year = {2024},
pages = {e23148}
}

@article{doi:10.1137/0909044,
author = {Carrier, J. and Greengard, L. and Rokhlin, V.},
title = {A Fast Adaptive Multipole Algorithm for Particle Simulations},
journal = {SIAM J. Sci. Stat. Comput.},
volume = {9},
number = {4},
pages = {669-686},
year = {1988},
}

@ARTICLE{714591,
  author={Greengard, L. and Jingfang Huang and Rokhlin, V. and Wandzura, S.},
  journal={IEEE Comput. Sci. Eng.}, 
  title={Accelerating fast multipole methods for the {H}elmholtz equation at low frequencies}, 
  year={1998},
  volume={5},
  number={3},
  pages={32-38},}

@article{YING2004591,
title = {A kernel-independent adaptive fast multipole algorithm in two and three dimensions},
journal = {J. Comput. Phys.},
volume = {196},
number = {2},
pages = {591-626},
year = {2004},
author = {Lexing Ying and George Biros and Denis Zorin},
}

@article{GREENGARD1987325,
title = {A fast algorithm for particle simulations},
journal = {J. Comput. Phys.},
volume = {73},
number = {2},
pages = {325-348},
year = {1987},
author = {L. Greengard and V. Rokhlin},
}

@book{10.7551/mitpress/5750.001.0001,
    author = {Greengard, Leslie F.},
    title = {The Rapid Evaluation of Potential Fields in Particle Systems},
    publisher = {The MIT Press},
    year = {1988},
    month = {04}

}

@article{Greengard_Rokhlin_1997, title={A new version of the Fast Multipole Method for the {L}aplace equation in three dimensions}, 
volume={6}, 
journal={Acta Numerica}, 
author={Greengard, Leslie and Rokhlin, Vladimir}, 
year={1997}, 
pages={229–269}}

@article{CHENG1999468,
title = {A Fast Adaptive Multipole Algorithm in Three Dimensions},
journal = {J. Comput. Phys.},
volume = {155},
number = {2},
pages = {468-498},
year = {1999},
author = {H. Cheng and L. Greengard and V. Rokhlin}
}

@article{GIMBUTAS20095621,
title = {A fast and stable method for rotating spherical harmonic expansions},
author = {Z. Gimbutas and L. Greengard},
journal = {J. Comput. Phys.},
volume = {228},
number = {16},
pages = {5621-5627},
year = {2009},
}

@article{white96,
    author = {White, Christopher A. and Head‐Gordon, Martin},
    title = "{Rotating around the quartic angular momentum barrier in fast multipole method calculations}",
    journal = {J. Chem. Phys.},
    volume = {105},
    number = {12},
    pages = {5061-5067},
    year = {1996},
}

@article{ROKHLIN1990414,
title = {Rapid solution of integral equations of scattering theory in two dimensions},
journal = {J. Comput. Phys.},
volume = {86},
number = {2},
pages = {414-439},
year = {1990},
author = {Vladimir Rokhlin},
}

@article{doi:10.1137/24M1650703,
author = {Bao, Gang and Lu, Wangtao and Yin, Tao and Zhang, Lu},
title = {A Highly Accurate PML-BIE Solver for the Electromagnetic Scattering Problem in a Multilayered Medium},
journal = {SIAM J. Sci. Comput.},
volume = {46},
number = {6},
pages = {A3849-A3872},
year = {2024},

}

@article{complexification,
      title={Complex Scaling for the Helmholtz Equation with Dirichlet Boundary Conditions in a Perturbed Half-Space}, 
      author={Charles L. Epstein and Leslie Greengard and Jeremy Hoskins and Shidong Jiang and Manas Rachh},
      volume = {64},
      number = {2},
      pages = {510-536},
      year = {2026},
      journal={SIAM J. Numer. Anal.},
}

@article{goodwill2024,
      title={A numerical method for scattering problems with unbounded interfaces}, 
      author={Tristan Goodwill and Charles L. Epstein},
      year={2024},
      journal={J. Sci. Comput.},
      volume = {107},
      number = {70},
      year = {2026},
}

@article{epstein2025complexscalingopenwaveguides,
      title={Complex scaling for open waveguides}, 
      author={Charles L. Epstein and Tristan Goodwill and Jeremy Hoskins and Solomon Quinn and Manas Rachh},
      year={2025},
        journal={arXiv preprint arXiv: 2506.10263}
}

@article{hoskins2023quadraturesingularintegraloperators,
title = {On quadrature for singular integral operators with complex symmetric quadratic forms},
journal = {Appl. Comput. Harmon. Anal.},
volume = {74},
pages = {101721},
year = {2025},
author = {Jeremy Hoskins and Manas Rachh and Bowei Wu}
}

@article{doi:10.1137/23M1607866,
author = {Dhia, Anne-Sophie Bonnet-Ben and Faria, Luiz M. and P\'{e}rez-Arancibia, Carlos},
title = {A Complex-Scaled Boundary Integral Equation for Time-Harmonic Water Waves},
volume = {84},
number = {4},
pages = {1532-1556},
year = {2024},
journal={SIAM J. Appl. Math.}
}

@article{WANG2026114956,
title = {A high-accuracy PML-BIE solver for three-dimensional acoustic wave scattering in axisymmetric layered media},
journal = {J. Comput. Phys.},
volume = {559},
pages = {114956},
year = {2026},
author = {Hangya Wang and Wangtao Lu}
}

@article{doi:10.1137/22M1532457,
author = {Lu, Wangtao and Xu, Liwei and Yin, Tao and Zhang, Lu},
title = {A Highly Accurate Perfectly-Matched-Layer Boundary Integral Equation Solver for Acoustic Layered-Medium Problems},
journal = {SIAM J. Sci. Comput.},
volume = {45},
number = {4},
pages = {B523-B543},
year = {2023},
}

@article{doi:10.1137/21M1439705,
author = {Yu, Xiuchen and Hu, Guanghui and Lu, Wangtao and Rathsfeld, Andreas},
title = {PML and High-Accuracy Boundary Integral Equation Solver for Wave Scattering by a Locally Defected Periodic Surface},
journal = {SIAM J. Numer. Anal.},
volume = {60},
number = {5},
pages = {2592-2625},
year = {2022}
}

@article{10.1098/rsta.2024.0342,
    author = {Lu, Wangtao and Li, Zhilin},
    title = {A perfectly matched layer–boundary integral equation method for wave scattering in a two-layer medium of a step-like interface},
    journal = {Philos Trans A Math Phys Eng Sci},
    volume = {383},
    number = {2303},
    pages = {20240342},
    year = {2025},
    month = {08}
}

@article{agocs2026integralequationsflexuralscattering,
      title={Integral equations for flexural scattering problems with periodic boundaries}, 
      author={Fruzsina Agocs and Tristan Goodwill and Jeremy G. Hoskins and Peter Nekrasov},
      year={2026},
      journal={arXiv preprint arXiv: 2603.00366}
}

@book{abramowitz+stegun,
  author = {Abramowitz, Milton and Stegun, Irene A.},
  publisher = {Dover},
  title = {Handbook of Mathematical Functions with Formulas, Graphs, and Mathematical Tables},
  year = 1964
}

@misc{NIST:DLMF,
         key = "{\relax DLMF}",
       title = "{\it NIST Digital Library of Mathematical Functions}",
howpublished = "\url{https://dlmf.nist.gov/}, Release 1.2.3 of 2024-12-15",
         url = "https://dlmf.nist.gov/",
        note = "F.~W.~J. Olver, A.~B. {Olde Daalhuis}, D.~W. Lozier, B.~I. Schneider,
                R.~F. Boisvert, C.~W. Clark, B.~R. Miller, B.~V. Saunders,
                H.~S. Cohl, and M.~A. McClain, eds."}

@book{watson1995treatise,
  title={A Treatise on the Theory of Bessel Functions},
  author={Watson, G.N.},
  series={Cambridge Mathematical Library},
  year={1995},
  publisher={Cambridge University Press}
}

@book{whittaker1996course,
  title={A Course of Modern Analysis},
  author={Whittaker, E.T.  and Watson, G.N.},
  year={1996},
  publisher={Cambridge University Press}
}

@article{chirikjian2016harmonic,
  title={Harmonic Analysis for Engineers and Applied Scientists: Updated and Expanded Edition},
  author={Chirikjian, G.S. and Kyatkin, A.B.},
  year={2016},
  journal={Dover Publications}
}

@article{https://doi.org/10.1002/cpa.22240,
author = {Jiang, Shidong and Greengard, Leslie},
title = {A dual-space multilevel kernel-splitting framework for discrete and continuous convolution},
journal = {Comm. Pure Appl. Math.},
volume = {78},
number = {5},
pages = {1086-1143},
year = {2025}
}

@misc{chunkie,
author = {Travis Askham and Manas Rachh and Mike O’Neil and Jeremy Hoskins and Dan Fortunato and Shidong Jiang and 
Fredrik Fryklund and Tristan Goodwill and  Hai Yang Wang and Hai Zhu},
title = {chunkie: A {MATLAB} integral equation toolbox},
url = {https://github.com/fastalgorithms/chunkie},
year = {2024},
}

@article{malhotra2015pvfmm,
  title={{PVFMM}: A parallel kernel independent {FMM} for particle and volume potentials},
  author={Malhotra, Dhairya and Biros, George},
  journal={Commun. Comput. Phys.},
  volume={18},
  number={3},
  pages={808--830},
  year={2015},
  publisher={Cambridge University Press}
}

@article{Saad-gmres-1986,
  title={{GMRES}: A generalized minimal residual algorithm for solving nonsymmetric linear systems},
  author={Saad, Youcef and Schultz, Martin},
  journal={SIAM J. Sci. Stat. Comput.},
  volume={7},
  number={3},
  pages={856--869},
  year={1986},
}

@book{kress_2014,
  title = {{Linear Integral Equations}},
  author = {Rainer Kress},
  publisher = {Springer},
  year = 2014,
  address = {New York, NY}
  }

@article{rokhlin_1993,
  author = {Vladimir Rokhlin},
  title = {{Diagonal Forms of Translation Operators for the {H}elmholtz Equation in Three Dimensions}},
  journal = {Appl. Comput. Harmon. Anal.},
  volume = {1},
  year = {1993},
  pages = {82--93},
  }

@article{wideband3d,
  author = {Hongwei Cheng and William Y. Crutchfield and Zydrunas Gimbutas
             and Leslie Greengard and J. Frank Ethridge and Jingfang Huang
            and Vladimir Rokhlin and Norman Yarvin and Junsheng Zhao},
  title = {A wideband fast multipole method for the {H}elmholtz equation in
           three dimensions},
  journal = {J. Comput. Phys.},
  volume = {216},
  pages = {300--325},
  year = {2006}
}

@book{colton_kress_inverse,
  author="David Colton and Rainer Kress",
  title={Inverse {A}coustic and {E}lectromagnetic {S}cattering {T}heory},
  publisher={Springer},
  address={New York, NY},
  year={2012}
}

@book{colton_kress,
  author="David Colton and Rainer Kress",
  title="Integral {E}quation {M}ethods in {S}cattering {T}heory",
  publisher="John Wiley \& Sons, Inc.",
  year="1983"
}

@article{sushnikova2023fmm,
  title={{FMM-LU}: A fast direct solver for multiscale boundary integral equations in three dimensions},
  author={Sushnikova, Daria and Greengard, Leslie and O'Neil, Michael and Rachh, Manas},
  journal={Multiscale Model. Simul.},
  volume={21},
  number={4},
  pages={1570--1601},
  year={2023},
  publisher={SIAM}
}

@article{lu2018perfectly,
  title={Perfectly matched layer boundary integral equation method for wave scattering in a layered medium},
  author={Lu, Wangtao and Lu, Ya Yan and Qian, Jianliang},
  journal={SIAM J. Appl. Math.},
  volume={78},
  number={1},
  pages={246--265},
  year={2018},
  publisher={SIAM}
}

@article{tan2026uniformlyhighaccuracypmlbiemethod,
      title={A Uniformly High-Accuracy PML-BIE Method for Scattering by Periodic Arrays of Obstacles: The 2D Case}, 
      author={Yan Tan and Carlos Pérez-Arancibia and Tao Yin},
      year={2026},
      journal={arXiv preprint arXiv:2606.07971}
}

@article{lu2026,
      title={Convergence of the PML-BIE Method for Acoustic Scattering in an Impedance Half-Space}, 
      author={Wangtao Lu},
      year={2026},
      journal={arXiv preprint arXiv:2606.15878}
}

@book{sakurai2021modern,
  author    = {Sakurai, J. J. and Napolitano, Jim},
  title     = {Modern Quantum Mechanics},
  edition   = {3},
  publisher = {Cambridge University Press},
  year      = {2021},
}

@article{pan2026multilayeredproblemsmultipletwolayered,
      title={From multi-layered problems to multiple two-layered problems: a novel frequency-time hybrid multiple-scattering integral equation solver}, 
      author={Shuai Pan and Tao Yin and Lu Zhang},
      year={2026}, 
      journal={arXiv preprint arXiv:2608.30287}
}
\end{document}